\documentclass[11pt,reqno]{amsart}
\usepackage{amssymb,mathrsfs,graphicx,subfigure, enumerate}
\usepackage{amsmath,amsfonts,amssymb,amscd,amsthm,bbm}
\usepackage{extpfeil}
\usepackage{graphicx,colortbl}
\usepackage{float}
\usepackage{epsfig}
\usepackage{caption}
\usepackage{bm}
\graphicspath{{Figure/}}

\usepackage[margin=1in]{geometry}

\makeatletter
\@namedef{subjclassname@2020}{\textup{2020} Mathematics Subject Classification}
\makeatother

\title[Navier-Stokes-Korteweg equations]{Long-time behavior towards viscous-dispersive shock for Navier-Stokes equations of Korteweg type}

\author[Han]{Sungho Han}
\address[Sungho Han]{\newline Department of Mathematical Sciences \newline Korea Advanced Institute of Science and Technology, Daejeon  34141, Republic of Korea}
\email{sungho\_han@kaist.ac.kr}

\author[Kang]{Moon-Jin Kang}
\address[Moon-Jin Kang]{\newline Department of Mathematical Sciences \newline Korea Advanced Institute of Science and Technology, Daejeon  34141, Republic of Korea}
\email{moonjinkang@kaist.ac.kr}

\author[Kim]{Jeongho Kim}
\address[Jeongho Kim]{\newline Department of Applied Mathematics, \newline Kyung Hee University, 1732 Deogyeong-daero, Giheung-gu, Yongin-si, Gyeonggi-do 17104, Republic of Korea}
\email{jeonghokim@khu.ac.kr}

\author[Lee]{Hobin Lee}
\address[Hobin Lee]{\newline Department of Mathematical Sciences \newline Korea Advanced Institute of Science and Technology, Daejeon  34141, Republic of Korea}
\email{lcuh11@kaist.ac.kr}

\begin{document}
	\newtheorem{theorem}{Theorem}[section]
	\newtheorem{lemma}{Lemma}[section]
	\newtheorem{corollary}{Corollary}[section]
	\newtheorem{proposition}{Proposition}[section]
	\newtheorem{remark}{Remark}[section]
	\newtheorem{definition}{Definition}[section]
	
	\renewcommand{\theequation}{\thesection.\arabic{equation}}
	\renewcommand{\thetheorem}{\thesection.\arabic{theorem}}
	\renewcommand{\thelemma}{\thesection.\arabic{lemma}}
	\newcommand{\bbr}{\mathbb R}
	\newcommand{\bbz}{\mathbb Z}
	\newcommand{\bbn}{\mathbb N}
	\newcommand{\bbs}{\mathbb S}
	\newcommand{\bbp}{\mathbb P}
	\newcommand{\bbt}{\mathbb T}
	\newcommand{\<}{\langle}
	\renewcommand{\>}{\rangle}
	\newcommand{\T}{\mathbb{T}}
	\newcommand{\N}{\mathbb{N}}
	\newcommand{\R}{\mathbb{R}}
	\newcommand{\lt}{\left}
	\newcommand{\rt}{\right}
	\newcommand{\bq}{\begin{equation}}
	\newcommand{\eq}{\end{equation}}
	\newcommand{\e}{\varepsilon}
	\newcommand{\mc}{\mathcal{C}}
	\newcommand{\pa}{\partial}
	\newcommand{\tU}{\widetilde{U}}
	\newcommand{\tu}{\widetilde{u}}
	\newcommand{\tv}{\widetilde{v}}
	\newcommand{\tw}{\widetilde{w}}
	\newcommand{\pv}{p(v)}
	\newcommand{\tp}{\widetilde{p}}
	\newcommand{\tpv}{p(\widetilde{v})}
	\newcommand{\norm}[1]{\left\lVert#1\right\rVert}
	\newcommand{\beq}{\begin{equation}}
	\newcommand{\eeq}{\end{equation}}

	
	
	
	\thanks{\textbf{Acknowledgment.} S. Han, M.-J. Kang and H. Lee were partially supported by the National Research Foundation of Korea  (NRF-2019R1C1C1009355  and NRF-2019R1A5A1028324) }

	\begin{abstract} 
		We consider the so-called Naiver-Stokes-Korteweg(NSK) equations for the dynamics of compressible barotropic viscous fluids with internal capillarity. 
		We handle the time-asymptotic stability in 1D of the viscous-dispersive shock wave that is a traveling wave solution to NSK as a  viscous-dispersive counterpart of  a Riemann shock.
		More precisely, we prove that when the prescribed far-field states of NSK are connected by a single Hugoniot curve, then solutions of NSK tend to the viscous-dispersive shock wave as time goes to infinity. To obtain the convergence, we extend the theory of $a$-contraction with shifts, used for the Navier-Stokes equations, to the NSK system. The main difficulty in analysis for NSK is due to the third-order derivative terms of the specific volume in the momentum equation. To resolve the problem, we introduce an auxiliary variable that is equivalent to the derivative of the specific volume. 
	\end{abstract}
	
	\maketitle
	
	\tableofcontents
	
	\section{Introduction}\label{sec:1}
	\setcounter{equation}{0}
	
	A study on the fluid model with an internal capillarity effect dates back to the works of Van der Waals and Korteweg \cite{K01, V94}, where the stress tensor may depend on the high-order derivative of the density. Later, Duun and Serrin \cite{DS85} introduced a thermodynamically consistent fluid model for internal capillarity, called the Navier-Stokes-Korteweg(NSK) equations. After its introduction, the NSK system has drawn a lot of attention and there has been numerous literature on the mathematical theory and application, due to its strong relationship with the quantum fluid models. We refer to the following literature and references therein for the readers who are interested in the state-of-the-art results on the Korteweg type fluids \cite{AS22,BGV19,BVY21,CL20}. 
	
	In this paper, we are interested in the time-asymptotic stability of the one-dimensional compressible fluid model of the Korteweg type. Consider the one-dimensional barotropic NSK equations in the Lagrangian mass coordinates:
	\begin{equation}\label{eq:NSK}
	\begin{aligned}
	&v_t-u_x=0,\quad (t,x)\in \bbr_+\times\bbr,\\
	&u_t+p(v)_x=\mu\left(\frac{u_x}{v}\right)_x + \kappa\left( \frac{-v_{xx}}{v^5} + \frac{5v_x^2}{2v^6}\right)_x,
	\end{aligned}
	\end{equation}
	where the unknown functions $v=v(t,x)$ and $u=u(t,x)$ represent the specific volume and velocity of the fluid, respectively. The pressure $p=p(v)$ is given by the $\gamma$-law, that is,
	\begin{equation*}
	p(v)=bv^{-\gamma}, \quad b>0, \quad \gamma>1.
	\end{equation*}
	Here, the constants $\mu>0$ and $\kappa>0$ represent the viscosity coefficient and capillary coefficient of the fluid, respectively. For simplicity, we normalize the coefficients so that $b=1, \mu=1$, and $\kappa=1$. The initial data of the NSK system \eqref{eq:NSK} is given by $(v_0,u_0)$, whose far-field states are prescribed as constants:
	\[\lim_{x \to \pm \infty} (v_0(x),u_0(x))=(v_\pm,u_\pm).\]
	When $\kappa=0$, that is, the capillarity effect is ignored, the NSK system \eqref{eq:NSK} is reduced to the standard compressible Navier-Stokes(NS) equations:
	\begin{equation}
	\begin{aligned}\label{eq:NS}
	&v_t-u_x=0,\quad (t,x)\in\bbr_+\times\bbr,\\
	&u_t +p(v)_x=\mu\left(\frac{u_x}{v}\right)_x.
	\end{aligned}
	\end{equation}
	Among many interesting topics on the NS equations \eqref{eq:NS}, the large-time behavior of solutions to \eqref{eq:NS} is one of the most important and motivated problems, as it is related to the inviscid limit to the Euler equation. Due to its significance, there has been a lot of previous literature on the time asymptotic behavior of the NS equations. Among the numerous results on the time-asymptotic stability of the NS equations \eqref{eq:NS}, we refer to \cite{G86,HKK23,KVW23,MN85,MN86,MW10}, although the list is totally not exhaustive. These results naturally motivate us to study the time-asymptotic behavior of the solution to the NSK equations \eqref{eq:NSK}. The large-time behavior of the NSK equations has a close relationship with the solution to the Euler equation
	\begin{equation}\label{eq:Euler}
	\begin{aligned}
	&v_t-u_x=0,\quad (t,x)\in\bbr_+\times\bbr,\\
	&u_t+p(v)_x=0,
	\end{aligned}
	\end{equation}
	subject to the Riemann initial data
	\begin{equation}
	\label{Riemann-data}
	(v(0,x),u(0,x))=
	\begin{cases} 
	(v_-,u_-), &x<0,\\
	(v_+,u_+), &x>0,
	\end{cases}
	\end{equation}
	as in the Navier-Stokes equations case \cite{M18}. We focus on the case when the end states $(v_\pm,u_\pm)$  are connected by a single Hugoniot curve. Without loss of generality, we only handle the case of a 2-shock curve. In other words, for a given right-end state $(v_+,u_+)$ we consider the left-end state $(v_-,u_-)$ that is on the 2-shock curve $S_2(v_+,u_+)$ satisfying the following Rankine-Hugoniot conditions:
	\begin{align}\label{RH}
	\begin{cases}
	-\sigma(v_+-v_-)-(u_+-u_-)=0,\\
	-\sigma(u_+-v_-)-(p(v_+)-p(v_-))=0,
	\end{cases}
	\quad \sigma:=\sqrt{-\frac{p(v_+)-p(v_-)}{v_+-v_-}}>0,
	\end{align}
	and the entropy condition:
	\begin{equation*}
	v_-<v_+, \quad u_->u_+.
	\end{equation*}
	Then the Riemann solution $(\overline{v},\overline{u})$ to the Euler equations \eqref{eq:Euler}--\eqref{Riemann-data} is given by 2-shock wave
	\begin{equation}\label{Riemann-solution}
	(\overline{v}(t,x),\overline{u}(t,x))=
	\begin{cases}
	(v_-,u_-) &\text{if} \quad x<\sigma t,\\
	(v_+,u_+) &\text{if} \quad x>\sigma t.
	\end{cases}
	\end{equation}
	
	For the case of NSK equations \eqref{eq:NSK}, the counterpart of the Riemann solution \eqref{Riemann-solution} is a viscous-dispersive shock, as a traveling wave $(\tv,\tu)(x-\sigma t)$ solution to \eqref{eq:NSK}, that satisfies the following ODEs:
	\begin{align}
	\begin{aligned}\label{viscous-dispersive-shock}
	&-\sigma \tv'-\tu'=0,\\
	&-\sigma \tu'+p(\tv)'=\left( \frac{\tu'}{\tv}\right)'+\left( \frac{-\tv''}{\tv^5} +\frac{5 (\tv')^2}{2 \tv^6} \right)',\\
	&(\tv,\tu)(\pm\infty)=(v_\pm,u_\pm).
	\end{aligned}
	\end{align}
	Similar to the Navier-Stokes equations, the time-asymptotic stability of the NSK system has been investigated in many literature. A first study on the stability of the NSK equations is due to \cite{C12}, where the authors provided the stability and the large-time behavior of the solutions toward the rarefaction wave, followed by the analysis on the large-time behavior of the solution perturbed from the viscous-dispersive shock wave \cite{CHZ15}. We also mention several results on the stability of the non-isentropic Navier-Stokes-Kortweg system for the case of contact wave \cite{CX13} and the composition of contact and rarefaction waves \cite{CS19}. We also refer to the stability result for the planar rarefaction wave for the three-dimensional NSK equations \cite{LL22}.
	
	In particular, the authors in \cite{CHZ15} used a classical anti-derivative method (cf. \cite{MN86}) for obtaining the time-asymptotic stability of viscous-dispersive shock wave, where the zero-mass condition for the initial perturbation is crucially imposed. For the NS system as in  \cite{HKK23,KVW23}, this zero-mass constraint on the initial data was removed by using the theory of $a$-contraction with shifts. 
	
	Therefore, the goal of the paper is to prove the time-asymptotic stability of the viscous-dispersive shock wave for the NSK equation \eqref{eq:NSK} without the zero mass condition, based on the theory of $a$-contraction with shifts. 
	
	The method of $a$-contraction with shifts was developed in \cite{KV16} for the stability of extremal shocks in the hyperbolic system of conservation laws, especially for the Euler system. The first extension of the method to a viscous system was done in the 1D scalar case \cite{Kang-V-1}  (\cite{Kang19} for a more general case), and then in the multi-D case \cite{KVW}. In the context of the one-dimensional barotropic NS system, this method was used to prove the contraction property of any large perturbations for a single viscous shock in \cite{KV21,KV-Inven}, and for a composite wave of two shocks in \cite{KV-2shock}. Furthermore, the method was also used in \cite{KVW23} to show the long-time behavior of the barotropic NS system for the composition of shock and rarefaction under the 1D perturbation, and for a single shock under the multi-D perturbation in \cite{WW}. Its extension to the Navier-Stokes-Fourier system was discussed in \cite{KVW-NSF}. As for applications of the method to other viscous hyperbolic systems, we also refer to \cite{CKKV,CKV}, particularly in the context of the viscous hyperbolic system arising from a chemotaxis model. \\

	Our main theorem reads as follows.
	
	\begin{theorem}\label{thm:main}
		For a given state $(v_+,u_+)\in\bbr^+\times\bbr$, there exist positive constants $C_0,\delta_0$, and $\e_1$ such that the following holds.
		For any $(v_-,u_-)$ on the 2-shock curve $S_2(v_+,u_+)$, that is, satisfying the Rankine-Hugoniot condition \eqref{RH}, such that $|v_+-v_-|<\delta_0$, denote $(\tv,\tu)(x-\sigma t)$ the 2-viscous-dispersive shock defined in \eqref{viscous-dispersive-shock}. Let $(v_0,u_0)$ be any initial data such that
		\[\sum_{\pm}\left(\|v_0-v_{\pm}\|_{L^2(\R_\pm)}+\|u_0-u_{\pm}\|_{L^2(\R_{\pm})}\right)+\|v_{0x}\|_{H^1(\R)}+\|u_{0x}\|_{L^2(\R)}<\e_0,\]
		where $\R_-:=-\R_+=(-\infty,0)$. Then, the Navier-Stokes-Korteweg system \eqref{eq:NSK} admits a unique global-in-time solution $(v,u)$. Moreover, there exists a Lipschitz continuous shift $X(t)$ such that
		\begin{equation*}
		\begin{aligned}
		&v(t,x)-\tv(x-\sigma t-X(t))\in C(0,\infty;H^2(\R)),\\
		&u(t,x)-\tu(x-\sigma t-X(t))\in C(0,\infty;H^1(\R)).
		\end{aligned}
		\end{equation*}
		In addition, we have
		\[\lim_{t\to\infty}\sup_{x\in\R}\left|(v,u)(t,x)-(\tv,\tu)(x-\sigma t-X(t))\right|=0\]
		and
		\beq\label{xlimit}
		\lim_{t\to\infty} |\dot{X}(t)|=0.
		\eeq
	\end{theorem}

	\begin{remark}
		Since \eqref{xlimit} implies
		$$
		\lim_{t\rightarrow+\infty}\frac{X(t)}{t}=0,
		$$
		the shift function $X(t)$ grows at most sub-linearly as $t\to\infty$. Thus, the shifted wave $\tU(x-\sigma t-X(t))$ tends to the original wave $\tU(x-\sigma t)$ time-asymptotically.
	\end{remark}

	\begin{remark}
		The results of Theorem \ref{thm:main} still hold for the NSK system with a general pressure $p(v)>0$ satisfying $p'(v) < 0, p''(v) > 0$ for $v>0$, and smooth viscosity $\mu=\mu(v)$ and smooth capillary $\kappa=\kappa(v)$, without meaningful added difficulties, since we consider small $H^2$-perturbations for $v$ variables. So, our result especially includes the cases of $\mu(v)=v^{-\alpha}$ and $\kappa(v)=v^{-\beta}$ for $\alpha, \beta\in\bbr$ as (cf. \cite{CCDZ})
		\begin{equation}
		\begin{aligned}\label{eq:NSK-general}
		&v_t-u_x=0,\quad (t,x)\in\bbr_+\times\bbr,\\
		&u_t +p(v)_x=\left(\frac{u_x}{v^{\alpha+1}}\right)_x+\left(-\frac{v_{xx}}{v^{\beta+5}}+\frac{\beta+5}{2}\frac{v_x^2}{v^{\beta+6}}\right)_x.
		\end{aligned}
		\end{equation}
		In particular, when $\beta=-1$, then the system \eqref{eq:NSK-general} represents the one-dimensional quantum fluid model in the Lagrangian coordinate. 
	\end{remark}
	
	The rest of the paper is organized as follows. In Section \ref{sec:prelim}, we provide several preliminaries, such as technical estimates on the relative quantities or the properties of the viscous-dispersive shock \eqref{viscous-dispersive-shock}. We also introduce an extended system for the Navier-Stokes-Korteweg equations in this section, which enables us to use the relative entropy method to the NSK system. Section \ref{sec:apriori} provides the a priori estimate on the perturbation, which guarantees the global existence of the solution to the NSK equation, as well as the time-asymptotic behavior of the solution. Then, we focus on proving a priori estimate. In Section \ref{sec:rel_ent}, we obtain $L^2$ estimates by the method of $a$-contraction with shift, and then we obtain the estimates on the high-order terms in Section \ref{sec:high-order}.

	\section{Preliminaries}\label{sec:prelim}
	\setcounter{equation}{0}
	In this section, we present several preliminary estimates on the relative quantities for the pressure and the internal energy. We also provide the existence and properties of viscous-dispersive shock in this section. Finally, we introduce several $O(1)$-constants and related estimates on them.
	
	\subsection{Estimates on the relative quantities}
	We present several upper and lower bounds on the relative quantities that will be used in estimating the relative entropy. For any function $F:(0,\infty)\to \bbr$ and $v,w\in (0,\infty)$, we define the relative quantity $F(v|w)$ as
	\[F(v|w):=F(v)-F(w)-F'(w)(v-w).\]
	In particular, when $F$ is convex, then the relative quantity is always positive. In the following lemma, we present several lower and upper bounds on the relative quantities for the pressure $p(v)=v^{-\gamma}$ and the internal energy $Q(v)=\frac{v^{1-\gamma}}{\gamma-1}$.
	
	\begin{lemma}\label{lem : Estimate-relative} Let $\gamma>1$ and $v_+$ be given constants. Then, there exists constants $C, \delta_*$ such that the following assertions hold:
		\begin{enumerate}
			\item For any $v,\bar{v}$ satisfying $0<\bar{v}<2v_+$ and $0<v<3v_+$,
			\begin{equation*}
			|v-\bar{v}|^2 \le C Q(v|\bar{v}), \quad |v-\bar{v}|^2 \le C p(v|w).
			\end{equation*}
			\item For any $v,\bar{v}$ satisfying $v,\bar{v} > v_+ /2$,
			\begin{equation*}
			|p(v)-p(\bar{v})| \le C |v-\bar{v}|.
			\end{equation*}
			\item For any $0<\delta<\delta_*$ and any $(v,\bar{v}) \in \mathbb{R}^2_+$ satisfying $|p(v)-p(\bar{v})|<\delta$ and $|p(\bar{v})-p(v_+)| < \delta,$
			\begin{equation*}
			\begin{aligned}
			&p(v|\bar{v}) \le \left( \frac{\gamma +1 }{2 \gamma } \frac{1}{p(\bar{v})} +C \delta \right) |p(v)-p(\bar{v})|^2,\\ 
			&Q(v|\bar{v}) \ge \frac{p(\bar{v})^{- \frac{1}{\gamma}-1}}{2 \gamma}|p(v)-p(\bar{v})|^2- \frac{1+\gamma}{3\gamma^2}p(\bar{v})^{- \frac{1}{\gamma}-2}(p(v)-p(\bar{v}))^3,\\ 
			&Q(v|\bar{v}) \le \left( \frac{p(\bar{v})^{- \frac{1}{\gamma}-1}}{2 \gamma} +C \delta \right)|p(v)-p(\bar{v})|^2.
			\end{aligned}
			\end{equation*}
			
		\end{enumerate}	
	\end{lemma}
	
	\begin{proof}
		Since the proofs are duplicates of those of \cite[Lemma 2.4, 2.5, and 2.6]{KV21}, we omit the proof.
	\end{proof}

	\subsection{Viscous-dispersive shock wave}
	In the following lemma, we present the existence of the viscous-dispersive shock wave, and its properties that are useful in our analysis. We consider a 2-shock connecting $(v_-,u_-)$ and $(v_+,u_+)$ such that $(v_-,u_-) \in S_2(v_+,u_+)$. The existence of viscous-dispersive shock wave for the Navier-Stokes-Korteweg equations was already studied in \cite{CHZ15}, but the condition on the end states for the existence is too complicated. In particular, the existence of the viscous-dispersive shock wave is not guaranteed even for a weak shock, that is, $|v_+-v_-|\ll1$. On the other hand, the authors in \cite{FPZ23} provide the method of slow manifold theory to show the existence of viscous-dispersive shock wave, which satisfies ODE that is similar to ours. Moreover, the only condition for the existence is the smallness of shock strength, as in our case. Therefore, we provide the existence and several important properties of the wave in the same spirit as in \cite{FPZ23}.
	
	\begin{lemma} \label{lem:shock-property}
		For a given right-end state $(v_+,u_+)$, there exists a positive constant $\delta_0$ such that the following statement holds. For any left end state $(v_-,u_-) \in S_2(v_+,u_+)$ with $|v_+-v_-| \sim |u_+ - u_-|=:\delta_S<\delta_0$, there exists a unique solution $(\tv,\tu)(\xi)$ to \eqref{viscous-dispersive-shock} such that $\tv(0)=\frac{v_-+v_+}{2}$. Moreover, the following estimates hold: there exists a positive constant $C$ such that
		\begin{equation}
		\begin{aligned}\label{shock-property}
		& \tu'<0,\quad \tv'>0,\\
		&C^{-1} \tv'(x-\sigma t) \le \tu'(x-\sigma t) \le C \tv' (x-\sigma t), \quad x \in \mathbb{R}, t>0,\\
		&|\tv(\xi)-v_\pm|\le C\delta_Se^{-C\delta_S|\xi|},\quad \pm \xi>0,\\
		&|\tv'(\xi)|\le C\delta_S^2e^{-C\delta_S|\xi|},\quad |\tv''(\xi)|\le C\delta_S|\tv'(\xi)|.
		\end{aligned}
		\end{equation}
	\end{lemma}
	
	\begin{proof}
		Since the proof is technical and lengthy, we postpone the proof to Appendix \ref{sec:app-shock-proof}.
	\end{proof}
	
	\begin{remark}
		We remark that the monotonicity of the viscous-dispersive shock is a consequence of the smallness of shock strength. Indeed, when the shock strength $\delta_S$ is not small, we cannot guarantee the monotonicity of the shock profile, and there might exist oscillation in the profile. See Figure \ref{fig:1} for the numerical simulation of the viscous-dispersive shocks with different shock strengths. 
	\end{remark}
	
	\begin{figure}[h!]
		\includegraphics[width=0.48\textwidth]{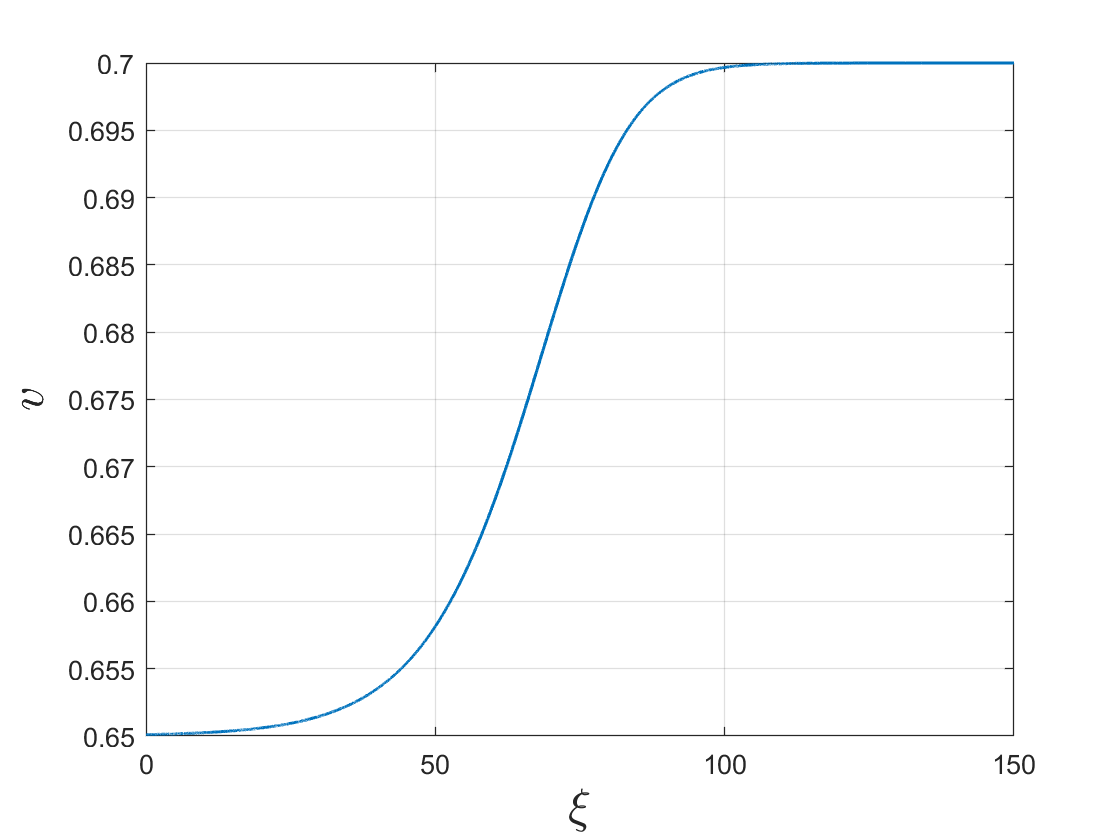}
		\includegraphics[width=0.48\textwidth]{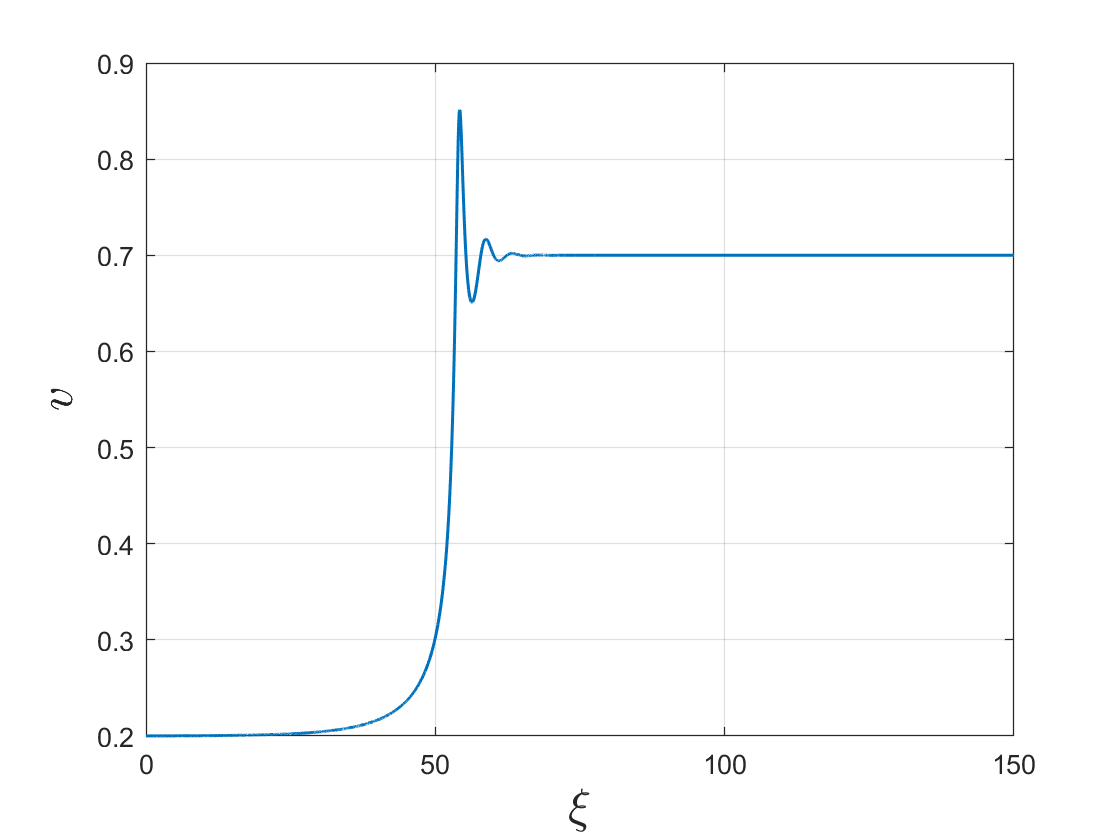}
		\caption{Profiles of $\tv$ for weak shock $\delta_S=0.05$ (left) and for strong shock $\delta_S=0.5$ (right), when the right-end state is fixed as $v_+=0.7$. The profile of the small shock is monotone, while that of the large shock has an oscillation. We also note that, even in the small shock case, the viscous-dispersive shock is not symmetric with respect to the inflection point, unlike the viscous shock of the classical Navier-Stokes equations.}
		\label{fig:1}
	\end{figure}
	
	\subsection{Useful $O(1)$-constants}
	In the later analysis, we will use the following $O(1)$-constants defined as
	\begin{equation} \label{O(1) constnats}
	\sigma_\ell:=\sqrt{-p'(v_-)}, \quad \alpha_\ell := \frac{\gamma+1}{2 \gamma \sigma_\ell p(v_-)}=\frac{p''(v_-)}{2|p'(v_-)|^2 \sigma_\ell}.
	\end{equation}
	These constants are indeed independent of the small shock strength $\delta_S$ since $v_+/2 \le v_- \le v_+$. Then, the following estimates on the $O(1)$-constants hold:
	\begin{equation}\label{shock_speed_est}
	|\sigma -\sigma_\ell|=\left|\sqrt{-\frac{p(v_+)-p(v_-)}{v_+-v_-}}-\sqrt{-p'(v_-)}\right| \le C \delta_S.
	\end{equation}
	Moreover, thanks to Lemma \ref{lem:shock-property}, the shock profile is monotone for the weak shock, and therefore $v_-\le \tv(x-\sigma t)\le v_+$ for all $x\in \R$. This yields the following estimates
	\begin{align}
	\begin{aligned}\label{shock_speed_est-2}
	&\|\sigma_\ell^2+p'(\tv)\|_{L^\infty}=\|p'(\tv)-p'(v_-)\|_{L^\infty} \le C \delta_S, \\
	&\left\| \frac{1}{\sigma_\ell^2}-\frac{p(\tv)^{-\frac{1}{\gamma}-1}}{\gamma} \right\|_{L^\infty}=\left\|\frac{(v_-)^{\gamma+1}}{\gamma }-\frac{p(\tv)^{-\frac{1}{\gamma}-1}}{\gamma}\right\|_{L^\infty} \le C \delta_S.
	\end{aligned}
	\end{align} 
	Throughout the paper, $C$ denotes a positive $O(1)$-constant which may change from line to line, but which is independent of the small constants like $\delta_S, \e_1$ and the lifespan $T$ given in Proposition \ref{apriori-estimate}.\\
	
	\subsection{Augmented system}
	We close this section, by introducing the augmented system for \eqref{eq:NSK}. Our first observation is that the natural dissipative entropy (or energy) for the NSK system \eqref{eq:NSK} is given by
	\[\int_{\R} \left(\frac{|u|^2}{2}+ Q(v) + \frac{(v_x)^2}{2v^5}\right)\,dx.\]
	Therefore, it is natural to introduce an auxiliary variable $w$ defined as
	\[w = -\frac{v_x}{v^{5/2}},\]
	so that the entropy can be written in terms of extended variable $U=(v,u,w)$ as
	\begin{equation}\label{entropy}
	\eta(U):=\int_{\R} \left(\frac{|u|^2}{2}+Q(v)+\frac{|w|^2}{2}\right)\,dx,
	\end{equation}
	whose relative functional would have a natural quadratic structure with respect to the variables $(v,u,w)$. Therefore, in order to use the well-established relative entropy method to control the $L^2$-perturbation, it is more convenient to consider an augmented system consisting of three variables $(v,u,w)$, instead of the original NSK equations \eqref{eq:NSK}. Using the equation of $v$, we deduce that $w$ satisfies
	\[w_t=\left( -\frac{v_t}{v^{5/2}} \right)_x=\left( -\frac{u_x}{v^{5/2}}\right)_x,\]
	and the Korteweg term in the momentum equation can be represented in terms of $w$ as
	\[-\frac{v_{xx}}{v^5}+\frac{5v^2_x}{2v^6}=\frac{w_x}{v^{5/2}}.\]
	Thus, the NSK system \eqref{eq:NSK} can be transformed into the following extended system with respect to $(v,u,w)$:
	\begin{equation}\label{NSK-w}
	\begin{cases}
	v_t-u_x=0,\\
	u_t+p(v)_x=\left(\frac{u_x}{v}\right)_x+ \left(\frac{w_x}{v^{5/2}}\right)_x,\\
	w_t=\left( -\frac{u_x}{v^{5/2}}\right)_x.
	\end{cases}
	\end{equation}
	Henceforth, we refer \eqref{NSK-w} as the NSK system, instead of the original system \eqref{eq:NSK}, unless otherwise specified. We also extend the viscous-dispersive shock wave $(\widetilde{v},\widetilde{u})$ to $(\widetilde{v},\widetilde{u},\widetilde{w})$, which obviously satisfies
	\begin{align}
	\begin{aligned}\label{viscous-dispersive-shock-ext}
	&-\sigma \tv_\xi-\tu_\xi=0,\\
	&-\sigma \tu_\xi+(p(\tv))_\xi=\left( \frac{\tu_\xi}{\tv} \right)_\xi+\left( \frac{\tw_\xi}{\tv^{5/2}} \right)_\xi,\\
	&-\sigma \tw_\xi=\left(-\frac{\tu_\xi}{\tv^{5/2}} \right)_\xi.
	\end{aligned}
	\end{align}
	
	As we will see in Section \ref{sec:rel_ent}, the system \eqref{NSK-w} can be written as a general hyperbolic equation, so that the classical relative entropy estimates \cite{D96, D79} can be directly applied.

	\section{A priori estimate and Proof of Theorem \ref{thm:main}}\label{sec:apriori}
	\setcounter{equation}{0}
	
	In this section, we first provide the a priori estimate for the perturbation, which is the key estimate for the main theorem. The proof of a priori estimate is presented in the next two sections. After stating the a priori estimate, we prove the global existence and time-asymptotic behavior of the solution, completing the proof of Theorem \ref{thm:main}.
	
	\subsection{Local existence}
	
	We first provide the local existence of strong solutions to the original NSK system \eqref{eq:NSK}, or equivalently, the NSK system \eqref{NSK-w}.
	
	\begin{proposition}\label{prop:local}
		Let $\underbar{v}$ and $\underbar{u}$ be smooth monotone functions such that
		\[\underbar{v}(x) = v_{\pm},\quad\underbar{u}(x)=u_{\pm},\quad\mbox{for}\quad\pm x\ge 1.\]
		Then, for any constants $M_0$, $M_1$, $\underline{\kappa}_0$, $\overline{\kappa}_0$, $\underline{\kappa}_{1}$, and $\overline{\kappa}_1$ with
		\[0<M_0<M_1,\quad\mbox{and}\quad0<\underline{\kappa}_1<\underline{\kappa}_0<\overline{\kappa}_0<\overline{\kappa}_1,\]
		there exists a finite time $T_0>0$ such that if the initial data $(v_0,u_0)$ satisfy
		\[\|v_0-\underline{v}\|_{H^2(\R)}+\|u_0-\underline{u}\|_{H^1(\R)}\le M_0,\quad\mbox{and}\quad \underline{\kappa}_0\le v_0(x)\le\overline{\kappa}_0,\quad\forall x\in\R,\]
		the Navier-Stokes-Korteweg equations \eqref{eq:NSK} admit a unique solution $(v,u)$ on $[0,T_0]$ satisfying
		\begin{align*}
		v-\underline{v}\in L^\infty ([0,T_0];H^2(\R))\cap L^2([0,T_0];H^3(\R)),\quad u-\underline{u}\in L^\infty([0,T_0];H^1(\R))\cap L^2([0,T_0];H^2(\R)),
		\end{align*}
		\[\|v-\underline{v}\|_{L^\infty([0,T_0];H^2(\R))}+\|u-\underline{u}\|_{L^\infty([0,T_0];H^1(\R))}\le M_1\]
		and
		\[\underline{\kappa}_1\le v(t,x)\le \overline{\kappa}_1,\quad \forall(t,x)\in [0,T_0]\times \R.\]
	\end{proposition}
	
	\begin{proof}
		The proof of the local existence can be obtained by using the standard argument of generating a sequence of approximate solutions and the Cauchy estimate, see for example \cite{S76}. For the brevity of the paper, we omit the proof.
	\end{proof}
	
	\subsection{Construction of shift}
	Next, we introduce the shift $X:\bbr_+\to\bbr$ as a solution to the following ODE:
	\begin{align}
	\begin{aligned}\label{ODE_X}
	\dot{X}(t)&=-\frac{M}{\delta_S}\Bigg(\int_{\R} a\left(x- \sigma t- X(t)\right)\widetilde{u}_x\left(x- \sigma t- X(t)\right)\big(u-\widetilde{u}(x- \sigma t- X(t)\big)\,d x\\
	&\hspace{2cm}+\frac{1}{\sigma}\int_\R a(x- \sigma t- X(t))\pa_xp\big(\widetilde{v}(x- \sigma t- X(t))\big)\big(v-\widetilde{v}(x- \sigma t- X(t))\big)\,d x\Bigg),
	\end{aligned}
	\end{align}
	where $M=\frac{5\sigma_\ell^3\alpha_\ell}{4}$. Then, the standard existence theorem for the ODE can be applied to guarantee the existence of the shift.
	
	\begin{proposition}
		For any $c_1,c_2,c_3>0$, there exists a constant $C>0$ such that the following is true. For any $T>0$, and any function $v,u\in L^\infty((0,T)\times\R)$ with
		\[c_1\le v(t,x)\le c_2,\quad |u(t,x)|\le c_3,\quad(t,x)\in[0,T]\times\R,\]
		the ODE \eqref{ODE_X} has a unique Lipschitz continuous solution $X$ on $[0,T]$. Moreover, we have
		\[|X(t)|\le Ct,\quad t\in[0,T].\]
	\end{proposition}
	
	As the name implies, the constructed shift $X(t)$ will play an important role in the theory of $a$-contraction with shift. In the following, we use the following abbreviated notation for the shifted function. For any function $g:\bbr\to\bbr$, we define
	\[g^X(\cdot):=g(\cdot-X(t)),\quad t\ge0.\]
	
	\subsection{A priori estimate}
	We now state the a priori estimate, which is the key estimate for obtaining the time-asymptotic behavior of the NSK equations.
	
	\begin{proposition}\label{apriori-estimate}
		For a given state $(v_+,u_+)\in\bbr^+\times\bbr$, there exist positive constants $C_0,\delta_0$, and $\e_1$ such that the following holds:
		
		Suppose that $(v,u,w)$ is the solution to \eqref{NSK-w} on $[0,T]$ for some $T>0$, and $(\tv,\tu,\tw)$ is defined in \eqref{viscous-dispersive-shock-ext}. Let $X$ be the Lipschitz continuous solution to \eqref{ODE_X} with weight function $a$ defined in \eqref{a}. Assume that the shock strength $\delta_S$ is less than $\delta_0$ and that
		\begin{align*}
		&v-\tv^X\in L^\infty(0,T;H^2(\bbr)),\\
		&u-\tu^X\in L^\infty(0,T;H^1(\bbr))\cap L^2(0,T;H^2(\bbr)),
		\end{align*}
		and
		\begin{equation}\label{smallness}
		\|v-\tv^X\|_{L^\infty(0,T;H^2(\bbr))}+\|u-\tu^X\|_{L^\infty(0,T;H^1(\bbr))}\le \e_1.
		\end{equation}
		Then, for all $0\le t\le T$,
		
		\begin{equation}
		\begin{aligned}\label{a-priori-1}
		&\sup_{t\in[0,T]}\left(\norm{v-\tv^X}_{L^2(\mathbb{R})}^2 +\norm{u-\tu^X}_{H^1(\mathbb{R})}^2+\norm{w-\tw^X}_{H^1(\mathbb{R})}^2\right)+\delta_S \int_0^t | \dot{X}(s)|^2 \, d s \\ 
		&\quad +\int_0^t \left( G_1+G_3+G^S \right) \, ds +  \int_0^t \left( D_{u_1} + D_{u_2} + G_{w} + G_{w_1} + G_{w_2} \right)\, ds  \\ 
		& \le C_0 \left(\norm{v_0-\tv}_{L^2(\mathbb{R})}^2 +\norm{u_0-\tu}_{H^1(\mathbb{R})}^2+\norm{w_0-\tw}_{H^1(\mathbb{R})}^2\right),
		\end{aligned}
		\end{equation}
		
		where $C_0$ is independent of $T$, and
		\begin{equation*}
		\begin{aligned}
		&G_1:=\int_\R |a_x^X|\left|p(v)-p(\tv^X)-\frac{u-\tu^X}{2C_*}\right|^2\,dx,\\
		&G_3:=\int_\R |a^X_x||w-\tw^X|^2\,dx,\\
		&G^S:=\int_\R |\tv_x^X||u-\tu^X|^2\,dx,\\
		&D_{u_1}:=\int_\R |(u-\tu^X)_x|^2\,dx,\quad D_{u_2}:=\int_\R |(u-\tu^X)_{xx}|^2\,dx,\\
		&G_w:=\int_\R |w-\tw^X|^2\,dx,\quad G_{w_1}:=\int_{\R}|(w-\tw^X)_x|^2\,dx,\quad G_{w_2}:=\int_{\R}|(w-\tw^X)_{xx}|^2\,dx.
		\end{aligned}
		\end{equation*}
		Here, $C_*$ is a positive constant defined in \eqref{C_star}.
	\end{proposition}
	
	\begin{remark}
		By the small perturbation of $v$ in $H^1$, and the definition of $w$-variable, $w$ is equivalent to the derivative of $v$. Therefore, the estimate \eqref{a-priori-1} is equivalent to the following another formulation for the a priori estimate:
		\begin{equation}
		\begin{aligned}\label{a-priori-2}
		& \sup_{t\in[0,T]}\left(\norm{v-\tv^X}_{H^2 (\mathbb{R})}^2 +\norm{u-\tu^X}_{H^1(\mathbb{R})}^2\right)+\delta_S \int_0^t | \dot{X}(s)|^2 \, d s \\ 
		&\quad +\int_0^t \left( G_1+G_3+G^S \right) \, ds +  \int_0^t \left( D_{u_1} + D_{u_2} + G_{w} + G_{w_1} + G_{w_2} \right)\, ds  \\ 
		& \le C_0 \left(\norm{v_0-\tv}_{H^2 (\mathbb{R})}^2 +\norm{u_0-\tu}_{H^1(\mathbb{R})}^2\right),
		\end{aligned}
		\end{equation}
		where $(v,u)$ is the solution to the original NSK equations \eqref{eq:NSK} and $(\tv,\tu)$ is defined in \eqref{viscous-dispersive-shock}
	\end{remark}
	
	\subsection{Global existence of perturbed solution}
	
	Using the a priori estimate \eqref{a-priori-1} and the equivalent form \eqref{a-priori-2}, we can extend the local solution obtained from Proposition \ref{prop:local} to the global one by using the standard continuation argument. We first choose smooth functions $\underline{v}$ and $\underline{u}$ that satisfy
	\begin{equation}\label{ubarvbar}
	\sum_{\pm}\left(\|\underline{v}-v_\pm\|_{L^2(\R_\pm)}+\|\underline{u}-u_\pm\|_{L^2(\R_\pm)}\right)+\|\pa_x\underline{v}\|_{H^1(\R)}+\|\pa_x\underline{u}\|_{H^1(\R)}\le C\delta_S.
	\end{equation}
	
	Then, we use the estimates on the shock wave \eqref{shock-property} to obtain
	\begin{equation}
	\begin{aligned}\label{est-init}
	&\norm{\underline{v}-\tv^X}_{H^2(\mathbb{R})}+\norm{\underline{u}-\tu^X}_{H^1(\mathbb{R})}\\
	&\le \sum_{\pm}\left(\norm{\underline{v}-v_{\pm}}_{L^2(\R_\pm)}+\norm{\underline{u}-u_{\pm}}_{L^2(\R_\pm)}\right) + \|\tv^X-v_+\|_{L^2(\R_+)}+\|\tv^X-v_-\|_{L^2(\R_-)}\\
	&\quad +\|\pa_x \underline{v}\|_{H^1(\mathbb{R})} + \|\tv^X_x\|_{H^1(\mathbb{R})}+\|\tu^X-u_+\|_{L^2(\R_+)}+\|\tu^X-u_-\|_{L^2(\R_-)}+ \|\pa_x \underline{u}\|_{H^1(\mathbb{R})}+\|\tu^X_x\|_{L^2(\mathbb{R})}\\
	&\le C\sqrt{\delta_S} 
	\end{aligned}
	\end{equation}
	Now, for sufficiently small $\delta_S$ we choose $\e_0$ as 
	\[\e_0 < \frac{\e_1}{3}-C\sqrt{\delta_S}.\]
	Consider any initial data $(v_0,u_0)$ such that
	\[\sum_{\pm}\left(\|v_0-v_{\pm}\|_{L^2(\R_\pm)}+\|u_0-u_\pm\|_{L^2(\R_\pm)}\right)+\|v_{0x}\|_{H^1}+\|u_{0x}\|_{L^2}<\e_0.\]
	Then, we use \eqref{ubarvbar} to obtain
	\begin{align*}
	&\|v_0-\underline{v}\|_{H^2(\mathbb{R})}+\|u_0-\underline{u}\|_{H^1(\mathbb{R})}\\
	&\le\sum_{\pm}\left(\|v_0-v_\pm\|_{L^2(\R_\pm)}+\|u_0-u_{\pm}\|_{L^2(\R_\pm)}+\|\underline{v}-v_\pm\|_{L^2(\R_\pm)}+\|\underline{u}-u_{\pm}\|_{L^2(\R_\pm)}\right)\\
	&\quad + \|v_{0x}\|_{H^1(\mathbb{R})}+\|u_{0x}\|_{L^2(\mathbb{R})}+\|\underline{v}_x\|_{H^1(\mathbb{R})}+\|\underline{u}_x\|_{L^2(\mathbb{R})}\\
	&\le \e_0 +C\sqrt{\delta_S}< \frac{\e_1}{3}.
	\end{align*}
	From the smallness of $\e_1$ and Sobolev embedding, we have
	\[\frac{v_-}{2}\le v_0(x)\le 2v_+,\quad x\in\R\]
	and by the local existence result in Proposition \ref{prop:local}, there exists $T_0>0$ such that
	\begin{equation}\label{est-local-1}
	\|v-\underline{v}\|_{L^\infty(0,T_0;H^2(\mathbb{R}))}+\|u-\underline{u}\|_{L^\infty(0,T_0;H^1(\mathbb{R}))}\le \frac{\e_1}{2}
	\end{equation}
	and
	\[\frac{v_-}{3}\le v(t,x)\le 3v_+,\quad (t,x)\in[0,T_0]\times\R.\]
	On the other hand, we estimate the difference between $(\underline{v},\underline{u})$ and $(\tv^X,\tu^X)$ by using similar estimate as in \eqref{est-init} and \eqref{shock-property} as 
	\begin{align*}
	&\|\underline{v}-\tv^X(t,\cdot)\|_{H^2(\mathbb{R})}+\|\underline{u}-\tu^X(t,\cdot)\|_{H^1(\mathbb{R})}\\
	&\le \sum_{\pm}\left(\|\underline{v}-v_{\pm}\|_{L^2(\R_\pm)}+\|\underline{u}-u_{\pm}\|_{L^2(\R_\pm)}+\|\tv^X-v_\pm\|_{L^2(\R_\pm)}+\|\tu^X-u_\pm\|_{L^2(\R_\pm)}\right)\\
	&\quad +\|\tv^X_x\|_{H^1(\mathbb{R})}+\|\tu^X_x\|_{L^2(\mathbb{R})}+\|\underline{v}_x\|_{H^1(\mathbb{R})}+\|\underline{u}_x\|_{L^2(\mathbb{R})}\\
	&\le C\sqrt{\delta_S}(1+\sqrt{|X(t)|})\le C\sqrt{\delta_S}(1+\sqrt{t}).
	\end{align*}
	Taking $T_1\in(0,T_0)$ small enough so that $C\sqrt{\delta_S}(1+\sqrt{T_1})\le \frac{\e_1}{2}$, we have
	\begin{equation}\label{est-local-2}
	\|\underline{v}-\tv^X\|_{L^\infty(0,T_1;H^2(\mathbb{R}))}+\|\underline{u}-\tu^X\|_{L^\infty(0,T_1;H^1(\mathbb{R}))}\le\frac{\e_1}{2}.
	\end{equation}
	Combining \eqref{est-local-1} and \eqref{est-local-2} yields
	\[\|v-\tv^X\|_{L^\infty(0,T_1;H^2(\mathbb{R}))}+\|u-\tu^X\|_{L^\infty(0,T_1;H^1(\mathbb{R}))}\le\e_1.\]
	Then, the a priori estimate \eqref{a-priori-2} implies that $T_1$ can be extended to $+\infty$, and the global existence is proved. In particular, we have
	\begin{equation}
	\begin{aligned}\label{est-infinite}
	& \sup_{t>0}\left(\norm{v-\tv^X}_{H^2 (\mathbb{R})}^2 +\norm{u-\tu^X}_{H^1(\mathbb{R})}^2\right)+\delta_S \int_0^\infty | \dot{X}(t)|^2 \, d t \\ 
	&\quad +\int_0^\infty \left( G_1+G_3+G^S \right) \, dt+  \int_0^\infty \left( D_{u_1} + D_{u_2} + G_{w} + G_{w_1} + G_{w_2} \right)\, dt  \\ 
	& \le C_0 \left(\norm{v_0-\tv}_{H^2 (\mathbb{R})}^2 +\norm{u_0-\tu}_{H^1(\mathbb{R})}^2\right)<\infty
	\end{aligned}
	\end{equation}
	and, for $t>0$,
	\begin{equation} \label{eq: X bound}
	|\dot{X}(t)|\le C_0 \left(\|(v-\tv^X)(t,\cdot)\|_{L^\infty(\bbr)}+\|(u-\tu^X)(t,\cdot)\|_{L^\infty(\bbr)}\right).
	\end{equation}
	
	\subsection{Time-asymptotic behavior}
	We are now ready to prove the time-asymptotic behavior of the perturbation. We first define
	\[g(t):=\|(v-\tv^X)_x\|_{L^2(\mathbb{R})}^2+\|(u-\tu^X)_x\|_{L^2(\mathbb{R})}^2.\]
	We will show that $g\in W^{1,1}(\R_+)$ which implies $\lim_{t\to\infty}g(t)=0$. Then, the Gagliardo-Nirenberg interpolation inequality and the uniform bound estimate \eqref{est-infinite} implies
	\begin{equation} \label{eq: v,u limit}
	\lim_{t\to\infty}\left(\|v-\tv^X\|_{L^\infty(\mathbb{R})}+\|u-\tu^X\|_{L^\infty(\mathbb{R})}\right)=0.
	\end{equation}
	Furthermore, \eqref{eq: X bound} and \eqref{eq: v,u limit} imply that
	\[ \lim_{t \to \infty} |\dot{X}(t)| \le C_0 \lim_{t \to \infty} \left( \norm{(v-\tv^X)(t,\cdot)}_{L^\infty(\mathbb{R})} + \norm{(u-\tu^X)(t,\cdot)}_{L^\infty(\mathbb{R})} \right)=0.\]
	Therefore, it remains to show that $g\in W^{1,1}(\R_+)$.\\
	
	\noindent (1) $g\in L^1(\bbr_+)$: We use the definition of $w$ variable to observe that
	\begin{align*}
	|(v-\tv^X)_x| &= |wv^{5/2}-\tw^X(\tv^X)^{5/2}|\le v^{5/2}|w-\tw^X|+|\tw^X||v^{5/2}-(\tv^X)^{5/2}|\\
	&\le C|w-\tw^X| + C|\tw^X||v-\tv^X|\le C|w-\tw^X|+C|\tv_x^X||v-\tv^X|\\
	&\le C|w-\tw^X|+C|\tv_x^X||p(v)-p(\tv^X)|\\
	&\le C|w-\tw^X|+C|\tv_x^X|\left|p(v)-p(\tv^X)-\frac{u-\tu^X}{2C^*}\right|+C|\tv_x^X||u-\tu^X|.
	\end{align*}
	Therefore, we obtain
	
	\begin{align*}
	\int_0^\infty |g(t)|\,dt&=\int_0^\infty\int_{\R} |(v-\tv^X)_x|^2+|(u-\tu^X)_x|^2\,dxdt\\
	&\le C\int_0^\infty\int_\R |w-\tw^X|^2+ |\tv_x^X|^2\left|p(v)-p(\tv^X)-\frac{u-\tu^X}{2C^*}\right|^2\\
	&\hspace{2cm} +|\tv^X_x|^2|u-\tu^X|^2+|(u-\tu^X)_x|^2\,dxdt\\
	&\le C\int_0^\infty (G_w + G_1+G^S+ D_{u_1})\,d t<\infty,
	\end{align*}
	where we used \eqref{est-infinite} in the last inequality. This implies $g\in L^1(\R_+)$.\\
	
	\noindent (2) $g'\in L^1(\R_+)$: We combine the system \eqref{NSK-w} and \eqref{viscous-dispersive-shock-ext} to obtain
	\begin{align}
	\begin{aligned}\label{eq:diff}
	&(v-\tv^X)_t - (u-\tu^X)_x = \dot{X}(t)\tv_x^X,\\
	&(u-\tu^X)_t + (p(v)-p(\tv^X))_x = \left(\frac{u_x}{v}-\frac{\tu_x^X}{\tv^X}\right)_x+\left(\frac{w_x}{v^{5/2}}-\frac{\tw_x^X}{(\tv^X)^{5/2}}\right)_x+\dot{X}(t)\tu_x^X,\\
	&(w-\tw^X)_t = -\left(\frac{u_x}{v^{5/2}}-\frac{\tu_x^X}{(\tv^X)^{5/2}}\right)_x +\dot{X}(t) \tw_x^X.
	\end{aligned}
	\end{align}
	Then, using the equation \eqref{eq:diff}, we estimate the time-integration of $g'$ as
	
	\begin{equation}
	\begin{aligned}\label{est:gprime}
	\int_0^\infty&|g'(t)|\,dt\\
	&=\int_0^\infty 2\left|\int_\R (v-\tv^X)_x(v-\tv^X)_{xt}\,dx+\int_\R (u-\tu^X)_x(u-\tu^X)_{xt}\,dx\right|\,dt\\
	&\le 2\int_0^\infty \left|\int_\R (v-\tv^X)_x\left((u-\tu^X)_{xx}+\dot{X}(t)\tv^X_{xx}\right)\,dx\right|\,dt\\
	&\quad + 2\int_0^\infty \Bigg|\int_\R (u-\tu^X)_x\Bigg(-(p(v)-p(\tv^X))_{xx}+\left(\frac{u_x}{v}-\frac{\tu_x^X}{\tv^X}\right)_{xx}\\
	&\hspace{4cm}+\left(\frac{w_x}{v^{5/2}}-\frac{\tw_x^X}{(\tv^X)^{5/2}}\right)_{xx}+\dot{X}(t)\tu_{xx}^X\Bigg)\Bigg|\,dt\\
	&=2\int_0^\infty \left|\int_\R (v-\tv^X)_x\left((u-\tu^X)_{xx}+\dot{X}(t)\tv^X_{xx}\right)\,dx\right|\,dt\\
	&\quad + 2\int_0^\infty \Bigg|\int_\R (u-\tu^X)_{xx}\Bigg((p(v)-p(\tv^X))_{x}-\left(\frac{u_x}{v}-\frac{\tu_x^X}{\tv^X}\right)_{x}-\left(\frac{w_x}{v^{5/2}}-\frac{\tw_x^X}{(\tv^X)^{5/2}}\right)_{x}\Bigg)\\
	&\hspace{3cm} + \int_\R (u-\tu^X)_x\dot{X}(t)\tu^X_{xx}\Bigg|\,dt\\
	&\le C\int_0^\infty (G_w+G_1+G^S+D_{u_1}+D_{u_2}+|\dot{X}(t)|^2)\,dt\\
	&\quad + C\int_0^\infty \int_\R \left|\left(\frac{u_x}{v}-\frac{\tu_x^X}{\tv^X}\right)_{x}\right|^2+\left|\left(\frac{w_x}{v^{5/2}}-\frac{\tw_x^X}{(\tv^X)^{5/2}}\right)_{x}\right|^2\,dxdt.
	\end{aligned}
	\end{equation}
	
	Since the first term in the right-hand side of \eqref{est:gprime} can be bounded by \eqref{est-infinite}, we only need to estimate the last two terms. The first one of the last two terms can be estimated as in the usual Navier-Stokes equations. However, as we obtain a uniform $H^2$-norm of $v$ perturbation, we can directly control the $L^\infty$-norm of $(v-\tv^X)_x$, which makes the estimate simpler. Precisely, we obtain
	\begin{align*}
	\int_0^\infty& \int_\R \left|\left(\frac{u_x}{v}-\frac{\tu^X_x}{\tv^X}\right)_x\right|^2\,dxdt\\
	&=\int_0^\infty\int_\R \Bigg|\frac{1}{v}(u-\tu^X)_{xx}+\tu^X_{xx}\left(\frac{1}{v}-\frac{1}{\tv^X}\right)-\frac{1}{v^2}(v-\tv^X)_{x}(u-\tu^X)_x\\
	&\hspace{2cm}-\frac{\tv^X_x}{v^2}(u-\tu^X)_x-\frac{\tu^X_x}{v^2}(v-\tv^X)_x-\tv^X_x\tu^X_x\left(\frac{1}{v^2}-\frac{1}{(\tv^X)^2}\right)\Bigg|^2\,dxdt\\
	&\le C\int_0^\infty\int_\R \Bigg(|(u-\tu^X)_{xx}|^2 + |\tu_x^X|^2|v-\tv^X|^2+|(u-\tu^X)_x|^2|(v-\tv^X)_x|^2\\
	&\hspace{3cm}+|\tv_x^X|^2|(u-\tu^X)_x|^2+|\tu_x^X|^2|(v-\tv^X)_x|^2+|\tv^X_x|^2|\tu^X_x|^2|v-\tv^X|^2\Bigg)\,dxdt,
	\end{align*}
	and consequently,
	\begin{align*}
	\int_0^\infty& \int_\R \left|\left(\frac{u_x}{v}-\frac{\tu^X_x}{\tv^X}\right)_x\right|^2\,dxdt\\
	&\le C\int_0^\infty\left(G_w+G_1+G^S+D_{u_1}+D_{u_2}\right)\,dt\\
	&\quad + C\|(v-\tv^X)_x\|^2_{L^\infty((0,\infty)\times\R)}\int_0^\infty\int_\R |(u-\tu^X)_x|^2\,dxdt\\
	&\le C\int_0^\infty\left(G_w+G_1+G^S+D_{u_1}+D_{u_2}\right)\,dt+ C\int_0^\infty D_{u_1}\,dt<+\infty.
	\end{align*}
	
	Finally, we estimate the last term as
	\begin{align*}
	\int_0^\infty&\int_\R \left|\left(\frac{w_x}{v^{5/2}}-\frac{\tw^X_x}{(\tv^X)^{5/2}}\right)_x\right|^2\,dxdt\\
	&=\int_0^\infty\int_\R\Bigg|\frac{(w-\tw^X)_{xx}}{v^{5/2}}+\tw^X_{xx}\left(\frac{1}{v^{5/2}}-\frac{1}{(\tv^X)^{5/2}}\right)-\frac{5}{2}\frac{1}{v^{7/2}}(v-\tv^X)_x(w-\tw^X)_x\\
	&\hspace{2cm}-\frac{5}{2}\frac{\tv^X_x}{v^{7/2}}(w-\tw^X)_x-\frac{5}{2} \frac{\tw^X_x}{v^{7/2}}(v-\tv^X)_x-\frac52\tv^X_x\tw^X_x\left(\frac{1}{v^{7/2}}-\frac{1}{(\tv^X)^{7/2}}\right)\Bigg|^2\,dxdt\\
	&\le C\int_0^\infty\int_\R \Bigg(|(w-\tw^X)_{xx}|^2+|\tw^X_{xx}|^2|v-\tv^X|^2+|(v-\tv^X)_x|^2|(w-\tw^X)_x|^2 \\
	&\hspace{3cm}+ |\tw^X_x|^2|v-\tv^X|^2 + |\tv^X_x|^2|(w-\tw^X)_x|^2 + |\tv^X_x|^2|\tw^X_x|^2|v-\tv^X|^2\Bigg)\,dxdt.
	\end{align*}
	
	However, using the definition of $w$ variable $\tw = -\frac{\tv_x}{\tv^{5/2}}$, we derive
	\[\tw_x = -\frac{\tv_{xx}}{\tv^{5/2}}+\frac{5}{2}\frac{(\tv_x)^2}{\tv^{7/2}},\quad \tw_{xx} = -\frac{\tv_{xxx}}{\tv^{5/2}}+\frac{15}{2}\frac{\tv_x\tv_{xx}}{\tv^{7/2}}-\frac{35}{4}\frac{(\tv_x)^2}{\tv^{9/2}},\]
	which, together with \eqref{shock-property}, implies
	\[|\tw_x|\le C|\tv_x|,\quad |\tw_{xx}|\le C|\tv_x|.\]
	Therefore, we obtain
	\begin{align*}
	\int_0^\infty&\int_\R\left|\left(\frac{w_x}{v^{5/2}}-\frac{\tw^X_x}{(\tv^X)^{5/2}}\right)_x\right|^2\,dxdt\\
	&\le C\int_0^\infty \left(G_1+G^S+D_{w_1}+D_{w_2}\right) \,dt+ C\int_0^\infty |(v-\tv^X)_x|^2|(w-\tw^X)_x|^2\,dxdt\\
	&\le C\int_0^\infty \left(G_1+G^S+D_{w_1}+D_{w_2}\right) \,dt+ C\|v-\tv^X\|_{L^\infty(0,\infty;H^2(\R))}\int_0^\infty D_{w_1}\,dt<+\infty.
	\end{align*}
	which proves $g'\in L^1(\R_+)$. Thus, we have shown that $g\in W^{1,1}(\R_+)$. This completes the proof of the asymptotic behavior of the NSK equations. Therefore, once we have the a priori estimate in Proposition \ref{apriori-estimate}, we prove the time asymptotic behavior of the NSK equations. \\
	
	In the following sections, we will prove Proposition \ref{apriori-estimate} by using the theory of $a$-contraction with shift. In Section \ref{sec:rel_ent}, we provide the estimate on the relative entropy between the solution to the NSK equation and the viscous dispersive shock with shifts, which gives the $L^2$-estimate for $(v,u,w)$ perturbations. Then, we obtain $H^1$-estimate for $(v,u,w)$ in Section \ref{sec:high-order}.
	
	\section{Estimate on the weighted relative entropy with the shift}\label{sec:rel_ent}
	\setcounter{equation}{0}
	In this section, we estimate the $L^2$-perturbation of a solution to the NSK equations \eqref{NSK-w} from the viscous-dispersive shock profile \eqref{viscous-dispersive-shock-ext} by using the method of $a$-contraction. 
	The main goal of this section is to verify the following control on the $L^2$-perturbation between the solution $(v,u,w)$ to \eqref{NSK-w} and the viscous-dispersive shock $(\widetilde{v},\widetilde{u},\widetilde{w})$ defined as in  \eqref{viscous-dispersive-shock-ext}.
	\begin{lemma} \label{Main Lemma} There exists a positive constant $C$ such that for all $t \in [0,T],$
		\begin{equation} \label{energy-est}
		\begin{aligned}
		&\int_\mathbb{R} \big( |u-\tu^X|^2 +|v-\tv^X|^2  + |w-\tw^X|^2 \big) \, d x + \int_0^t \left( \delta_S | \dot{X}|^2 + G_1+G_3+G^S+D \right) \, ds \\ 
		&\quad \le C \int_\mathbb{R} \big( |u_0-\tu|^2 +|v_0- \tv|^2 + |w_0-\tw|^2 \big) \, dx + C\sqrt{\delta_S} \int_0^t \| (w-\tw^X)_x\|_{L^2(\mathbb{R})}^2 \, ds,
		\end{aligned}
		\end{equation}
		where
		\begin{align*}
		G_1&= \int_\mathbb{R} |a_x^X| \left| p(v)-p(\tv^X) - \frac{u-\tu^X}{2C_*} \right|^2 \, dx,\\
		G_3 &= \int_\mathbb{R} |a_x^X| |w-\tw^X|^2 \, dx,\\
		G^S &= \int_\mathbb{R} |\tv_x^X| |u-\tu^X|^2 \, dx,\\
		D_1 &= \int_\mathbb{R} |\partial_x(u-\tu^X)|^2 \, dx.
		\end{align*}
		
	\end{lemma}

	\subsection{Construction of weight function}
	Instead of directly estimating the relative entropy as in \eqref{energy-est}, we will estimate the weighted relative entropy. To this end, we first construct the weight function $a=a(t,x)$ as
	\begin{equation} \label{a}
	a(t,x):=1+\frac{u_--\widetilde{u}(x-\sigma t)}{\sqrt{\delta_S}},
	\end{equation}
	where $\delta_S:=|u_- - u_+| $ denotes the shock strength. It follows from the definition of the weight function $a$ that $1 \le a \le 1+\sqrt{\delta_S}<\frac{3}{2}$ and
	\begin{equation}\label{a_x}
	\partial_x a=-\frac{\partial_x \tu}{\sqrt{\delta_S}}=\frac{\sigma\pa_x \tv}{\sqrt{\delta_S}}>0,\quad \mbox{and therefore},\quad |a_x|\sim \frac{|\pa_x\tv|}{\sqrt{\delta_S}}
	\end{equation}
	where we used \eqref{viscous-dispersive-shock-ext}$_1$.
	
	\subsection{Relative entropy method}
	To prove Lemma \ref{Main Lemma}, we basically rely on the method of $a$-contraction with shifts, which uses relative entropy. Therefore, we rewrite the NSK system \eqref{NSK-w} in the following abstract form:
	\begin{equation}\label{eq:NS-abs}
	\partial_t U +\partial_x A(U) = \partial_x (M(U) \partial_x D \eta (U)),
	\end{equation}
	where 
	\[U:=\begin{pmatrix}
	v\\u\\w
	\end{pmatrix},\quad 
	A(U):=\begin{pmatrix}
	-u\\  p(v) \\0
	\end{pmatrix},\quad 
	D\eta(U):=\begin{pmatrix}
	-p(v)\\  u \\w,
	\end{pmatrix},\]
	and 
	\begin{equation*}
	M(U):=
	\begin{pmatrix}
	0 &0 &0\\
	0 &v^{-1} &v^{-5/2}\\
	0 &-v^{-5/2} &0
	\end{pmatrix}.
	\end{equation*}
	Here, we use the convex entropy $\eta$ defined in \eqref{entropy} and $D\eta$ denotes the gradient of $\eta$ with respect to the variables $U=(v,u,w)$. Similarly, the viscous-dispersive shock profile $\widetilde{U}=(\tv,\tu,\tw)$ satisfies
	\[-\sigma \widetilde{U}_\xi+(A(\widetilde{U}))_\xi=(M(\tU)(D\eta(\tU))_\xi)_\xi.\]
	
	Now, consider the viscous-dispersive wave with the shift $\widetilde{U}^X$ as
	\begin{equation} \label{shifted underlying wave}
	\widetilde{U}^X(t,x):=\widetilde{U}(x-\sigma t-X(t)),
	\end{equation}
	where the shift $X(t)$ is a Lipschitz continuous function determined later in \eqref{shift}. Then, it is straightforward to observe that $\tU^X$ satisfies
	\[\partial_t\widetilde{U}^X +\partial_x A(\widetilde{U}^X)=\partial_x \left(M(\tU^X) \partial_x D \eta(\tU^X) \right) -\dot{X} \partial_x \widetilde{U}^X.\]
	
	As we mentioned above, we will use the relative entropy to measure the perturbation between two solutions. We define the relative entropy between $U=(v,u,w)$ and $\overline{U}=(\overline{v},\overline{u},\overline{w})$ as
	\[\eta(U|\overline{U}):=\frac{|u-\overline{u}|^2}{2}+Q(v|\overline{v})+\frac{|w-\overline{w}|^2}{2}\]
	and also define the relative flux $A(U|\overline{U})$ and relative entropy flux $G(U;\overline{U})$ as
	\[A(U|\overline{U})=A(U)-A(\overline{U})-DA(\overline{U})(U-\overline{U}),\]
	and 
	\[G(U;\overline{U})=G(U)-G(\overline{U})-D\eta(\overline{U})(A(U)-A(\overline{U})),\]
	where $G$ is the entropy flux for $\eta$ satisfying the condition $D_iG(U)=\sum_{k=1}^3 D_k \eta(U)D_iA_k(U)$. In the case of NSK system, we consider $G(U)=p(v)u$, and therefore, we can compute $A(U|\overline{U})$ and $G(U;\overline{U})$ as
	\[A(U|\overline{U})=\begin{pmatrix}
	0\\
	p(v|\overline{v})\\
	0
	\end{pmatrix}, \quad \text{and} \quad 
	G(U;\overline{U})=(p(v)-p(\overline{v}))(u-\overline{u}).\]
	Here, the relative internal energy $Q(v|\overline{v})$ and the relative pressure $p(v|\overline{v})$ are defined as
	\[Q(v|\overline{v})=Q(v)-Q(\overline{v})-Q'(\overline{v})(v-\overline{v}) \quad \mbox{and} \quad p(v|\overline{v})=p(v)-p(\overline{v})-p'(\overline{v})(v-\overline{v}).\]

	In order to obtain $L^2$-perturbation estimate in Lemma \ref{Main Lemma}, we focus on estimating the weighted relative entropy between the solution $U$ and the shifted viscous-dispersive shock wave $\tU^X$:
	\[\int_\mathbb{R} a^X (t,x) \eta\left(U(t,x)|\widetilde{U}^X(t,x)\right)\,d x,\quad\mbox{where}\quad a^X(t,x):=a(t,x-X(t)).\]
	
	\begin{lemma} \label{lem:rel-ent}
		Let $a$ be the weight function defined by \eqref{a} and $X:[0,T]\to\bbr$ be any Lipschitz continuous function. Let $U$ be a solution to \eqref{eq:NS-abs},  and $\tU^X$ be the shifted viscous-dispersive shock wave defined in  \eqref{shifted underlying wave}.  Then
		\begin{equation}\label{est-weight-rel-ent} 
		\frac{d}{dt}\int_\mathbb{R} a^X (t,x) \eta (U(t,x))|\tU^X(t,x)) \, dx=\dot{X}(t)Y(U)+\mathcal{J}^{\textup{bad}}(U)-\mathcal{J}^{\textup{good}}(U),
		\end{equation}
		where the terms $Y(U)$, $\mathcal{J}^{\textup{bad}}$, and $\mathcal{J}^{\textup{good}}$ are defined as
		\begin{align*}
		Y(U)&:=-\int_\mathbb{R} a_x^X \eta(U|\widetilde{U}^X)\,d x +\int_\mathbb{R} a^X D^2\eta(\widetilde{U}^X)\widetilde{U}_x^X (U-\widetilde{U}^X)\,d x,\\
		\mathcal{J}^{\textup{bad}}(U)&:=\int_\mathbb{R} a^X_x (p(v)-p(\tv^X))(u-\tu^X) \, dx-\int_\mathbb{R} a^X \tu_x^X p(v | \tv^X) \,d x\\
		&\quad-\int_\mathbb{R} a^X_x\left( \frac{u-\tu^X}{v} \partial_x (u-\tu^X) +\frac{(u-\tu^X)}{v^{5/2}} \partial_x (w-\tw^X)-\frac{(w-\tw^X)}{v^{5/2}} \partial_x (u-\tu^X)\right) dx\\
		&\quad+ \int_\mathbb{R} a^X_x \Bigg(  \frac{(u-\tu^X)(v-\tv^X) \partial_x \tu^X}{v \tv^X} \\
		&\hspace{2cm}+\frac{(v^{5/2}-(\tv^X)^{5/2})(u-\tu^X)\partial_x \tw^X}{v^{5/2}(\tv^X)^{5/2}} -\frac{(v^{5/2}-(\tv^X)^{5/2})(w-\tw^X)\partial_x \tu^X}{v^{5/2}(\tv^X)^{5/2}}  \Bigg) dx\\
		&\quad+\int_\mathbb{R} a^X \Bigg(\partial_x (u-\tu^X) \left( \frac{(v-\tv^X)\partial_x \tu^X}{v \tv^X} +\frac{(v^{5/2}-(\tv^X)^{5/2})\partial_x \tw^X}{v^{5/2}(\tv^X)^{5/2}} \right) \\
		&\hspace{2cm}-\partial_x(w-\tw^X)\frac{(v^{5/2}-(\tv^X)^{5/2})\partial_x \tu^X}{v^{5/2}(\tv^X)^{5/2}}  \Bigg) dx,\\
		\mathcal{J}^{\textup{good}}(U)&:= \frac{\sigma}{2}\int_\mathbb{R} a^X_x |u-\tu^X|^2 \, d x+\frac{\sigma}{2}\int_\mathbb{R} a_x^X |w-\tw^X|^2 \, dx+\sigma \int_\mathbb{R} a^X_x Q(v|\tv^X) \, dx\\
		&\quad+ \int_\mathbb{R} \frac{a^X}{v} |\partial_x (u-\tu^X)|^2 \,dx.
		\end{align*}
	\end{lemma}
	\begin{proof}
		Since the system \eqref{eq:NS-abs} is written in a general hyperbolic system, we may use the same computations as in \cite[Lemma 2.3]{KV21} to estimate the time derivative of the weighted relative entropy as
		\begin{align*}
		\frac{d}{dt} \int_\mathbb{R} a^X \eta ( U | \tU^X) \, dx &=\dot{X}(t)Y(U)-\sigma\int_\mathbb{R} a_x^X\eta(U|\tU^X) \, dx+ \sum_{i=1}^5 I_{1i},
		\end{align*}
		where
		
		\begin{align*}
		I_{11}&:=-\int_\R a^X \partial_x G(U;\tU^X) \, dx,\\
		I_{12}&:=-\int_\R a^X \partial_x D \eta (\tU^X) A(U| \tU^X) \, dx,\\
		I_{13}&:=\int_\R a^X \left(D\eta(U)-D\eta(\tU^X) \right) \partial_x \left( M(U)\partial_x \left(D\eta (U)-D\eta(\tU^X) \right)\right) \, dx,\\
		I_{14}&:=\int_\R a^X \left( D \eta(U)-D\eta(\tU^X)\right)\partial_x \left( (M(U) -M(\tU^X)) \partial_x D \eta(\tU^X)\right)\,dx,\\
		I_{15}&:=\int_\R a^X (D\eta)(U|\tU^X)\partial_x \left( M(\tU^X) \partial_x D \eta(\tU^X) \right)\, dx.
		\end{align*}
		Then, we substitute the exact expressions of $\eta$, $A$ and $G$ to represent $I_{1i}$ in terms of $(u,v,w)$ as
		
		\begin{align*}
		I_{11}&=\int_\R a_x^X (p(v)-p(\tv^X) )(u-\tu^X)\, dx, \quad I_{12}=-\int_\R a^X \tu_x^X p(v|\tv^X) \, dx, \\
		I_{13}&=\int_\R a^X \left( (u-\tu^X) \partial_x \left( \frac{\partial_x (u-\tu^X)}{v}+\frac{\partial_x (w-\tw^X) }{v^{5/2}} \right) +(w-\tw^X)\partial_x \left( -\frac{\partial_x(u-\tu^X)}{v^{5/2}} \right) \right) dx\\
		&=-\int_\R a^X \frac{|\partial_x(u-\tu^X)|^2}{v} \, dx \\
		&\quad -\int_\R a_x^X \left( (u-\tu^X) \frac{\partial_x(u-\tu^X)}{v}+(u-\tu^X)\frac{\partial_x(w-\tw^X)}{v^{5/2}}-(w-\tw^X)\frac{\partial_x(u-\tu^X)}{v^{5/2}}\right) dx,\\
		I_{14}&=\int_\R a^X \Bigg( (u-\tu^X) \partial_x \left( \left(\frac{1}{v}-\frac{1}{\tv^X} \right)\partial_x \tu^X+ \left( \frac{1}{v^{5/2}}-\frac{1}{(\tv^X)^{5/2}}\right) \partial_x \tw^X \right)\\
		&\hspace{2cm} -(w-\tw^X) \partial_x \left( \left(\frac{1}{v^{5/2}}-\frac{1}{(\tv^X)^{5/2}} \right) \partial_x \tu^X \right) \Bigg) \,dx \\
		&=\int_\R a^X_x\Bigg( \frac{(v-\tv^X)(u-\tu^X)\partial_x \tu^X}{v \tv^X}+\frac{(v^{5/2}-(\tv^X)^{5/2})(u-\tu^X)\partial_x \tw^X}{v^{5/2}(\tv^X)^{5/2}}\\
		&\hspace{3cm}-\frac{(v^{5/2}-(\tv^X)^{5/2})(w-\tw^X)\partial_x \tu^X}{v^{5/2}(\tv^X)^{5/2}} \Bigg) \, dx \\
		&\quad + \int_\R a^X \Bigg(\partial_x (u-\tu^X) \left( \frac{(v-\tv^X)\partial_x \tu^X}{v \tv^X}+\frac{(v^{5/2}-(\tv^X)^{5/2})\partial_x \tw^X}{v^{5/2} (\tv^X)^{5/2}}\right) \\
		&\hspace{3cm}-\partial_x (w-\tw^X) \frac{(v^{5/2}-(\tv^X)^{5/2})\partial_x \tu^X}{v^{5/2}(\tv^X)^{5/2}}\Bigg) \, dx,
		\end{align*}
		and $I_{15}=0$. Combining all the estimates on $I_{1i}$, we obtain the desired estimate.
	\end{proof}

	\subsection{Maximization on $p(v)-p(\tv^X)$} 
	Among the terms in $\mathcal{J}^\text{bad}$, a primary bad term is 
	\[\int_\mathbb{R} a^X_x (p(v)-p(\tv^X))(u-\tu^X) \, d x\]
	where the perturbations for $p(v)$ and $u$ are coupled. In order to exploit the parabolic term on $u$-variable and hence use the Poincar\'e-type inequality, we separate $u-\tu^X$ from $p(v)-p(\tv^X)$ by using the quadratic structure of $p(v)-p(\tv^X)$. We first obtain the following estimates on several terms in $\mathcal{J}^{\textup{bad}}(U)$ and $\mathcal{J}^{\textup{good}}(U)$. 

	\begin{lemma}\label{lem:quad}
		There exists a positive constant $C_*$ such that 
		\begin{align}
		\begin{aligned}\label{quadratic estimate}
		-\int_{\mathbb{R}}& a^X \tu^X_x p(v|\tv^X) \, dx-\sigma \int_{\mathbb{R}} a^X_x Q(v|\tv^X) \, dx\\
		&\le -C_*\int_{\R}a_x^X|p(v)-p(\tv^X)|^2\,dx\\
		&\quad +C\delta_S \int_\R a_x^X\big|p(v)-p(\tv^X)\big|^2 \, d x+C\int_\R a_x^X\big|p(v)-p(\tv^X)\big|^3 \, d x.
		\end{aligned}
		\end{align}
	\end{lemma}
	\begin{proof}
		Define $I_{21}$ and $I_{22}$ as
		\begin{align*}
		I_{21}&:=-\int_{\mathbb{R}} a^X \tu^X_x p(v|\tv^X) \, dx =\int_{\mathbb{R}}  a^X \sqrt{\delta_S} a^X_x p(v | \tv^X) \,dx, \\
		I_{22}&:=\sigma \int_{\mathbb{R}} a^X_x Q(v|\tv^X) \, dx,
		\end{align*}
		where we use \eqref{a_x} in rewriting $I_{21}$. We use the estimate on $p(v|\tv)$ and $Q(v|\tv)$ in Lemma \ref{lem : Estimate-relative} to obtain upper and lower bounds of $I_{21}$ and $I_{22}$ respectively:
		\begin{align*}
		I_{21} &\le \int_\mathbb{R} a^X a^X_x (1+\sqrt{\delta_S})\sqrt{\delta_S}\left( \frac{\gamma+1}{2 \gamma } \frac{1}{p(\tv^X)} +C|p(v)-p(\tv^X)| \right) |p(v)-p(\tv^X)|^2 \, d x, \\
		I_{22} &\ge \sigma \int_\mathbb{R} a^X_x \left (\frac{p(\tv^X)^{-\frac{1}{\gamma}-1}}{2 \gamma} |p(v)-p(\tv^X)|^2 -\frac{1+\gamma}{3 \gamma^2} p(\tv^X)^{-\frac{1}{\gamma}-2} (p(v)-p(\tv^X))^3 \right) \, dx.
		\end{align*}
		
		Then using \eqref{shock_speed_est} and \eqref{shock_speed_est-2}, we have
		\begin{align}
		\begin{aligned}\label{I21I22}
		I_{21} -I_{22} &\le - \frac{1}{2} \left(\frac{1}{\sigma_\ell}- (\sqrt{\delta_S}+\delta_S)\frac{\gamma+1}{ \gamma } \frac{1}{p(v_-)}\right)  \int_\R a^X_x |p(v)-p(\tv^X)|^2 \, d x \\
		&\quad +C\delta_S \int_\R a_x^X\big|p(v)-p(\tv^X)\big|^2 \, d x+C\int_\R a_x^X\big|p(v)-p(\tv^X)\big|^3 \, d x.
		\end{aligned}
		\end{align}
		Therefore we may choose a positive constant $C_*$ as 
		\begin{equation}\label{C_star}
		C_* = \frac{1}{2}\left(\frac{1}{\sigma_{\ell}}-(\sqrt{\delta_S}+\delta_S)\frac{\gamma+1}{\gamma}\frac{1}{p(v_-)}\right).
		\end{equation}
		Then it follows from \eqref{I21I22} that
		\begin{align*}
		I_{21}-I_{22}&\le -C_*\int_{\R}a_x^X|p(v)-p(\tv^X)|^2\,dx\\
		&\quad +C\delta_S \int_\R a_x^X\big|p(v)-p(\tv^X)\big|^2 \, d x+C\int_\R a_x^X\big|p(v)-p(\tv^X)\big|^3 \, d x
		\end{align*} 
		which is the desired estimate \eqref{quadratic estimate}.
	\end{proof}
	Thanks to Lemma \ref{lem:quad}, we extract the quadratic structure of $p(v)-p(\tv)$ from $\mathcal{J}^{\textup{bad}}(U)$ and $\mathcal{J}^{\textup{good}}(U)$ as follows:
	\begin{align*}
	&\int_{\R}a^X_x (p(v)-p(\tv^X))(u-\widetilde{u}^X)\,dx-\int_{\R}a^X\tu_x^Xp(v|\tv^X)\,dx-\sigma\int_{\R}a_x^XQ(v|\tv^X)\,dx\\
	&\le  \int_{\R} a^X_x (p(v)-p(\tv^X))(u-\widetilde{u}^X)\,dx - C_* \int_{\R} a^X_x |p(v)-p(\tv^X)|^2\, dx \\ 
	& \quad +C\delta_S \int_\R a_x^X\big|p(v)-p(\tv^X)\big|^2 \, d x+C\int_\R a_x^X\big|p(v)-p(\tv^X)\big|^3 \, d x \\
	&= \int_{\R} a^X_x \left[ -C_*  \left( (p(v) -p(\tv^X))^2-\dfrac{(p(v) -p(\tv^X))(u-\widetilde{u}^X)}{C_*} + \dfrac{(u-\widetilde{u}^X)^2}{4 C_*^2} \right) + \dfrac{(u-\widetilde{u}^X)^2}{4C_*}  \right]\,dx \\
	& \quad +C\delta_S \int_\R a_x^X\big|p(v)-p(\tv^X)\big|^2 \, d x+C\int_\R a_x^X\big|p(v)-p(\tv^X)\big|^3 \, d x \\
	&=  -C_*\int_{\R} a^X_x \left| (p(v) -p(\tv^X))- \frac{u-\widetilde{u}^X}{2 C^*} \right|^2 \,dx+ \frac{1}{4C_*} \int_{\R}a^X_x |u-\tu^X|^2\,dx \\ 
	& \quad +C\delta_S \int_\R a_x^X\big|p(v)-p(\tv^X)\big|^2 \, d x+C\int_\R a_x^X\big|p(v)-p(\tv^X)\big|^3 \, d x.
	\end{align*}
	Then using the above estimate, we derive an upper bound of $\mathcal{J}^{\text{bad}}(U)-\mathcal{J}^{\text{good}}(U)$ as
	\[\mathcal{J}^{\text{bad}}(U) - \mathcal{J}^{\text{good}}(U) \le \mathcal{B}(U)-\mathcal{G}(U),\]
	where
	\begin{align*}
	\mathcal{B}(U)&:=\frac{1}{4 C_*}\int_\mathbb{R} a^X_x |u-\tu^X|^2 \, dx+C\delta_S \int_\R a_x^X\big|p(v)-p(\tv^X)\big|^2 \, d x+C\int_\R a_x^X\big|p(v)-p(\tv^X)\big|^3 \, d x \\
	&\quad-\int_\mathbb{R} a^X_x\left( \frac{u-\tu^X}{v} \partial_x (u-\tu^X)+\frac{(u-\tu^X)}{v^{5/2}} \partial_x (w-\tw^X)-\frac{(w-\tw^X)}{v^{5/2}} \partial_x (u-\tu^X) \right) dx\\
	&\quad+\int_\mathbb{R} a^X_x \Bigg(  \frac{(u-\tu^X)(v-\tv^X) \partial_x \tu^X}{v \tv^X} +\frac{(v^{5/2}-(\tv^X)^{5/2})(u-\tu^X)\partial_x \tw^X}{v^{5/2}(\tv^X)^{5/2}} \\
	&\hspace{6cm}-\frac{(v^{5/2}-(\tv^X)^{5/2})(w-\tw^X)\partial_x \tu^X}{v^{5/2}(\tv^X)^{5/2}} \Bigg) dx\\
	&\quad+\int_\mathbb{R} a^X \Bigg( \partial_x (u-\tu^X) \left( \frac{(v-\tv^X)\partial_x \tu^X}{v \tv^X} +\frac{(v^{5/2}-(\tv^X)^{5/2})\partial_x \tw^X}{v^{5/2}(\tv^X)^{5/2}} \right)\\
	&\hspace{5cm} -\partial_x(w-\tw^X)\frac{(v^{5/2}-(\tv^X)^{5/2})\partial_x \tu^X}{v^{5/2}(\tv^X)^{5/2}} \Bigg) dx,\\ 
	\mathcal{G}(U)&:=C_* \int_\mathbb{R} a^X_x \left| p(v)-p(\tv^X)-\frac{u-\tu^X}{2C_*}\right|^2 \, d x+ \frac{\sigma}{2} \int_\mathbb{R} a^X_x |u-\tu^X|^2 \, dx\\
	&\quad  +\frac{\sigma}{2}\int_\mathbb{R} a_x^X |w-\tw^X|^2 \, dx+ \int_\mathbb{R} \frac{a^X}{v} |\partial_x (u-\tu^X)|^2 \, dx.
	\end{align*} 
	Therefore, the estimate \eqref{est-weight-rel-ent} in Lemma \ref{lem:rel-ent} can be further bounded as
	\begin{equation}\label{est-1}
	\frac{d}{dt}\int_{\bbr}a^X\eta(U|\widetilde{U}^X)\,dx \le \dot{X}(t) Y(U) + \mathcal{B}(U)-\mathcal{G}(U).
	\end{equation}
	We now decompose the bad terms $\mathcal{B}(U)$ and the good terms $\mathcal{G}(U)$ as
	\begin{align*}
	&\mathcal{B}(U) =\sum_{i=1}^6 \mathcal{B}_i(U) + \sum_{i=1}^6 \mathcal{K}_i(U),\\
	&\mathcal{G}(U) = \mathcal{G}_1(U) + \mathcal{G}_2(U) + \mathcal{G}_3(U)+ \mathcal{D}(U).
	\end{align*}
	Here the terms $\mathcal{B}_i(U)$ are defined as
	\begin{align*}
	&\mathcal{B}_1(U) :=\frac{1}{4 C_*} \int_\mathbb{R} a^X_x  |u-\tu^X|^2 \, dx, & &\mathcal{B}_2(U) :=-\int_\mathbb{R} a^X_x \frac{u-\tu^X}{v} \partial_x (u-\tu^X) \, dx, \\ 
	&\mathcal{B}_3(U) := \int_\mathbb{R} a^X_x(u-\tu^X)(v-\tv^X) \frac{\partial_x \tu^X}{v \tv^X} \, dx, & &\mathcal{B}_4(U) :=\int_\mathbb{R} a^X \partial_x (u-\tu^X)  \frac{v-\tv^X}{v \tv^X} \partial_x \tu^X \, dx, \\
	&\mathcal{B}_5(U) :=C \delta_S \int_\R a_x^X\big|p(v)-p(\tv^X)\big|^2 \, d x, & &\mathcal{B}_6(U) :=C\int_\R a_x^X\big|p(v)-p(\tv^X)\big|^3 \, d x.
	\end{align*}
	and the terms $\mathcal{K}_i(U)$ are given as
	\begin{align*}
	&\mathcal{K}_1(U)=-\int_{\R}a_x^X\frac{(u-\tu^X)}{v^{5/2}}\pa_x(w-\tw^X)\,dx,\\
	&\mathcal{K}_2(U)=\int_{\R}a_x^X\frac{w-\tw^X}{v^{5/2}}\pa_x(u-\tu^X)\,dx,\\
	&\mathcal{K}_3(U)=\int_{\R}a_x^X\frac{(v^{5/2}-(\tv^X)^{5/2})(u-\tu^X)\pa_x\tw^X}{v^{5/2}(\tv^X)^{5/2}},\\
	&\mathcal{K}_4(U)=-\int_{\R}a_x^X\frac{(v^{5/2}-(\tv^X)^{5/2})(w-\tw^X)\pa_x\tu^X}{v^{5/2}(\tv^X)^{5/2}}\,dx,\\
	&\mathcal{K}_5(U)=\int_{\R} a^X\pa_x(u-\tu^X)\frac{(v^{5/2}-(\tv^X)^{5/2})\pa_x\tw^X}{v^{5/2}(\tv^X)^{5/2}}\,dx,\\
	&\mathcal{K}_6(U)=-\int_{\R}a^X\pa_x(w-\tw^X)\frac{(v^{5/2}-(\tv^X)^{5/2})\pa_x\tu^X}{v^{5/2}(\tv^X)^{5/2}}\,dx.
	\end{align*}
	We note that the terms $\mathcal{K}_i(U)$ are the terms from the Korteweg force, compared to the classical NS equations. The decomposition of the good terms is also defined as follows:
	\begin{align*}
	&\mathcal{G}_1(U):=C_* \int_\mathbb{R} a^X_x \left|p(v)-p(\tv^X)-\frac{u-\tu^X}{2C_*} \right|^2 \, d x,\quad \mathcal{G}_2(U):=\frac{\sigma}{2} \int_\mathbb{R} a^X_x |u-\tu^X|^2  \, dx,\\
	&\mathcal{G}_3(U):=\frac{\sigma}{2} \int_\mathbb{R} a^X_x |w-\tw^X|^2 \, dx,\quad \mathcal{D}(U):= \int_\mathbb{R} \frac{a^X}{v} |\partial_x (u-\tu^X)|^2 \, dx.
	\end{align*}
	On the other hand, since $Y(U)$ is expanded as
	\begin{align*}
	Y(U)&=-\int_{\bbr} a^X_x \eta(U|\widetilde{U}^X)\,dx +\int_\bbr a^X D^2\eta(\widetilde{U}^X)(\widetilde{U})^X_x (U-\widetilde{U}^X)\,dx\\
	& = -\int_{\bbr} a^X_x \left(\frac{|u-\widetilde{u}^X|^2}{2}+\frac{|w-\tw^X|^2}{2}+Q(v|\widetilde{v}^X)\right)\,dx +\int_{\bbr} a^X \tu^X_x (u-\widetilde{u}^X)\,dx\\
	&\quad -\int_{\bbr} a^X p'(\widetilde{v}^X)\widetilde{v}^X_x (v-\widetilde{v}^X)\,dx+\int_\mathbb{R} a^X \tw^X_x(w-\tw^X) \, dx,
	\end{align*}
	we decompose $Y$ as
	\[Y= \sum_{i=1}^7 Y_{i},\]
	where
	\begin{align*}
	Y_{1} &:=\int_{\mathbb{R}} a^X \tu^X_x (u-\tu^X) \,dx,\quad Y_{2}:=\frac{1}{\sigma}\int_{\mathbb{R}} a^X p'(\tv^X)\tv^X_x (v-\tv^X) \,dx,\\
	Y_{3} &:=-\frac{1}{2}\int_{\mathbb{R}} a^X_x \Big( u-\tu^X -2C_* \big(p(v)-p(\tv^X) \big)\Big)\Big(u-\tu^X +2C_* \big(p(v)-p(\tv^X) \big)\Big) \, dx,\\
	Y_{4} &:=-\frac{1}{2} \int_{\mathbb{R}} a^X_x 4C_*^2 \big(p(v)-p(\tv^X) \big)^2 \,dx-\int_{\mathbb{R}} a^X_x Q(v|\tv^X)dx,\\
	Y_{5} &:=- \int_\mathbb{R} a^X p'(\tv^X) \tv^X_x \left(v-\tv^X+\frac{2C_*}{\sigma}(p(v)-p(\tv^X)) \right) \, d x,\\
	Y_{6} &:= \int a^X p'(\tv^X) \tv^X_x \frac{2C_*}{\sigma}\left(p(v)-p(\tv^X)-\frac{u-\tu^X}{2C_*}\right) \, d x,\\
	Y_{7} &:=-\int_\mathbb{R} a^X_x \frac{|w-\tw^X|^2}{2} \, dx +\int_\mathbb{R} a^X \tw^X_x (w-\tw^X) \, dx.
	\end{align*}
	We now define a shift function $X(t)$ so that it satisfies the following ODE:  
	\begin{equation}\label{shift}
	\dot{X} = -\frac{M}{\delta_S}(Y_{1}+Y_{2}),\quad X(0)=0.
	\end{equation}
	With this choice of shift $X$, the term $\dot{X}(t)Y(U)$ in \eqref{est-1} can be written as
	\[\dot{X}(t)Y(U)= -\frac{\delta_S}{M}|\dot{X}|^2+\dot{X}\sum_{i=3}^7Y_{i}.\]
	In conclusion, we decompose the right-hand side of \eqref{est-1} as
	\begin{align}
	\begin{aligned}\label{est}
	\frac{d}{dt}\int_{\bbr}a^X\eta(U|\widetilde{U}^X)\,dx &= \underbrace{-\frac{\delta_S}{2M}|\dot{X}|^2 + \mathcal{B}_1 -\mathcal{G}_2-\frac{3}{4}\mathcal{D}}_{=: \mathcal{R}_1}\\
	&\quad \underbrace{-\frac{\delta_S}{2M}|\dot{X}|^2 + \dot{X}\sum_{i=3}^7Y_{i} +\sum_{i=2}^6\mathcal{B}_i +\sum_{i=1}^6 \mathcal{K}_i-\mathcal{G}_1-\mathcal{G}_3-\frac{1}{4}\mathcal{D}}_{=:\mathcal{R}_2}.
	\end{aligned}
	\end{align}
	In the following subsections, we estimate the terms in $\mathcal{R}_1$ and $\mathcal{R}_2$ respectively.
	
	\subsection{Estimate of $\mathcal{R}_1$} \label{Est-main-part}
	Estimation of $\mathcal{R}_1$ is the most important part in the proof of Lemma \ref{Main Lemma}, in which the Poincar\'e-type inequality is crucially used. For a fixed $t\ge0$, we define an auxiliary variable $y$ as
	\begin{equation*} 
	y:=\frac{u_- -\tu(x-\sigma t-X(t))}{\delta_S}.
	\end{equation*}
	Then it follows from the definition that the map $x\mapsto y=y(x)$ is one-to-one and 
	\begin{equation*}
	\frac{d y}{d x} = -\frac{1}{\delta_S} \tu^X_x>0,\quad\mbox{and}\quad\lim_{x\to -\infty} y=0, \quad \lim_{x \to \infty}y=1.
	\end{equation*}
	Using the new variable $y$, we will apply the following Poincarè-type inequality:
	
	\begin{lemma}\cite[Lemma 2.9]{KV21}\label{lem: KV inequality} For any $f:[0,1] \to \R$ with $\int_0^1 y(1-y) |f'|^2 \, dy <\infty$, 
		\begin{equation*}
		\int_0^1 \left| f-\int_0^1 f \,dy \right|^2 dy \le \frac{1}{2} \int_0^1 y(1-y)|f'|^2 \,dy.
		\end{equation*}
	\end{lemma}
	
	We apply Lemma \ref{lem: KV inequality} to the perturbation $f$ of the following form:
	\begin{equation*} 
	f:=\left(u(t,\cdot)-\tu(\cdot-\sigma t-X(t)\right)\circ y^{-1},
	\end{equation*}
	that is $f(y)=u(t,x)-\tu(x-\sigma t-X(t))$. Therefore, the goal of this subsection is to represent $\mathcal{R}_1$ in terms of $f$ and then use Lemma \ref{lem: KV inequality} to estimate it. In the following, we estimate the terms in $\mathcal{R}_1$ separately.\\
	
	\noindent $\bullet$  (Estimate of $\frac{\delta_S}{2M}|\dot{X}|^2$): 
	We use the definition of $Y_1$ and change of variables for $y$ to observe that
	\[Y_1=\int_{\R} a^X \tu_x^X (u-\tu^X)\,dx
	=-\delta_S \int_0^1 a^X f \,dy.\]
	Then, using \eqref{shock_speed_est} and $\| a^X-1 \|_{L^\infty(\R_+ \times \R)} \le \sqrt{\delta_S}$, we have
	\[\left| Y_1+\delta_S \int_0^1 f \, dy \right|\le\delta_S\int_0^1|a^X-1||f| dy \le   \delta_S^{3/2} \int_0^1 |f| \, dy.\]
	To estimate $Y_2$, we first use the relation $\sigma \tv^X_x=-\tu^X_x $ and change of variables for $y$ to yield
	\[Y_2=-\frac{1}{\sigma^2} \int_{\mathbb{R}} a^X p'(\tv^X)\tu^X_x (u-\tu^X) \, d x=\frac{\delta_S}{\sigma^2} \int_0^1 a^X p'(\tv^X) f \, dy.\]
	This, together with the estimates \eqref{shock_speed_est}, \eqref{shock_speed_est-2} and $\| a^X-1 \|_{L^\infty(\R_+ \times \R)} \le \sqrt{\delta_S}$, implies 
	\begin{align*}
	\left| Y_2 + \delta_S \int_0^1 f \, dy \right|\le \delta_S\int_0^1\left|\frac{a^X p'(\tv^X)}{\sigma^2}+1\right||f|dy \le C \delta_S (\sqrt{\delta_S}+\delta_S) \int_0^1 |f| \, dy.
	\end{align*}
	Since $\dot{X}=-\frac{M}{\delta_S}(Y_1+Y_2)$, we combine the estimates for $Y_1$ and $Y_2$ to obtain
	\begin{align*}
	\left| \dot{X} - 2 M \int_0^1 f \, dy \right| \le \frac{M}{\delta_S} \left( \left| Y_1 +\delta_S \int_0^1 f \, dy \right|+\left| Y_2+\delta_S \int_0^1 f\, dy \right| \right)\le C (\sqrt{\delta_S}+ \delta_S) \int_0^1 |f| \, dy,
	\end{align*}
	which implies
	\begin{align*}
	\left( \left| 2M \int_0^1 f \, dy \right| -|\dot{X}| \right)^2 \le C(\sqrt{\delta_S}+\delta_S)^2 \int_0^1 |f|^2 \, dy\le C\delta_S\int_0^1|f|^2\,d y.
	\end{align*}
	We use an elementary inequality $\frac{p^2}{2}-q^2 \le (p-q)^2$ for $p,q \in \mathbb{R}$ to obtain 
	\begin{align*}
	2M^2\left(\int_0^1 f \, dy \right)^2-|\dot{X}|^2 \le C \delta_S \int_0^1 |f|^2 \, dy,
	\end{align*}
	and by rearranging the terms, we obtain
	\begin{equation}\label{est-dotX}
	-\frac{\delta_S}{2M} |\dot{X}|^2 \le -M \delta_S \left(\int_0^1 f \, dy \right)^2 +C \delta_S \int_0^1 |f|^2 \, dy.
	\end{equation}
	
	\noindent $\bullet$ (Estimates of $\mathcal{B}_1$ and  $\mathcal{G}_2$):
	Recall that $\mathcal{B}_1$ and $\mathcal{G}_2$ are 
	\begin{align*}
	&\mathcal{B}_1:=\frac{1}{4C_*}\int_\mathbb{R} a^X_x |u-\tu^X|^2 \, d x, \\
	&\mathcal{G}_2:=\frac{\sigma}{2} \int_\mathbb{R} a^X_x |u-\tu^X|^2 \, d x.
	\end{align*}
	Therefore,
	\begin{align*}
	\mathcal{B}_1-\mathcal{G}_2&=\left(\frac{1}{4C_*}-\frac{\sigma}{2}\right)\int_{\R}a_x^X|u-\tu^X|^2\,dx=-\left(\frac{1}{4C_*}-\frac{\sigma}{2}\right)\frac{1}{\sqrt{\delta_S}}\int_{\R}\tu_x^X|u-\tu^X|^2\,dx\\
	&=\sqrt{\delta_S}\left(\frac{1}{4C_*}-\frac{\sigma}{2}\right)\int_0^1 f^2\,dy.
	\end{align*}
	where $C_*$ defined in \eqref{C_star} can be written as
	\begin{equation*}
	C_*= \frac{1}{2 \sigma_\ell} -(\sqrt{\delta_S}+\delta_S)\alpha_\ell \sigma_\ell,
	\end{equation*}
	with simple notation $\alpha_\ell=\frac{\gamma+1}{2 \gamma \sigma_\ell p(v_-)}$ in \eqref{O(1) constnats}. On the other hand, using \eqref{shock_speed_est}, \eqref{shock_speed_est-2}, we obtain
	\begin{align*}
	\sqrt{\delta_S}\left(\frac{1}{4C_*}-\frac{\sigma}{2}\right)&\le \frac{\sigma_\ell}{2} \frac{\sqrt{\delta_S}}{1-2 (\sqrt{\delta_S}+\delta_S) \sigma_\ell^2 \alpha_\ell}- \frac{\sigma }{2} \sqrt{\delta_S} 
	\\
	&\le \frac{\sqrt{\delta_S}}{2} \left( \frac{1}{1-2 (\sqrt{\delta_S}+\delta_S) \sigma_\ell^2 \alpha_\ell} (\sigma_\ell -\sigma) +\sigma \left(\frac{1}{1-2 (\sqrt{\delta_S}+\delta_S) \sigma_\ell^2 \alpha_\ell}-1\right) \right) \\
	&\le C  \delta_S^{3/2} +\delta_S\sigma_\ell^3 \alpha_\ell.
	\end{align*}

	Therefore, we have
	\begin{equation}\label{R1.2}
	\mathcal{B}_1-\mathcal{G}_2 \le   C\delta_S^{3/2} \int_0^1 f^2 \, dy+  \sigma_\ell^3 \alpha_\ell \delta_S \int_0^1 f^2 \, dy. 
	\end{equation}

	\noindent $\bullet$ (Estimate of $\mathcal{D}(U)$):
	First, using $a \geq 1$ and change of variables, we estimate the diffusion term $\mathcal{D}$ in terms of $f$:
	\begin{equation*}
	\mathcal{D} \geq \int_\R \frac{1}{v}|\partial_x (u-\tu^X)|^2 dx=\int_0^1 |\partial_y f|^2 \frac{1}{v} \left(\frac{dy}{d x} \right) dy.
	\end{equation*}
	
	Similar to the NS system \cite[Lemma 4.5]{KVW23}, there exists $C>0$ such that the following estimate holds:
	\begin{equation}\label{Diffusion}
	\left|   \frac{1}{y(1-y)} \frac{1}{\tv^X} \left( \frac{dy}{dx} \right)-\frac{\sigma}{2 \sigma_\ell}\frac{\delta_S v''(p_-)}{ |v'(p_-)|^2 } \right| \le C \delta_S^2.
	\end{equation}
	We present the proof of the estimate \eqref{Diffusion} in Appendix \ref{Diffusion proof}. On the other hand, since $C^{-1} \leq v \leq C$, we have
	\begin{equation}\label{estimate in diffusion}
	\left| \frac{\tv^X}{v}-1 \right| \leq C \left|\tv^X-v \right| 
	\leq C \varepsilon_1.
	\end{equation}
	Then, using \eqref{Diffusion} and \eqref{estimate in diffusion}, we obtain the lower bound for $\mathcal{D}$ as
	\begin{align*}
	\mathcal{D} &\geq \int_0^1 |\partial_y f|^2 \frac{\tv^X}{v} \frac{1}{\tv^X} \left(\frac{dy}{dx} \right) dy \\
	&\geq (1-C \varepsilon_1) \left(\frac{\sigma}{2 \sigma_\ell}\frac{\delta_S v''(p_-)}{ |v'(p_-)|^2 }-C \delta_S^2 \right) \int_0^1 y(1-y) \left| \partial_y f \right|^2 dy.
	\end{align*}
	Finally, since
	\begin{align*}
	\sigma_\ell^3 \alpha_\ell=\frac{1}{2} (1+\gamma) \frac{1}{v_-}=\frac{1}{2 }\frac{v''(p_-)}{|v'(p_-)|^2 },
	\end{align*}
	we obtain
	\begin{equation} \label{R1.3}
	\mathcal{D} \ge \sigma_\ell^3 \alpha_\ell \delta_S (1-C (\delta_0 + \varepsilon_1)) \int_0^1 y(1-y) |\partial_y f|^2 \, dy.
	\end{equation}
	
	\noindent $\bullet$ (Estimate of $\mathcal{R}_1$): We now combine the estimates \eqref{R1.2} and \eqref{R1.3}, we have
	\begin{align*}
	\mathcal{B}_1-\mathcal{G}_2-\frac{3}{4}\mathcal{D} \le \sigma_\ell^3 \alpha_\ell \delta_S \left(   (1+C \sqrt{\delta_S})\int_0^1 f^2 \, dy - \frac{3}{4}(1-C(\delta_0+\varepsilon_1)) \int_0^1 y(1-y) |\partial_y f|^2 \, dy \right),
	\end{align*}
	which together with the smallness of $\delta_0,\varepsilon_1$ yields
	\[\mathcal{B}_1-\mathcal{G}_2-\frac{3}{4}\mathcal{D} \le \sigma_\ell^3 \alpha_\ell \delta_S \left(   \frac{9}{8} \int_0^1 f^2 \, dy -\frac{5}{8}  \int_0^1 y(1-y) |\partial_y f|^2 \, dy \right). \] 
	Then, using Lemma \ref{lem: KV inequality}  with the identity
	\[\int_0^1 |f-\overline{f}|^2 \, dy=\int_0^1 f^2 \, dy-\overline{f}^2, \quad \overline{f}:=\int_0^1 f \, dy,\]
	we have
	\begin{align*}
	\mathcal{B}_1-\mathcal{G}_2-\frac{3}{4}\mathcal{D} \le -\frac{\sigma_\ell^3 \alpha_\ell \delta_S}{8} \int_0^1 f^2 \, dy + \frac{5 \sigma_\ell^3 \alpha_\ell \delta_S}{4} \left(\int_0^1 f\,dy\right)^2.
	\end{align*}
	Finally, using \eqref{est-dotX} with the choice $M=\dfrac{5 \sigma_\ell^3 \alpha_\ell}{4}$, we have
	\begin{align*}
	&-\frac{\delta_S}{2M}|\dot{X}|^2+\mathcal{B}_1-\mathcal{G}_2-\frac{3}{4}\mathcal{D} \le  -\frac{\sigma_\ell^3 \alpha_\ell \delta_S}{16} \int_0^1 f^2 \, dy,
	\end{align*}
	which implies 
	\begin{equation} \label{R1}
	\begin{aligned}
	\mathcal{R}_1 \le -C_1 \int_\mathbb{R} |\tv^X_x| |u-\tu^X|^2 \, dx=:-C_1G^S.
	\end{aligned}
	\end{equation}
	\subsection{Estimate of remaining terms} \label{Est-Remaining-terms}
	We now estimate the remaining terms in $\mathcal{R}_2$. We first substitute the estimate of $\mathcal{R}_1$ \eqref{R1} to \eqref{est} and use the Young's inequality
	\begin{equation*}
	\dot{X}\sum_{i=3}^7 Y_i \le \frac{\delta_S}{4M} |\dot{X}|^2 +\frac{C}{\delta_S}\sum_{i=3}^7 |Y_i|^2
	\end{equation*}
	to have
	\begin{equation}\label{est-R}
	\begin{aligned}{}
	\frac{d}{dt} \int_\mathbb{R} a \eta(U|\tU) \, dx \le -C_1 G^S -\frac{\delta_S}{4M} |\dot{X}|^2 +\frac{C}{\delta_S}\sum_{i=3}^7 |Y_i|^2 + \sum_{i=2}^6 \mathcal{B}_i + \sum_{i=1}^6 \mathcal{K}_i-\mathcal{G}_1 -\mathcal{G}_3 -\frac{1}{4} \mathcal{D}.
	\end{aligned}
	\end{equation}
	Therefore, to close the estimate of the weighted relative entropy, it suffices to control the remaining terms $|Y_i|^2$, $\mathcal{B}_i$, and $\mathcal{K}_i$. Below, we estimate each term separately.\\
	

	\noindent $\bullet$ (Estimate of $\frac{C}{\delta_S} |Y_i|^2$ for $i=3,\ldots,7$): We use Cauchy-Schwarz inequality to estimate $Y_3$ as
	\begin{align*}
	|Y_{3}|^2 &= \left|-\frac{1}{2}\int_{\mathbb{R}} a^X_x \Big( u-\tu^X -2C_* \big(p(v)-p(\tv^X) \big)\Big)\Big(u-\tu^X +2C_* \big(p(v)-p(\tv^X) \big)\Big) \, dx \right|^2 \\
	&\le C \left( \int_\mathbb{R} a^X_x \left|p(v)-p(\tv^X)-\frac{u-\tu^X}{2C_*} \right|^2 \, dx \right) \left( \int_\mathbb{R} a^X_x \left|p(v)-p(\tv^X)+\frac{u-\tu^X}{2C_*} \right|^2 \,dx \right) \\
	&\le C \mathcal{G}_1 \|a^X_x\|_{L^\infty}  \left( \|u-\tu^X\|_{L^\infty(0,T;L^2(\mathbb{R}))}^2 + \|v-\tv^X\|_{L^\infty(0,T;L^2(\mathbb{R}))}^2  \right) \\
	&\le C \delta^{3/2}_S \varepsilon_1^2 \mathcal{G}_1,
	\end{align*}
	where we used $\|\tv^X_x\|_{L^\infty}\le C\delta_S^2$ to obtain $\|a^X_x\|_{L^\infty}\le \frac{1}{\sqrt{\delta_S}}\|\tv^X_x\|_{L^\infty}\le C\delta_S^{3/2}$. This yields 
	\begin{equation*} 
	\frac{C}{\delta_S}|Y_3|^2 \le C \sqrt{\delta_S} \varepsilon_1^2 \mathcal{G}_1.
	\end{equation*}
	On the other hand, using $Q(v|\tv^X) \le C |p(v)-p(\tv^X)|$, we estimate $Y_4$ as 
	\begin{align*}
	|Y_{4}| & \le C \int_\mathbb{R} |a_x^X| |p(v)-p(\tv^X)|^2 \, dx,  
	\end{align*}
	which implies
	\begin{align*}
	\frac{C}{\delta_S } |Y_4|^2 &\le \frac{C}{\delta_S} \left( \int_\mathbb{R} |a^X_x| |p(v)-p(\tv^X)|^2  \, dx \right)^2 \\
	&\le \frac{C}{\delta_S} \left( \int_\mathbb{R} |a^X_x| \left|p(v)-p(\tv^X)-\frac{u-\tu^X}{2C_*}\right|^2 \, dx+\int_\mathbb{R} |a^X_x| |u-\tu^X|^2 \, dx \right)^2 \\
	&\le \frac{C}{\delta_S} \left( \int_\mathbb{R} |a^X_x| \left|p(v)-p(\tv^X)-\frac{u-\tu^X}{2C_*}\right|^2 \,dx \right)^2+ \frac{C}{\delta_S} \left(\int_\mathbb{R} |a^X_x| |u-\tu^X|^2 \, dx\right)^2 .
	\end{align*}
	Then, using the estimates $\|a^X_x\|_{L^\infty (\R)} \le C \delta^{3/2}_S$  and $\|v-\tv^X\|_{L^2(\R)}+\|u-\tu^X\|_{L^2(\R)} \le C \varepsilon_1$, we derive the following estimate for $Y_4$:
	\begin{align*}
	\begin{aligned} 
	\frac{C}{\delta_S} |Y_4|^2 &\le \frac{C}{\delta_S} \|a^X_x\|_{L^\infty(\mathbb{R})} \left(\|v-\tv^X \|_{L^2(\mathbb{R})}^2+\|u-\tu^X \|_{L^2(\mathbb{R})}^2 \right) \int_\mathbb{R} |a^X_x| \left|p(v)-p(\tv^X)-\frac{u-\tu^X}{2C_*}\right|^2 \,dx \\
	&\quad + \frac{C}{\delta_S^{3/2}} \|a^X_x\|_{L^\infty(\mathbb{R})} \|u-\tu^X\|_{L^2(\R)}^2 \int_\mathbb{R} |\tv^X_x| |u-\tu^X|^2 \, dx \\
	&\le  C \sqrt{\delta_S} \varepsilon_1^2 \mathcal{G}_1 + C \varepsilon_1^2 G^S.
	\end{aligned}
	\end{align*}
	To estimate $Y_5$, we first note that
	\begin{align*}
	Y_{5} &= - \int_\mathbb{R} a^Xp'(\tv^X)\tv^X_x \left( v-\tv^X+\frac{p(v)-p(\tv^X)}{\sigma\sigma_l}-\frac{2}{\sigma}\alpha_l\sigma_l(\sqrt{\delta_S}+\delta_S)(p(v)-p(\tv^X)) \right) \, dx. \\ 
	\end{align*}
	Since $v=p(v)^{-\frac{1}{\gamma}}$, we use Taylor expansion of the map $z\mapsto z^{-\frac{1}{\gamma}}$ to observe 
	\[\left| v-\widetilde{v}^X+\frac{p(v)-p(\widetilde{v}^X)}{\gamma p^{1+\frac{1}{\gamma}}(\widetilde{v}^X)}\right|\le C|p(v)-p(\widetilde{v}^X)|^2. \]
	This, together with \eqref{shock_speed_est} and \eqref{shock_speed_est-2}, implies
	\[\left| v-\widetilde{v}^X+\frac{p(v)-p(\widetilde{v}^X)}{\sigma\sigma_l}\right|\le C(\varepsilon_1+\delta_S)|p(v)-p(\widetilde{v}^X)|. \]
	Therefore, we estimate $Y_5$ as 
	\begin{align*}
	|Y_{5}| &\le C \left( \int_\mathbb{R} |\tv^X_x| \left| v-\widetilde{v}^X+\frac{p(v)-p(\widetilde{v}^X)}{\sigma\sigma_l} \right| \, dx + \sqrt{\delta_S}\int_\mathbb{R} |\tv^X_x| |p(v)-p(\tv^X)|  \, dx \right) \\
	&\le C (\varepsilon_1+\delta_S+\sqrt{\delta_S}) \int_\mathbb{R} |\tv^X_x| \left|p(v)-p(\widetilde{v}^X)\right| \,dx \\ 
	&\le C (\varepsilon_1+\sqrt{\delta_S}) \left(\int_\mathbb{R} |\tv^X_x| \left|p(v)-p(\widetilde{v}^X)-\frac{u-\widetilde{u}^X}{2C_\ast}\right| \, dx+\int_\mathbb{R} |\tv^X_x| \left|u-\widetilde{u}^X\right| \, dx\right) \\ 
	&\le C (\varepsilon_1+\sqrt{\delta_S}) \sqrt{\int_\mathbb{R} |\widetilde{v}^X_x| dx} \left( \sqrt{\int_{\mathbb{R}} |a^X_x| \left|p(v)-p(\tv^X)-\frac{u-\widetilde{u}^X}{2C_\ast}\right|^2  \, dx}+\sqrt{\int_{\mathbb{R}} |\widetilde{v}^X_x| \left|u-\widetilde{u}^X\right|^2  \, dx}\right) \\
	&\le C(\varepsilon_1+\sqrt{\delta_S})\sqrt{\delta_S}(\sqrt{\mathcal{G}_1} + \sqrt{G^S}),
	\end{align*}
	which yields
	\begin{align*}
	\frac{C}{\delta_S} |Y_5|^2 & \le  C (\varepsilon_1+\sqrt{\delta_S})^2 (\mathcal{G}_1 + G^S).
	\end{align*}
	For $Y_6$, we use H\"older inequality as
	\begin{align*}
	|Y_{6}| & \le C \int_\mathbb{R}  |\tv^X_x| \left| p(v)-p(\tv^X)-\frac{u-\tu^X}{2C_*} \right| \, d x\\
	&\le C \sqrt{\int_\mathbb{R} |\tv^X_x| \, dx} \sqrt{\int_\mathbb{R} |\tv^X_x| \left| p(v)-p(\tv^X)-\frac{u-\tu^X}{2C_*} \right|^2 \, dx},
	\end{align*}
	which gives
	\begin{align*}
	\frac{C}{\delta_S} |Y_6|^2 \le \frac{C}{\delta_S} \int_\mathbb{R} |\tv^X_x| \, dx \int_\mathbb{R} |\tv^X_x| \left| p(v)-p(\tv^X)-\frac{u-\tu^X}{2C_*} \right|^2 \, dx \le C \sqrt{\delta_S} \mathcal{G}_1.
	\end{align*}
	Finally, for $Y_7$, we use the relations
	\[|\tw^X_x| \le C \left( |\tv^X_{xx}| + |\tv^X_x| \right) \le C |\tv^X_x|,\quad |\tv^X_x| \sim \sqrt{\delta_S} |a^X_x|\]
	and H\"older's inequality to obtain
	\begin{align*}
	|Y_{7}| &\le C \left( \int_\mathbb{R} |a^X_x| |w-\tw^X|^2\,dx +  \int_\mathbb{R} |\tv^X_x| |w-\tw^X| \, dx \right) \\
	&\le C \left(  \int_\mathbb{R} |a^X_x| |w-\tw^X|^2\,dx +  \sqrt{\int_\mathbb{R} |\tv^X_x| \, dx} \sqrt{\int_\mathbb{R} |\tv^X_x| |w-\tw^X|^2 \, dx} \,  \right)\\
	&\le C \left(  \int_\mathbb{R} |a^X_x| |w-\tw^X|^2\,dx +  \sqrt{\int_\mathbb{R} |\tv^X_x| \, dx} \sqrt{ \sqrt{\delta_S} \int_\mathbb{R} |a^X_x| |w-\tw^X|^2 \, dx} \,  \right), 
	\end{align*}
	which gives 
	\begin{align*}
	\frac{C}{\delta_S} |Y_7|^2 \le \frac{C}{\delta_S} \left( \| a^X_x \|_{L^\infty(\R)} \|w-\tw^X\|_{L^2(\R)}  + \delta_S^{3/2} \right) \mathcal{G}_3 \le C \sqrt{\delta_S} \mathcal{G}_3 , 
	\end{align*}
	where the last inequality, we used $\| a^X_x \|_{L^\infty(\R)}<C\delta_S^{\frac{3}{2}}$ and $\|w-\tw^X\|_{L^2(\R)}^2 \le C \|(v-\tv^X)_x\|_{L^2(\R)}^2 \le C \varepsilon_1^2$. 
	
	Combining all the estimates of $Y_i$ with $i=3,4,\ldots,7$, and use the smallness of the parameters, we conclude 
	that
	\begin{equation}\label{est-Y}
	\begin{aligned}
	\frac{C}{\delta_S}\sum_{i=3}^7 |Y_i|^2 &\le C (\sqrt{\delta_S}+\e_1^2)\mathcal{G}_1 + C \sqrt{\delta_S} \mathcal{G}_3 + C  \left( \varepsilon_1^2 +\delta_S^2 \right)G^S \\
	&\le \frac{1}{100} (\mathcal{G}_1 + \mathcal{G}_3 + C_1 G^S),
	\end{aligned}
	\end{equation}
	where the coefficient $\frac{1}{100}$ is chosen as a small coefficient, but still order of $O(1)$.\\
	
	\noindent $\bullet$ (Estimate of  $\mathcal{B}_i(U)$ for $i=2,3,4,5,6$): We estimate $\mathcal{B}_2$ by using Young's inequality, $\|a^X_x\|_{L^\infty(\R)} \le C \delta^{3/2}_S$ and $|a^X_x| \sim \frac{|\tv^X_x|}{\sqrt{\delta_S}} $ as
	\begin{align*}
	|\mathcal{B}_2(U)| &\le C \int_\mathbb{R} |a_x^X| |u-\tu^X| |(u-\tu^X)_x| \, dx,\\
	&\le \frac{1}{100} \mathcal{D} + C \int_\mathbb{R} |a_x^X|^2 |u-\tu^X|^2 \, dx\\
	&\le \frac{1}{100} \mathcal{D} + C \delta_S \int_\mathbb{R} | \tv^X_x| |u-\tu^X|^2 \, dx \le \frac{1}{100} ( \mathcal{D}+C_1 G^S).
	\end{align*}
	To estimate $\mathcal{B}_3$, we use Young's inequality, $|\tu^X_x|\le C\delta_S^2$ and $|a^X_x| \sim \frac{|\tv^X_x|}{\sqrt{\delta_S}} $ to obtain
	\begin{align*}
	|\mathcal{B}_3(U)|  &\le C \delta_S^2 \int_\mathbb{R} |a^X_x| |u-\tu^X| |p(v)-p(\tv^X)| \, dx\\
	&\le C \delta_S^2 \left( \int_\mathbb{R} |a^X_x|  \left| p(v)-p(\tv^X)\right|^2 \, dx +  \int_\mathbb{R} |a^X_x| \left| u-\tu^X \right|^2  \, dx \right) \\
	&\le C \delta_S^2 \left( \int_\mathbb{R} |a^X_x|  \left| p(v)-p(\tv^X)-\frac{u-\tu^X}{2C_*} \right|^2 \, dx +  \int_\mathbb{R} |a^X_x| \left| u-\tu^X \right|^2  \, dx \right) \\
	&\le C \delta_S^2 \mathcal{G}_1 + C \delta^{3/2}_S G^S \le \frac{1}{100} ( \mathcal{G}_1+C_1 G^S).
	\end{align*}
	We use $|\tu^X_x| \sim |\tv^X_x|$, Young's inequality, $\|\tv^X_x\|_{L^\infty(\R)} \le C \delta_S^2$, and $|a^X_x| \sim \frac{|\tv^X_x|}{\sqrt{\delta_S}}$ to estimate $\mathcal{B}_4$ as
	\begin{align*}
	|\mathcal{B}_4(U)| &\le C \int_{\mathbb{R}}   |\tv^X_x| |(u-\tu^X)_x| |p(v)-p(\tv^X)| \, dx\\
	&\le C \int_\mathbb{R} |\tv^X_x|^{3/2} |p(v)-p(\tv^X)|^2 \, dx + \int_\mathbb{R} |\tv^X_x|^{1/2} |(u-\tu^X)_x|^2 \, dx\\
	&\le C \delta_S \left( \int_{\mathbb{R}} |\tv^X_x|  \left| p(v)-p(\tv^X)-\frac{u-\tu^X}{2C_*} \right|^2 \, dx + \int_\mathbb{R} |\tv^X_x| |u-\tu^X|^2 \, dx +  \int_\mathbb{R}  |(u-\tu^X)_x |^2 \, dx \right)\\
	&\le C \delta_S (\mathcal{G}_1 + G^S + \mathcal{D}) \le \frac{1}{100} \left( \mathcal{D} +C_1 G^S + \mathcal{G}_1 \right).
	\end{align*}
	We estimate $\mathcal{B}_5$ by using similar argument as before
	\begin{align*}
	|\mathcal{B}_5(U)|&\le C\delta_S\int_{\mathbb{R}} |a_x^X| \left|p(v)-p(\widetilde{v}^X)-\frac{u-\widetilde{u}^X}{2C_\ast}\right|^2 \,dx+C\sqrt{\delta_S} \int_{\mathbb{R}} |\tv_x^X| |u-\widetilde{u}^X|^2 \,dx \\
	&\le \frac{1}{100}\left(\mathcal{G}_1+C_1G^S\right).
	\end{align*}
	Finally, to estimate $\mathcal{B}_6$, we use
	\[|p|^3\le 8(|p-q|^3+|q|^3),\quad \mbox{for any}\quad p,q\in \mathbb{R}\]
	and the interpolation inequality to obtain
	\begin{align*}
	|\mathcal{B}_6(U)|&\le C\int_{\mathbb{R}} |a_x^X| \left|p(v)-p(\widetilde{v}^X)-\frac{u-\widetilde{u}^X}{2C_\ast}\right|^3dx+C\int_{\mathbb{R}} |a^X_x| \left|u-\widetilde{u}^X\right|^3dx \\
	&\le C\varepsilon_1\mathcal{G}_1+C\lVert u-\widetilde{u}^X \rVert_{L^\infty{(\mathbb{R})}}^2 \int_{\mathbb{R}} |a_x^X||u-\widetilde{u}^X| \,dx\\
	&\le C\varepsilon_1\mathcal{G}_1+C\lVert u-\widetilde{u}^X \rVert_{L^2{(\mathbb{R})}}\lVert (u-\widetilde{u}^X)_x \rVert_{L^2{(\mathbb{R})}}\int_{\mathbb{R}}\frac{|\widetilde{v}_x^X|}{\sqrt{\delta_S}}|u-\widetilde{u}^X|dx \\
	&\le C\varepsilon_1\mathcal{G}_1+\frac{1}{100}\mathcal{D}+\frac{C}{\delta_S}\lVert u-\widetilde{u}^X \rVert_{L^2{(\mathbb{R})}}^2 \int_{\mathbb{R}} |\widetilde{v}_x^X| dx\int_{\mathbb{R}} |\widetilde{v}_x^X| |u-\tu^X|^2 dx \\
	&\le C\varepsilon_1\mathcal{G}_1+\frac{1}{100}\mathcal{D}+C\varepsilon_1^2G^S\le \frac{1}{100}(\mathcal{G}_1+C_1G^S+\mathcal{D}). 
	\end{align*}

	\noindent $\bullet$(Estimate of $\mathcal{K}_i$ for $i=1,\ldots,6$): Finally, we control the terms $\mathcal{K}_i$. We estimate $\mathcal{K}_1$ by using $|a^X_x| \sim \frac{|\tv^X_x|}{\sqrt{\delta_S}} $ and Young's inequality as
	\begin{align*}
	|\mathcal{K}_1(U)| &\le C \int_\mathbb{R} |a^X_x| |u-\tu^X| |(w-\tw^X)_x| \, dx \\
	&\le \frac{C}{\sqrt{\delta_S}} \int_\mathbb{R} |\tv^X_x| |u-\tu^X| |(w-\tw^X)_x| \, dx \\
	&\le \frac{C}{\sqrt{\delta_S}} \left( \int_\mathbb{R} |\tv^X_x|^{3/2} |u-\tu^X|^2 \, dx + \int_\mathbb{R} |\tv^X_x|^{1/2} |(w-\tw^X)_x|^2 \, dx \right) \\
	&\le \frac{C}{\sqrt{\delta_S}} \left( \|\tv^X_x\|_{L^\infty(\R)}^{1/2} G^S + \|\tv^X_x\|_{L^\infty(\R)}^{1/2} \|(w-\tw^X)_x\|_{L^2(\R)}^2 \right) \le C\sqrt{\delta_S} (G^S + \|(w-\tw^X)_x\|_{L^2(\R)}^2).
	\end{align*}
	For $\mathcal{K}_2$, we use Young's inequality and $\|a^X_x\|_{L^\infty(\R)} \le C \delta_S^{3/2} $ to find
	\begin{align*}
	|\mathcal{K}_2(U)| &\le C \int_\mathbb{R} |a^X_x| |w-\tw^X| |(u-\tu^X)_x| \, dx \\
	&\le C \left( \int_\mathbb{R} |a^X_x|^{3/2} |w-\tw^X|^2 \, dx + \int_\mathbb{R} |a^X_x|^{1/2} |(u-\tu^X)_x|^2 \, dx \right) \\
	&\le C \left(\|a^X_x\|_{L^\infty(\R)}^{1/2} \mathcal{G}_3 +\|a^X_x\|_{L^\infty(\R)}^{1/2} \mathcal{D} \right) \le C \delta_S^{3/4} (\mathcal{G}_3 + \mathcal{D}).
	\end{align*}
	The term $\mathcal{K}_3$ can be bounded by using  $|\tw^X_x| \le C \left( |\tv^X_{xx}| + |\tv^X_x| \right) \le C |\tv^X_x| $ and Young's inequality as
	\begin{align*}
	|\mathcal{K}_3(U)| &\le C \int_\mathbb{R} |a^X_x| |v-\tv^X| |u-\tu^X| |\tv^X_x| \, dx \\
	&\le C \delta_S^2 \left( \int_\mathbb{R} |a^X_x|  \left| p(v)-p(\tv^X) \right|^2 \, dx +  \int_\mathbb{R} |a^X_x| \left| u-\tu^X \right|^2  \, dx \right) \\
	&\le C \delta_S^2 \left( \int_\mathbb{R} |a^X_x|  \left| p(v)-p(\tv^X)-\frac{u-\tu^X}{2C_*} \right|^2 \, dx +  \int_\mathbb{R} |a^X_x| \left| u-\tu^X \right|^2  \, dx \right) \\
	&\le C \delta_S^2 \mathcal{G}_1 + C  \delta_S^{3/2} G^S.
	\end{align*}
	Similarly, $\mathcal{K}_4$ is bounded by using $|\tu^X_x| \sim |\tv^X_x|$ and Young's inequality as
	\begin{align*}
	|\mathcal{K}_4(U)| &\le C \int_\mathbb{R} |a^X_x| |v-\tv^X| |w-\tw^X| |\tv^X_x| \, dx \\
	&\le C \delta_S^2\left( \int_\mathbb{R} |a^X_x| |p(v)-p(\tv^X)|^2 \, dx + \int_\mathbb{R} |a^X_x| |w-\tw^X|^2 \, dx\right) \\
	&\le C \delta_S^2 \left( \int_\mathbb{R} |a^X_x| \left|p(v)-p(\tv^X) - \frac{u-\tu^X}{2 C_* } \right|^2 \, dx + \int_\mathbb{R} |a^X_x| |u-\tu^X|^2 \, dx + \int_\mathbb{R} |a^X_x| |w-\tw^X |^2 \, dx\right) \\
	&\le C \left( \delta_S^2 \mathcal{G}_1 + \delta_S^{3/2} G^S + \delta_S^2 \mathcal{G}_3 \right)\le C\delta_S^{3/2}(\mathcal{G}_1+G^S+\mathcal{G}_3).
	\end{align*}
	We estimate $\mathcal{K}_5$ by using $|\tw^X_x| \le C \left( |\tv^X_{xx}| + |\tv^X_x| \right) \le C |\tv^X_x| $ as
	\begin{align*}
	|\mathcal{K}_5 (U)| &\le C \int_\mathbb{R} |\tv^X_x| |(u-\tu^X)_x| |v-\tv^X| \, dx \\
	&\le C \int_\mathbb{R} |\tv^X_x|^{3/2} |p(v)-p(\tv^X)|^2 \, dx + \int_\mathbb{R} |\tv^X_x|^{1/2} |(u-\tu^X)_x|^2 \, dx\\
	&\le C \delta_S \left( \int_{\mathbb{R}} |\tv^X_x|  \left| p(v)-p(\tv^X)-\frac{u-\tu^X}{2C_*} \right|^2 \, dx + \int_\mathbb{R} |\tv^X_x| |u-\tu^X|^2 \, dx +  \int_\mathbb{R}  |(u-\tu^X)_x |^2 \, dx \right)\\
	&\le C \delta_S (\mathcal{G}_1 + G^S + \mathcal{D}).
	\end{align*}
	Finally, we bound $\mathcal{K}_6$ by using $|\tu^X_x| \sim |\tv^X_x|$ as
	\begin{align*}
	|\mathcal{K}_6 (U)| &\le C \int_\mathbb{R} |\tv^X_x| |(w-\tw^X)_x| |v-\tv^X| \, dx \\
	&\le C \int_\mathbb{R} |\tv^X_x|^{3/2} |p(v)-p(\tv^X)|^2 \, dx + \int_\mathbb{R} |\tv^X_x|^{1/2} |(w-\tw^X)_x|^2 \, dx\\
	&\le C \delta_S \left( \int_{\mathbb{R}} |\tv^X_x|  \left| p(v)-p(\tv^X)-\frac{u-\tu^X}{2C_*} \right|^2 \, dx + \int_\mathbb{R} |\tv^X_x| |u-\tu^X|^2 \, dx +  \|(w-\tw^X)_x\|_{L^2 (\R)}^2 \right)\\
	&\le C \delta_S \left(\mathcal{G}_1 + G^S + \|(w-\tw^X)_x\|_{L^2(\R)}^2 \right) .
	\end{align*}
	Combining all the estimates for $\mathcal{B}_i$ $(i=2,\ldots,6)$ and for $\mathcal{K}_i$ $(i=1,\ldots,6)$ and using the smallness of the parameters, we conclude that
	\begin{equation}\label{est-B,K}
	\begin{aligned}
	\sum_{i=2}^6 \mathcal{B}_i + \sum_{i=1}^6 \mathcal{K}_i &\le \left(\frac{4}{100}+C \left(\delta_S^2 + \delta_S \right)\right) \mathcal{G}_1 + C_1 \left(\frac{5}{100} + C( \delta_S^{3/2} +\delta_S +\sqrt{\delta_S} ) \right) G^S +C \left( \delta_S^{3/4} + \delta_S^2 \right) \mathcal{G}_3\\ 
	&\quad  +  \left(  \frac{3}{100} +C\left(\delta_S + \delta_S^{3/4}\right) \right) \mathcal{D} + C \left( \sqrt{\delta_S} + \delta_S \right) \| (w-\tw^X)_x \|_{L^2(\R)}^2 \\ 
	&\le \frac{6}{100} \left( \mathcal{G}_1 + \mathcal{G}_3 +  C_1 G^S +  \mathcal{D}\right) + C\sqrt{\delta_S} \|(w-\tw^X)_x\|_{L^2(\mathbb{R})}^2. 
	\end{aligned}
	\end{equation}
	
	\subsection{Proof of Lemma \ref{Main Lemma}}
	We are now ready to prove the key lemma, Lemma \ref{Main Lemma}. We combine all the estimates in \eqref{est-R}, \eqref{est-Y}, and \eqref{est-B,K} to derive the following control on the weighted relative entropy:
	\begin{align*}
	\frac{d}{dt} \int_\mathbb{R} a^X \eta (U | \tU^X ) \, dx &\le -\frac{\delta_S}{4 M} |\dot{X}|^2 -\frac{93}{100} \mathcal{G}_1 -\frac{93}{100} \mathcal{G}_3 \\ 
	&\quad -\frac{93}{100}C_1 G^S-\frac{19}{100} \mathcal{D} + C \sqrt{\delta_S} \| (w-\tw^X)_x\|_{L^2(\R)}^2.
	\end{align*}
	After integrating the above inequality on $[0,t]$ for any $t \le T$, we conclude that
	\begin{align*}
	&\int_\mathbb{R} a^X(t,x) \eta (U (t,x) | \tU^X (t,x) ) \, dx + \int_0^t \left( \delta_S |\dot{X}|^2 + \mathcal{G}_1+ \mathcal{G}_3 + G^S+ \mathcal{D} \right) \, ds \\ 
	& \quad \le C \left( \int_\mathbb{R} a(0,x) \eta ( U_0(x)| \tU(0,x)) \, dx + \sqrt{\delta_S} \int_0^t \|(w-\tw^X)_x\|_{L^2(\R)}^2 \right).
	\end{align*}
	However, since the bounds $1/2 \le a \le 2$, $D(U) \le C \mathcal{D}(U)$, $G_1(U) \le C \mathcal{G}_1(U)$, $G_3(U) \le C \mathcal{G}_3(U)$, and the relation
	\begin{equation*}
	\|U-\tU^X\|_{L^2(\mathbb{R})}^2 \sim \int_\mathbb{R} \eta(U|\tU^X) \, dx, \quad \forall t \in [0,T]
	\end{equation*} 
	holds, this completes the proof of Lemma \ref{Main Lemma}.\\
	
	\begin{remark}
		In the following section, we omit the dependency on the shift $X$ in \eqref{viscous-dispersive-shock-ext} and \eqref{a} for the simplicity of notation as follows:
		\begin{equation*}
		(\tv,\tu,\tw)(t,x):=(\tv,\tu,\tw)(x-\sigma t -X(t)),\quad a(t,x):=a(x-\sigma t - X(t)).
		\end{equation*}
		
	\end{remark}
	
	\section{Estimate on the $H^1$-perturbation}\label{sec:high-order}
	\setcounter{equation}{0}
	In this section, we provide the estimate on the $H^1$-perturbation, and complete the proof of Proposition \ref{apriori-estimate}. To achieve this, in addition to the estimate of the $L^2$-perturbation between $(v,u,w)$ and $(\widetilde{v},\widetilde{u},\widetilde{w})$ obtained in the previous section, we need to derive higher-order estimates for the $H^1$-perturbation between $(u,w)$ and $(\widetilde{u},\widetilde{w})$. Therefore, our goal of this section is to establish the following lemma.
	\begin{lemma} \label{Main Lemma 2} Under the hypotheses of Proposition \ref{apriori-estimate}, there exists a positive constant $C$, independent of $\delta_S,\varepsilon_1,T$ such that for all $t \in [0,T],$
		\begin{equation}
		\begin{aligned}\label{eq:H1-est}
		& \norm{(v-\tv)(t,\cdot)}_{L^2(\mathbb{R})}^2 +\norm{(u-\tu)(t,\cdot)}_{H^1(\mathbb{R})}^2+\norm{(w-\tw)(t,\cdot)}_{H^1(\mathbb{R})}^2+\delta_S \int_0^t | \dot{X}(s)|^2 \, d s \\ 
		&\quad +\int_0^t \left( G_1+G_3+G^S \right) \, ds +  \int_0^t \left( D_{u_1} + D_{u_2} + G_{w} + D_{w_1} + D_{w_2} \right)\, ds  \\ 
		& \le C \norm{(v-\tv)(0,\cdot)}_{L^2(\mathbb{R})}^2 +\norm{(u-\tu)(0,\cdot)}_{H^1(\mathbb{R})}^2+\norm{(w-\tw)(0,\cdot)}_{H^1(\mathbb{R})}^2,
		\end{aligned}
		\end{equation}
		or equivalently,
		\begin{equation*}
		\begin{aligned}
		& \norm{(v-\tv)(t,\cdot)}_{H^2 (\mathbb{R})}^2 +\norm{(u-\tu)(t,\cdot)}_{H^1(\mathbb{R})}^2+\delta_S \int_0^t | \dot{X}(s)|^2 \, d s \\ 
		&\quad +\int_0^t \left( G_1+G_3+G^S \right) \, ds +  \int_0^t \left( D_{u_1} + D_{u_2} + G_{w} + G_{w_1} + G_{w_2} \right)\, ds  \\ 
		& \le C \norm{(v-\tv)(0,\cdot)}_{H^2 (\mathbb{R})}^2 +\norm{(u-\tu)(0,\cdot)}_{H^1(\mathbb{R})}^2 ,
		\end{aligned}
		\end{equation*}
		where
		\begin{align*}
		G_1 &:= \int_\mathbb{R} |a_x| \left| p(v)-p(\tv) - \frac{u-\tu}{2C_*} \right|^2 \, dx,\\
		G_3 &:= \int_\mathbb{R} |a_x| |w-\tw|^2 \, dx,\\
		G^S &:= \int_\mathbb{R} |\tv_x| |u-\tu|^2 \, dx,\\
		D_{u_1} &:= \int_\mathbb{R} |\partial_x(u-\tu)|^2 \, dx,\quad D_{u_2} := \int_\mathbb{R} |\partial_{xx}(u-\tu)|^2 \, dx, \\
		G_{w} &:= \int_\mathbb{R} |(w-\tw)|^2 \, dx,\quad D_{w_1} := \int_\mathbb{R} |\partial_x(w-\tw)|^2 \, dx,\quad D_{w_2} := \int_\mathbb{R} |\partial_{xx}(w-\tw)|^2 \, dx. 
		\end{align*}
	\end{lemma}
	For simplicity, we introduce the following notations for the perturbation:
	\begin{equation} \label{Notations: phi,psi,omega}
	\phi(t,x):=v(t,x)-\tv(t,x),\ \psi(t,x):=u(t,x)-\tu(t,x),\ \omega (t,x):=w(t,x)-\tw(t,x).
	\end{equation}
	Then the triplet $(\phi,\psi,\omega)$ satisfies 
	\begin{equation} \label{Underlying wave with shift}
	\begin{aligned}
	&\phi_t-\psi_x=\dot{X} \tv_x,\\
	&\psi_t + \left(p(v)-p(\tv) \right)_x = \left( \frac{u_x}{v}-\frac{\tu_x}{\tv} \right)_x + \left( \frac{w_x}{v^{5/2}} - \frac{\tw_x}{\tv^{5/2}} \right)_x + \dot{X} \tu_x,\\
	&\omega _t= -\left( \frac{u_x}{v^{5/2}} - \frac{\tu_x}{\tv^{5/2}} \right)_x + \dot{X} \tw_x.
	\end{aligned}
	\end{equation}
	For later reference, we will represent some of the nonlinear terms in the above systems \eqref{Underlying wave with shift} as follows:
	\begin{equation} \label{eq: P,F_1,F_2}
	\begin{aligned}
	&\left( p(v)-p(\tv) \right)_x =p'(v) \phi_x + \tv_x \left( p'(v) - p'(\tv) \right), \\ 
	&\left( \frac{u_x}{v}-\frac{\tu_x}{\tv} \right)_x = \frac{\psi_{xx}}{v}-\frac{\psi_x v_x}{v^2}+\tu_{xx} \left( \frac{1}{v}- \frac{1}{\tv} \right) + \tu_x \left( \frac{\tv_x}{\tv^2}-\frac{v_x}{v^2} \right),\\
	&\left( \frac{w_x}{v^{5/2}} - \frac{\tw_x}{\tv^{5/2}} \right)_x = v^{-5/2} \omega _{xx} -\frac{5}{2} v^{-7/2} v_x \omega _x \\
	&\hspace{4cm}+ \tw_{xx} \left( v^{-5/2} - \tv^{-5/2} \right) +\frac{5}{2} \tw_x \left(  \tv^{-7/2} \tv_x -v^{-7/2} v_x \right). 
	\end{aligned}
	\end{equation}
	We note that the a priori assumptions \eqref{smallness} and the Sobolev inequality imply the control on the $L^\infty$-norms for the perturbations
	\begin{equation} \label{eq: L^infty bounds}
	\norm{\phi}_{L^\infty((0,T)\times\R)}+\norm{\phi_x}_{L^\infty((0,T)\times\R)}+\norm{\psi}_{L^\infty((0,T)\times\R)} \le C \varepsilon_1.
	\end{equation}
	The following lemma will be used throughout this section.
	\begin{lemma} \label{lem: Useful bad terms estimates} Under the hypotheses of Proposition \ref{apriori-estimate}, there exists a positive constants $C$ that is independent of $\delta_S$ and $\varepsilon_1$, such that for all $t \in [0,T],$ and for $p>0$
		\begin{equation*} 
		\begin{aligned}
		&\int_0^t\int_\mathbb{R} (|\tv_x|+|\tu_x|)^{1+p} \left|u-\tu\right|^2 \, dx d s \le C \delta_S^{2p} \int_0^t G^S \, d s, \\ 
		&\int_0^t\int_\mathbb{R} (|\tv_x|+|\tu_x|)^{1+p} \left|v-\tv\right|^2 \, dx d s  \le C \delta_S^{2p} \int_0^t \left( G_1 +G^S \right) d s.
		\end{aligned}
		\end{equation*}
	\end{lemma}
	\begin{proof} 
		It follows from \eqref{shock-property} that $|\tv_x|\sim |\tu_x|$, and $\|\tv_x\|_{L^\infty}\le C\delta_S^2$, which implies
		\begin{align*}
		&\int_0^t\int_\mathbb{R} (|\tv_x|+|\tu_x|)^{1+p} \left|u-\tu\right|^2 \, dx d s \le C \norm{\tv_x}_{L^\infty((0,T)\times \mathbb{R})}^p \int_0^t  G^S \,  ds  \le C\delta_S^{2p} \int_0^t  G^S \,  ds.
		\end{align*}
		Similarly, we use the equivalence between $|v-\tv|$ and $|p(v)-p(\tv)|$ to obtain
		\begin{align*}
		&\int_0^t\int_\mathbb{R} (|\tv_x|+|\tu_x|)^{1+p} \left|v-\tv\right|^2 \, dx d s \le C \int_0^t \int_\mathbb{R} |\tv_x|^{1+p} \left|p(v)-p(\tv)\right|^2 \, dx ds \\ 
		&\le C \int_0^t \int_\mathbb{R} |\tv_x|^{1+p} \left|p(v)-p(\tv) - \frac{u-\tu}{2C_*}\right|^2 \, dx ds +C \int_0^t \int_\mathbb{R} |\tv_x|^{1+p} \left|u-\tu\right|^2 \, dx ds \\ 
		&\le C \norm{\tv_x}_{L^\infty((0,T)\times \mathbb{R})}^p \int_0^t  (G_1 + G^S) \,  ds  \le C \delta_S^{2p} \int_0^t  (G_1 + G^S) \,  ds.
		\end{align*}
	\end{proof}
	
	We now briefly present the idea of proving Lemma \ref{Main Lemma 2}. In the following, we present three lemmas, Lemma \ref{lem: est-0}, Lemma \ref{lem: est-1}, and Lemma \ref{lem: est-2}, which provide a control on the different quantities for the perturbations $(v-\tv,u-\tu,w-\tw)$. After we prove those lemmas, we combine these estimates by multiplying appropriate weights to obtain the desired results in the main lemma, Lemma \ref{Main Lemma 2}.
	
	\begin{lemma} \label{lem: est-0}Under the hypotheses of Proposition \ref{apriori-estimate}, there exists a positive constant $C$ that is independent of $\delta_S$ and $\e_1$, such that for all $t \in [0,T],$
		\begin{equation} \label{eq: est-0}
		\begin{aligned}
		&\norm{ \left(u-\widetilde{u}, w-\widetilde{w} \right)_x(t,\cdot)}_{L^2(\mathbb{R})}^2+  \int_0^t D_{u_2} \, d s \\
		&\le C \norm{  \left( u-\tu,w-\tw\right)_x(0,\cdot)}_{L^2(\mathbb{R})}^2+C\delta_S \int_0^t |\dot{X}(s)|^2 \, d s + C \int_0^t \| (v-\tv)_x\|_{L^2 (\mathbb{R})}^2 \, d s \\ 
		& \quad +C \delta_S \int_0^t  \left( G_1 + G^S \right) \, d s \\
		&\quad + C (\varepsilon_1+\delta_S) \int_0^t  \left( D_{u_1}+\| (w-\tw)_x \|_{L^2(\mathbb{R})}^2 + \| (w-\tw)_{xx}\|_{L^2 (\mathbb{R})}^2 \right) \, d s .
		\end{aligned}
		\end{equation}
	\end{lemma}
	\begin{proof}
		We multiply the equation $\eqref{Underlying wave with shift}_2$ and \eqref{Underlying wave with shift}$_3$ by $-\psi_{xx}$ and $-\omega_{xx}$ to obtain
		\begin{equation}\label{eq: psi_t psi_xx}
		-\psi_t \psi_{xx} - \left(p(v)-p(\tv) \right)_x \psi_{xx} = -\left( \frac{u_x}{v}-\frac{\tu_x}{\tv} \right)_x \psi_{xx} - \left( \frac{w_x}{v^{5/2}} - \frac{\tw_x}{\tv^{5/2}} \right)_x \psi_{xx} - \dot{X} \tu_x \psi_{xx}
		\end{equation}
		and 
		\begin{equation}\label{eq: omega_t omega_xx}
		-\omega _t \omega_{xx}= \left( \frac{u_x}{v^{5/2}} - \frac{\tu_x}{\tv^{5/2}} \right)_x \omega_{xx} - \dot{X} \tw_x \omega_{xx}
		\end{equation}
		respectively. Subsequently, we add \eqref{eq: psi_t psi_xx} and \eqref{eq: omega_t omega_xx} and integrate over the whole space  $\mathbb{R}$ to derive
		\begin{align*}
		&\frac{d}{dt}\int_{\mathbb{R}} \frac{|\psi_x|^2}{2} \,dx+\frac{d}{dt}\int_{\mathbb{R}} \frac{|\omega_x|^2}{2} \,dx\\ 
		&\quad =-\dot{X}\int_{\mathbb{R}} \widetilde{u}_x\psi_{xx}+\tw_x \omega_{xx}\,dx+\int_{\mathbb{R}} (p(v)-p(\tv))_x \psi_{xx}\,dx \\ 
		&\qquad -\int_{\mathbb{R}} \left(\frac{u_x}{v}-\frac{\widetilde{u}_x}{\widetilde{v}}\right)_x\psi_{xx}\,dx +\int_{\mathbb{R}} \left[ \left( \frac{u_x}{v^{5/2}} - \frac{\tu_x}{\tv^{5/2}} \right)_x \omega_{xx}-\left(\frac{w_x}{v^{5/2}}-\frac{\widetilde{w}_x}{\widetilde{v}^{5/2}}\right)_x \psi_{xx} \right]\, dx\\
		&\quad =: K_1+K_2+K_3+K_4. 
		\end{align*}
		$\noindent \bullet$ (Estimate of $K_1$ and $K_2$): We first define the good term
		\[\mathcal{D}_{u_2}:=\int_{\mathbb{R}}\frac{|\psi_{xx}|^2}{v}\,dx,\]
		which will be derived from the estimate of $K_3$. Using $|\tu_x| \sim |\tv_x|$ and $|\tw_x| \le (|\tv_{xx}| + |\tv_x|) \le C |\tv_x|$ in Lemma \ref{lem:shock-property}, and applying H\"older's inequality, we estimate $K_1$ as
		\begin{align*}
		|K_1|
		&\le C |\dot{X}|\sqrt{\int_\mathbb{R} |\tv_x|^2 \, dx } \left( \sqrt{\int_{\mathbb{R}} |\psi_{xx}|^2 \,dx} + \sqrt{\int_\mathbb{R} |\omega_{xx}|^2 \, dx}\right) \\
		&\le C |\dot{X}| \delta_S^{3/2} \left( \norm{\psi_{xx}}_{L^2(\mathbb{R})} + \norm{\omega_{xx}}_{L^2(\mathbb{R})} \right) \\ 
		&\le \delta_S|\dot{X}|^2+C\delta_S^2 \left( \mathcal{D}_{u_2} + \lVert \omega_{xx} \rVert^2_{L^2(\mathbb{R})} \right) \le  \delta_S|\dot{X}|^2+ \frac{1}{16} \mathcal{D}_{u_2} + C \delta_S^2 \| \omega_{xx} \|_{L^2 (\mathbb{R})}^2.
		\end{align*}
		On the other hand, by using Lemma \ref{lem : Estimate-relative}, we have 
		\begin{equation*} 
		\left|\left(p(v)-p(\tv)\right)_x\right| = \left| p'(v) (v-\tv)_x + \tv_x \left( p'(v) - p'(\tv) \right) \right| \le C \left( |\phi_x| + |\tv_x| |\phi| \right).
		\end{equation*}
		Then, we use \eqref{eq: P,F_1,F_2}, Young's inequality, and apply Lemma \ref{lem: Useful bad terms estimates} to estimate $K_2$ as
		\begin{align*}
		|K_2|&\le C \int_\mathbb{R} |\phi_x| |\psi_{xx}| \, dx + C \int_\mathbb{R} |\tv_x| |\phi| |\psi_{xx}| \, dx \\
		&\le \frac{1}{16}\int_{\mathbb{R}} |\psi_{xx}|^2 dx + C \int_{\mathbb{R}} |\phi_x|^2 dx + C \int_\mathbb{R} |\tv_x|^2 |\phi|^2 \, dx  \\ 
		&\le \frac{1}{16} \mathcal{D}_{u_2}+C \lVert \phi_{x} \rVert^2_{L^2(\mathbb{R})}+C \delta_S^2 \left(G_1+G^S\right).
		\end{align*}
		
		\noindent $\bullet$ (Estimate of $K_3$): We decompose $K_3$ into
		\begin{align*}
		K_3=&-\int_{\mathbb{R}}\frac{|\psi_{xx}|^2}{v} \, dx-\int_{\mathbb{R}} \left(\frac{1}{v}\right)_x\psi_x\psi_{xx} \, dx \\ 
		&-\int_{\mathbb{R}} \widetilde{u}_{xx}\left(\frac{1}{v}-\frac{1}{\widetilde{v}}\right)\psi_{xx} \, dx-\int_{\mathbb{R}} \widetilde{u}_x\left(\frac{1}{v}-\frac{1}{\widetilde{v}}\right)_x\psi_{xx} \, dx \\
		=&-\mathcal{D}_{u_2}+K_{3,1}+K_{3,2}+K_{3,3}.
		\end{align*}
		We estimate $K_{3,1}$ by using $\left(\frac{1}{v}\right)_x \le C |v_x| \le C ( |\phi_x| + |\tv_x| )$ and the smallness assumption \eqref{eq: L^infty bounds} as
		\begin{align*}
		|K_{3,1}| &\le \| \phi_x \|_{L^\infty(\mathbb{R})} \| \psi_x \|_{L^2 (\mathbb{R})} \| \psi_{xx} \|_{L^2(\mathbb{R})} + \| \tv_x \|_{L^\infty(\mathbb{R})} \| \psi_x \|_{L^2 (\mathbb{R})} \| \psi_{xx} \|_{L^2(\mathbb{R})} \\ 
		&\le C (\varepsilon_1 + \delta_S ) ( \| \psi_x \|_{L^2(\mathbb{R})}^2 + \| \psi_{xx} \|_{L^2 (\mathbb{R})}^2 ) \\ 
		&\le \frac{1}{16} \mathcal{D}_{u_2} + C (\varepsilon_1 + \delta_S) D_{u_1}.
		\end{align*}
		To estimate $K_{3,2}$ and $K_{3,3}$, we first note that the following control on the perturbation holds for any $\alpha>0$ by the Mean value theorem:
		\begin{equation}\label{eq: (1/v^a-1/tv^a) and (1/v^a-1/tv^a)_x}
		\left|\frac{1}{v^\alpha}-\frac{1}{\tv^\alpha}\right| \le C|v-\tv| \quad \text{and} \quad
		\left|\left(\frac{1}{v^\alpha}-\frac{1}{\tv^\alpha} \right)_x\right| \le C \left(|(v-\tv)_x| + |\tv_x| |v-\tv| \right).
		\end{equation}
		Then, we use \eqref{eq: (1/v^a-1/tv^a) and (1/v^a-1/tv^a)_x}, $|\tu_{xx}| \le C |\tu_x|$ in Lemma \ref{lem:shock-property}, Young's inequality, and Lemma \ref{lem: Useful bad terms estimates} to estimate $K_{3,2}$ as
		\begin{align*}
		|K_{3,2}| &\le C \int_{\mathbb{R}} |\tu_x| |\phi||\psi_{xx}| \, dx  \le \frac{1}{16} \mathcal{D}_{u_2} + \int_\mathbb{R} |\tu_x|^2 |\phi|^2 \, dx \\
		&\le \frac{1}{16} \mathcal{D}_{u_2} + C\delta_S^2 \left(G_1+G^S \right).
		\end{align*}
		By the same argument, $K_{3,3}$ can be estimated as
		\begin{align*}
		|K_{3,3}| &\le C \int_{\mathbb{R}} |\tu_x| (|\phi_x| + |\tv_x| |\phi|) |\psi_{xx}| \, dx \\
		&\le \frac{1}{16} \mathcal{D}_{u_2} + C \int_\mathbb{R} |\tu_x|^2 |\phi_x|^2 \, dx +C \int_\mathbb{R} |\tu_x|^2 |\tv_x|^2 |\phi|^2 \, dx \\ 
		&\le \frac{1}{16} \mathcal{D}_{u_2} + C \delta_S^4 \left( \| \phi_x \|_{L^2(\mathbb{R})}^2 + G_1 + G^S \right).
		\end{align*}
		Combining the estimates for $K_{3,1},K_{3,2}$, and $K_{3,3}$, we obtain
		\begin{align*}
		K_3\le -\frac{13}{16}\mathcal{D}_{u_2}+C(\e_1+\delta_S)D_{u_1}+C\delta_S^2(G_1+G^S)+C\delta_S^4\|\phi_x\|_{L^2(\R)}^2.
		\end{align*}
		$\bullet$ (Estimate of $K_4$): We split $K_4$ into
		\begin{align*}
		K_4&=\int_{\mathbb{R}} \left[ \left( \frac{u_x}{v^{5/2}} - \frac{\tu_x}{\tv^{5/2}} \right)_x \omega_{xx}-\left(\frac{w_x}{v^{5/2}}-\frac{\widetilde{w}_x}{\widetilde{v}^{5/2}}\right)_x \psi_{xx}  \right] \, dx \\ 
		&= \int_{\mathbb{R}}\left( \frac{\psi_x}{v^{5/2}} + \tu_x \left( \frac{1}{v^{5/2}}-\frac{1}{\tv^{5/2}} \right) \right)_x \omega_{xx} \, dx -  \int_{\mathbb{R}}\left( \frac{\omega_x}{v^{5/2}} + \tw_x \left( \frac{1}{v^{5/2}}-\frac{1}{\tv^{5/2}} \right) \right)_x \psi_{xx} \, dx \\
		&=\int_{\mathbb{R}}\psi_x \left( \frac{1}{v^{5/2}}\right)_x \omega_{xx} \, dx + \int_{\mathbb{R}}\tu_{xx} \left( \frac{1}{v^{5/2}}-\frac{1}{\tv^{5/2}} \right) \omega_{xx}dx+\int_{\mathbb{R}} \tu_x \left( \frac{1}{v^{5/2}}-\frac{1}{\tv^{5/2}} \right)_x \omega_{xx} \, dx \\ 
		& \; \ \  -\int_{\mathbb{R}}\omega_x \left( \frac{1}{v^{5/2}}\right)_x \psi_{xx} \, dx + \int_{\mathbb{R}}\tw_{xx} \left( \frac{1}{v^{5/2}}-\frac{1}{\tv^{5/2}} \right) \psi_{xx} \, dx+\int_{\mathbb{R}}\tw_x \left( \frac{1}{v^{5/2}}-\frac{1}{\tv^{5/2}} \right)_x \psi_{xx} \, dx \\
		&=: K_{4,1}+K_{4,2}+K_{4,3}+K_{4,4}+K_{4,5}+K_{4,6}.
		\end{align*}
		
		We use $(\frac{1}{v^{5/2}})_x \le C |v_x| \le C(|\phi_x| + |\tv_x|)$ and \eqref{eq: L^infty bounds} to estimate $K_{4,1}$ as
		\begin{align*}
		\left|K_{4,1} \right| 
		&\le C \norm{ \phi_x }_{L^\infty(\mathbb{R})}  \norm{\psi_x}_{L^2(\mathbb{R})} \norm{\omega_{xx}}_{L^2 (\mathbb{R})} + C \norm{ \tv_x }_{L^\infty(\mathbb{R})}  \norm{\psi_x}_{L^2(\mathbb{R})} \norm{\omega_{xx}}_{L^2 (\mathbb{R})} \\
		&\le C (\varepsilon_1 + \delta_S ) ( \norm{\psi_x}_{L^2(\mathbb{R})}^2+ \norm{\omega_{xx}}_{L^2 (\mathbb{R})}^2  )\\  
		&\le  \frac{1}{16} \mathcal{D}_{u_2} +C (\varepsilon_1 + \delta_S ) \norm{\omega_{xx}}_{L^2 (\mathbb{R})}^2 .
		\end{align*}
		For $K_{4,2}$, we use \eqref{eq: (1/v^a-1/tv^a) and (1/v^a-1/tv^a)_x}, $|\tu_{xx}| \le C \delta_S |\tu_x|$ in Lemma \ref{lem:shock-property},  
		Young's inequality, and Lemma \ref{lem: Useful bad terms estimates} as
		\begin{align*}
		\left|K_{4,2}\right|
		&\le C \delta_S \int_\mathbb{R} |\tu_x | |\phi| |\omega_{xx} | \, dx \le C \delta_S \int_{\mathbb{R}} |\tu_x|^2 |\phi|^2 \, dx + C \delta_S \int_\mathbb{R}  |\omega_{xx}|^2 \, dx \\ 
		&\le C \delta_S^2 (G_1 + G^S )+C \delta_S \| \omega_{xx} \|_{L^2(\mathbb{R})}^2.
		\end{align*}
		Using the same argument, $K_{4,3}$ can be estimated as
		\begin{align*}
		\left|K_{4,3}\right| 
		&\le  C  \int_\mathbb{R} |\tu_x | \left( |\phi_x| + |\tv_x| |\phi| \right) |\omega_{xx} | \, dx \\ 
		&\le C \int_{\mathbb{R}} |\tu_x| |\phi_x|^2 \, dx +C \int_\mathbb{R} |\tu_x| |\tv_x|^2 |\phi|^2 \, dx + C \int_\mathbb{R} |\tu_x| |\omega_{xx}|^2 \, dx \\ 
		&\le C \delta_S^2 \left( \norm{\phi_x}_{L^2(\mathbb{R})}^2 + \| \omega_{xx} \|_{L^2(\mathbb{R})}^2 +G_1+G^S \right).
		\end{align*}
		The term $K_{4,4}$ can be treated by using the same argument as in $K_{4,1}$:
		\begin{align*}
		|K_{4,4}| 
		&\le C \int_\mathbb{R} |\omega_x| |v_x| |\psi_{xx}| \, dx  \\
		&\le C \| \phi_x \|_{L^\infty(\mathbb{R})} \| \omega_x \|_{L^2(\mathbb{R})} \|\psi_{xx} \|_{L^2 (\mathbb{R})} +C \| \tv_x \|_{L^\infty(\mathbb{R})} \| \omega_x \|_{L^2(\mathbb{R})} \|\psi_{xx} \|_{L^2 (\mathbb{R})}\\
		&\le C (\varepsilon_1 + \delta_S) (\| \omega_x \|_{L^2(\mathbb{R})}^2 + \|\psi_{xx} \|_{L^2 (\mathbb{R})}^2) \\ 
		&\le \frac{1}{16} \mathcal{D}_{u_2} +  C (\varepsilon_1 + \delta_S) \norm{\omega_x}_{L^2(\mathbb{R})}^2.
		\end{align*}
		We use $|\tw_{xx}| \le C ( |\tv_x|^3 + |\tv_x| |\tv_{xx}|+ |\tv_{xxx}|) \le C \delta_S^2 |\tv_x|$, and \eqref{eq: (1/v^a-1/tv^a) and (1/v^a-1/tv^a)_x} to estimate $K_{4,5}$ as
		\begin{align*}
		|K_{4,5}| &\le C \delta_S^2 \int_\mathbb{R} |\tv_x| |\phi| |\psi_{xx}| \, dx \le C \delta_S^2 \int_\mathbb{R} |\tv_x|^2 |\phi|^2 \, dx + C \delta_S^2 \int_\mathbb{R}  |\psi_{xx}|^2 \, dx \\
		&\le  \frac{1}{16} \mathcal{D}_{u_2} +C \delta_S^2 \left( G_1 + G^S \right).
		\end{align*}
		Similarly, we estimate $K_{4,6}$ by using \eqref{eq: (1/v^a-1/tv^a) and (1/v^a-1/tv^a)_x} and $|\tw_x|\le C|\tv_x|^2 + C |\tv_{xx}| \le C \delta_S |\tv_x|$ as
		\begin{align*}
		|K_{4,6}| 
		&\le C \int_\mathbb{R} |\tw_x| |\phi_x| |\psi_{xx}| \, dx +C \int_\mathbb{R} |\tw_x| |\tv_x| |\phi| |\psi_{xx}| \, dx \\ 
		&\le C \delta_S \int_\mathbb{R} |\tv_x| |\phi| |\psi_{xx}| \, dx + C \delta_S \int_\mathbb{R} |\tv_x|^2 |\phi| |\psi_{xx}| \, dx\\ 
		&\le C \delta_S \int_\mathbb{R} |\tv_x| |\phi_x|^2 \, dx + C \delta_S \int_\mathbb{R} |\tv_x| |\psi_{xx}|^2 \, dx +C \delta_S \int_\mathbb{R} |\tv_x|^2 |\phi|^2 \, dx \\
		&\le \frac{1}{16} \mathcal{D}_{u_2}+ C \delta_S^2 \left( \| \phi_x\|_{L^2(\mathbb{R})}^2 + G_1+G^S \right).
		\end{align*}
		Thus, $K_4$ can be bounded as
		\begin{align*}
		K_4\le \frac{1}{4}\mathcal{D}_{u_2}+C(\e_1+\delta_S)\|\omega_{xx}\|_{L^2(\R)}^2+C\delta_S^2(G_1+G^S)+C\delta_S^2\|\phi_x\|_{L^2(\R)}^2+C(\e_1+\delta_S)\|\omega_{x}\|_{L^2(\R)}^2.
		\end{align*}
		Therefore, combining all estimates for $K_i$ with $i=1,2,3,4$ and using the smallness of $\delta_S$ and $\varepsilon_1$, there exists a positive constant $C$, which is independent with $\delta_S$, $\varepsilon_1$, and $T$, such that
		\begin{align*}
		&\frac{d}{dt}\int_{\mathbb{R}} \frac{|\psi_x|^2}{2} \,dx+\frac{d}{dt}\int_{\mathbb{R}} \frac{|\omega_x|^2}{2} \,dx + \frac{1}{16} \mathcal{D}_{u_2} \\ 
		&\le \delta_S |\dot{X}|^2 +C \| \phi_x\|_{L^2(\mathbb{R})}^2+C \delta_S \left( G_1 + G^S \right)+ C (\varepsilon_1+\delta_S)  \left( D_{u_1}+ \norm{\omega_{x}}_{L^2(\mathbb{R})}^2 + \|\omega_{xx}\|_{L^2 (\mathbb{R})}^2 \right) .
		\end{align*}
		Integrating from $0$ to $t$ for any $t\in \left[0, T  \right]$, we obtain 
		\begin{align*}
		&\frac{1}{2}\norm{ \big(\psi_x, \omega_x \big)}_{L^2(\mathbb{R})}^2+ \frac{1}{16} \int_0^t \mathcal{D}_{u_2} \, d s \\
		&\le \frac{1}{2} \norm{ \big(\psi_{0x}, \omega_{0x} \big)}_{L^2(\mathbb{R})}^2+\delta_S \int_0^t |\dot{X}(s)|^2 \, d s + C \int_0^t \| (v-\tv)_x\|_{L^2 (\mathbb{R})}^2 \, d s \\ 
		& \quad +C \delta_S \int_0^t  \left( G_1 + G^S \right) \, d s + C (\varepsilon_1+\delta_S) \int_0^t  \left( D_{u_1}+\| (w-\tw)_x \|_{L^2(\mathbb{R})}^2 + \| (w-\tw)_{xx}\|_{L^2 (\mathbb{R})}^2 \right) \, d s .\\  
		\end{align*}
		By the smallness assumption \eqref{smallness}, $v$ is bounded below in $[0,T]$ by some positive constant, which implies $D_{u_2} \le C \mathcal{D}_{u_2}$. This completes the proof of Lemma \ref{lem: est-0}.
		
	\end{proof}

	\begin{lemma}\label{lem: est-1}
		Under the assumptions of Proposition \ref{apriori-estimate}, there exist positive constant $C$ that is independent of $\delta_S$ and $\e_1$, such that for $0 \le t \le T$,
		\begin{equation} \label{eq: est-1}
		\begin{aligned}
		&\norm{w-\tw}_{L^2(\mathbb{R})}^2+  \int_\mathbb{R} (u-\tu)(w-\tw) \, dx  +  \int_0^t \left(G_w + D_{w_1} \right) \, d s \\ 
		&\le  C \norm{w_0-\tw(0,\cdot)}_{L^2(\mathbb{R})}^2 + C \int_\mathbb{R} (u_0-\tu(0,\cdot))(w_0-\tw(0,\cdot)) \, dx \\
		&\quad + \int_\mathbb{R} | \tw (v-\tv) (w-\tw) |\, dx + C \int_\mathbb{R} \left|  \tw(0,\cdot) (v_0-\tv(0,\cdot)) (w_0-\tw(0,\cdot)) \right|\, dx \\
		& \quad + C \delta_S | \dot{X}|^2 + C \int_0^t D_{u_1} \, d s +C \delta_S \int_0^t \left( \| (v-\tv)_x \|_{L^2 (\mathbb{R})}^2 + D_{u_2}+G_1+ G^S \right) \, d s .
		\end{aligned}
		\end{equation}
	\end{lemma}

	\begin{proof}
		Recall that the perturbation $(\phi,\psi,\omega)$ satisfies \eqref{Underlying wave with shift}:
		\begin{equation}\label{Underlying wave NSK-w}
		\begin{cases}
		\begin{aligned}
		&\phi_t-\psi_x=\dot{X} \tv_x,\\
		&\psi_t-\left(p'(v)v^{5/2}w-p'(\widetilde{v})\widetilde{v}^{5/2}\widetilde{w} \right)=\left( \frac{u_x}{v}-\frac{\tu_x}{\tv} \right)_x + \left( \frac{w_x}{v^{5/2}} - \frac{\tw_x}{\tv^{5/2}} \right)_x + \dot{X} \tu_x,\\
		&\omega_t=-\left( \frac{u_x}{v^{5/2}}-\frac{\tu_x}{\tv^{5/2}}\right)_x+\dot{X} \tw_x.
		\end{aligned}
		\end{cases}
		\end{equation}
		Multiplying $\eqref{Underlying wave NSK-w}_2$ and \eqref{Underlying wave NSK-w}$_3$ by $\omega$ and $\psi$ yields
		\begin{equation*} 
		\omega \psi_t - \left(p'(v)v^{5/2}w-p'(\widetilde{v})\widetilde{v}^{5/2}\widetilde{w} \right) \omega - \left( \frac{w_x}{v^{5/2}} - \frac{\tw_x}{\tv^{5/2}} \right)_x \omega = \left( \frac{u_x}{v}-\frac{\tu_x}{\tv} \right)_x \omega + \dot{X}  \tu_x \omega,
		\end{equation*}
		and
		\begin{align*}
		\omega_t \psi = -\left( \frac{u_x}{v^{5/2}}-\frac{\tu_x} {\tv^{5/2}} \right)_x \psi + \dot{X} \tw_x \psi
		\end{align*}
		respectively. Then, using the above equations, together with  \eqref{eq: P,F_1,F_2}$_3$ and the identities
		\begin{equation*} 
		\begin{aligned}
		p'(v)v^{5/2}w-p'(\tv)\tv^{5/2} \tw &=p'(v)v^{5/2} \omega + \tw \left( p'(v)v^{5/2} - p'(\tv)\tv^{5/2} \right), \\ 
		\left( \frac{u_x}{v^{5/2}}-\frac{\tu_x}{\tv^{5/2}}\right)_x &=\left( \frac{\psi_x}{v^{5/2}}\right)_x-\left( \tu_x \left( \frac{1}{v^{5/2}}-\frac{1}{\tv^{5/2}}\right)\right)_x ,
		\end{aligned}
		\end{equation*}
		we obtain
		\begin{align*}
		&(\omega \psi)_t  -p'(v)v^{5/2} \omega^2 - \tw \left( p'(v)v^{5/2}-p'(\tv)\tv^{5/2} \right) \omega -\left( \frac{\omega_x }{v^{5/2}} \right)_x \omega -\left( \tw_x \left(\frac{1}{v^{5/2}}-\frac{1}{\tv^{5/2}} \right) \right)_x \omega \\ 
		&\quad = -\left( \frac{\psi_x}{v^{5/2}}\right)_x \psi- \left( \tu_x \left(\frac{1}{v^{5/2}}-\frac{1}{\tv^{5/2}} \right)\right)_x \psi + \left( \frac{\psi_x}{v} \right)_x \omega - \left( \tu_x \left( \frac{1}{v} - \frac{1}{\tv}\right) \right)_x \omega + \dot{X} \left( \tw_x \psi+ \tu_x \omega \right).
		\end{align*} 
		We integrate the equation and then use integration-by-parts to derive
		\begin{equation} \label{computation 1}
		\begin{aligned}
		\frac{d}{dt}& \int_\mathbb{R} \omega \psi \, dx  -\int_\mathbb{R} p'(v)v^{5/2} \omega^2 \, dx + \int_\mathbb{R} \frac{\omega_x^2}{v^{5/2}} \, dx \\ 
		&=\dot{X} \int_\mathbb{R} (\tw_x \psi + \tu_x \omega) \, dx+\int_\mathbb{R} \left( \frac{\psi_x}{v}\right)_x \omega \, dx \\
		&\quad+\int_\mathbb{R} \tw \left( p'(v)v^{5/2}-p'(\tv)\tv^{5/2} \right) \omega \, dx -\int_\mathbb{R} \tw_x \left( \frac{1}{v^{5/2}} -\frac{1}{\tv^{5/2}} \right) \omega_x \, dx \\ 
		&\quad+\int_\mathbb{R} \frac{\psi_x^2}{v^{5/2}} \, dx + \int_\mathbb{R}  \tu_x \left(\frac{1}{v^{5/2}}-\frac{1}{\tv^{5/2}} \right) \psi_x \, dx  - \int_\mathbb{R} \tu_x \left( \frac{1}{v}-\frac{1}{\tv} \right) \omega_x \, dx .
		\end{aligned}
		\end{equation}
		On the other hand, using the equation $\eqref{Underlying wave NSK-w}_1$ and
		\begin{equation*}
		(v-\tv)_x =-\left(v^{5/2} w - \tv^{5/2} \tw \right)=-\left( v^{5/2} \omega - \tw \left(v^{5/2} - \tv^{5/2}  \right) \right),
		\end{equation*}
		we express the second term on the right-hand side of  \eqref{computation 1} as 
		\begin{equation} \label{computation 2}
		\begin{aligned}
		\int_\mathbb{R} \left( \frac{\psi_x}{v}\right)_x \omega \, dx &=\int_\mathbb{R} \frac{\psi_{xx}}{v} \omega \, dx+\int_\mathbb{R} \psi_x \left( \frac{1}{v} \right)_x \omega \, dx \\
		&=\int_\mathbb{R} \frac{\phi_{tx}}{v} \omega \, dx - \dot{X} \int_\mathbb{R} \frac{ \tv_{xx}}{v} \omega \, dx +\int_\mathbb{R} \psi_x \left( \frac{1}{v} \right)_x \omega \, dx \\ 
		&=-\int_\mathbb{R} \frac{1}{v} \left( v^{5/2} \omega -\tw(v^{5/2}-\tv^{5/2}) \right)_t \omega \, dx - \dot{X} \int_\mathbb{R} \frac{ \tv_{xx}}{v} \omega \, dx +\int_\mathbb{R} \psi_x \left( \frac{1}{v} \right)_x \omega \, dx  \\
		&=- \frac{1}{2}\int_\mathbb{R} v^{3/2} \frac{d}{dt}\omega^2 \, dx -\frac{5}{2} \int_\mathbb{R}   v^{1/2} v_t \omega^2 \, dx -\int_\mathbb{R}\frac{1}{v} \left( (v^{5/2}-\tv^{5/2}) \tw \right)_t \omega \, dx \\ 
		&\quad - \dot{X} \int_\mathbb{R} \frac{ \tv_{xx}}{v} \omega \, dx +\int_\mathbb{R} \psi_x \left( \frac{1}{v} \right)_x \omega \, dx \\
		&=-\frac{1}{2} \frac{d}{dt} \int_\mathbb{R} v^{3/2} \omega^2 \, dx -\frac{7}{4} \int_\mathbb{R}    v^{1/2} v_t \omega^2 \, dx -\int_\mathbb{R}\frac{1}{v} \left( (v^{5/2}-\tv^{5/2}) \tw \right)_t \omega \, dx \\
		&\quad - \dot{X} \int_\mathbb{R} \frac{ \tv_{xx}}{v} \omega \, dx +\int_\mathbb{R} \psi_x \left( \frac{1}{v} \right)_x \omega \, dx .\\
		\end{aligned}
		\end{equation}
		We combine \eqref{computation 1} and \eqref{computation 2} to obtain
		\begin{equation*} 
		\begin{aligned}
		&\frac{1}{2} \frac{d}{dt} \int_\mathbb{R} v^{3/2} \omega^2 \, dx +\frac{d}{dt} \int_\mathbb{R} \omega \psi \, dx +\mathcal{G}_w +\mathcal{D}_{w_1} = \sum_{i=1}^9 S_i, 
		\end{aligned}
		\end{equation*}
		where
		\begin{align*}
		\mathcal{G}_w :=-\int_\mathbb{R} p'(v)v^{5/2} \omega^2 \, dx,\quad\mathcal{D}_{w_1} :=  \int_\mathbb{R} \frac{\omega_x^2}{v^{5/2}} \, dx
		\end{align*}
		and
		\begin{align*}
		&S_1 :=\dot{X}(t)\int_\mathbb{R}  \left( \tw_x \psi + \tu_x \omega-\frac{\tv_{xx}}{v}  \omega \right) \, dx,\quad S_2 := -\frac{7}{4} \int_\mathbb{R}   v^{1/2} v_t \omega^2   \, dx,\\ 
		&S_3 := - \int_\mathbb{R}  \frac{1}{v} \left( (v^{5/2}-\tv^{5/2}) \tw \right)_t \omega  \, dx,\quad S_4 :=  \int_\mathbb{R}  \tw \left( p'(v)v^{5/2}-p'(\tv)\tv^{5/2} \right) \omega  \, dx,\\
		&S_5 := - \int_\mathbb{R}  \tw_x \left( \frac{1}{v^{5/2}} -\frac{1}{\tv^{5/2}} \right) \omega_x  \, dx,\quad S_6 := - \int_\mathbb{R}  \tu_x \left( \frac{1}{v}-\frac{1}{\tv} \right) \omega_x  \, dx,\\ 
		&S_7 :=  \int_\mathbb{R} \frac{\psi_x^2}{v^{5/2}} \, dx,\quad S_8 :=  \int_\mathbb{R}  \tu_x \left(\frac{1}{v^{5/2}}-\frac{1}{\tv^{5/2}} \right) \psi_x \, dx,\quad S_9 :=  \int_\mathbb{R} \psi_x \left( \frac{1}{v} \right)_x \omega  \, dx. 
		\end{align*}
		In the following, we estimate each $S_i$ term by term.\\
		
		\noindent $\bullet$ (Estimate of $S_1$): We use H\"older's inquality, $|\tw_x| \le C \left( |\tv_x|^2 +|\tv_{xx}| \right) \le C \delta_S |\tv_x| $, and Young's inequality to estimate $S_1$ as
		\begin{align*}
		S_1 &\le C  |\dot{X}| \left( \sqrt{ \int_\mathbb{R} |\tw_x| \, dx}\sqrt{\int_\mathbb{R} |\tw_x| |\psi|^2 \, dx }+\sqrt{ \int_\mathbb{R} |\tu_x|^2 \, dx}\sqrt{\int_\mathbb{R} |\omega|^2 \, dx}+\sqrt{ \int_\mathbb{R} |\tv_{xx}|^2 \, dx}\sqrt{\int_\mathbb{R} |\omega|^2 \, dx} \right) \\
		&\le C \delta_S |\dot{X}| \left(\sqrt{\int_\mathbb{R} |\tv _x| |\psi|^2 \, dx }+\sqrt{\int_\mathbb{R} |\omega|^2 \, dx}\right) \le \frac{\delta_S}{2} |\dot{X}|^2 + C \delta_S \left( G^S+\mathcal{G}_w \right) \\ 
		&\le \frac{\delta_S}{2}|\dot{X}|^2 +\frac{1}{8} \mathcal{G}_w + C \delta_S G^S.
		\end{align*}
		\noindent$\bullet$ (Estimate of $S_2$):
		We estimate $S_2$ by using $v_t=u_x=\psi_x + \tu_x$ as
		\begin{align*}
		S_2 &\le C \int_\mathbb{R} |u_x| |\omega|^2 \, dx \le C \int_\mathbb{R} |\psi_x| |\omega|^2 \, dx + C \int_\mathbb{R} |\tu_x| |\omega|^2 \, dx .
		\end{align*}
		Then, we use
		\begin{equation*}
		|w-\tw| = \left|v^{-5/2} \left(v_x-\tv_x\right) +\tv_x \left(v^{-5/2}-\tv^{-5/2} \right) \right| \le C  \left( | (v-\tv)_x| + |\tv_x| |v-\tv| \right) 
		\end{equation*}
		and \eqref{eq: L^infty bounds} to obtain
		\begin{align*}
		S_2 &\le C  \int_\mathbb{R} |\psi_x| \left( | \phi_x| + |\tv_x| |\phi| \right)  |\omega| \, dx + C \delta_S^2 \int_\mathbb{R} |\omega|^2 \, dx\\
		&\le \int_\mathbb{R} |\psi_x| |\phi_x| |\omega| \, dx + C \int_\mathbb{R} |\psi_x| |\tv_x| |\phi| |\omega| \, dx + C \delta_S^2 \mathcal{G}_w\\
		&\le C \varepsilon_1 \left( \int_\mathbb{R} |\psi_x| |\omega| \, dx + \int_\mathbb{R} |\tv_x| |\psi_x| |\omega| \, dx \right) + C \delta_S^2 \mathcal{G}_w\\
		&\le C \varepsilon_1 \left( \int_\mathbb{R} |\psi_x|^2 \, dx + \int_\mathbb{R}  |\omega|^2 \, dx   \right) + C \delta_S^2 \mathcal{G}_w \le C\varepsilon_1 \left( D_{u_1} + \mathcal{G}_w \right) +C \delta_S^2 \left(\mathcal{G}_w \right) \\ 
		&\le \frac{1}{8} \mathcal{G}_w +  C \varepsilon_1 D_{u_1}.
		\end{align*}
		\noindent $\bullet$ (Estimate of $S_3$): We split $S_3$ into $S_{3,1}$ and $S_{3,2}$ as  
		\begin{align*}
		S_3&=-\frac{d}{dt} \int_\mathbb{R} \frac{1}{v} (v^{5/2}-\tv^{5/2})\tw \omega \, dx + \int_\mathbb{R} (v^{5/2}-\tv^{5/2}) \tw \left( \frac{1}{v} \omega \right)_t \,dx\\
		&=:S_{3,1}+S_{3,2} .
		\end{align*}
		Since $S_{3,1}$ has the form of time-derivative of an integral over $\mathbb{R}$, we remain as it is.
		On the other hand, to estimate $S_{3,2}$, we further split $S_{3,2}$ into $S_{3,21}$ and $S_{3,22}$ as 
		\begin{align*}
		S_{3,2} &=   \int_\mathbb{R} (v^{5/2}-\tv^{5/2}) \tw \left( \frac{1}{v} \right)_t \omega  \,dx +\int_\mathbb{R} (v^{5/2}-\tv^{5/2}) \tw  \frac{1}{v}  \omega_t \,dx\\ 
		&=: S_{3,21} + S_{3,22}.
		\end{align*}
		For $S_{3,21}$, we use $|\tw| \le C |\tv_x|$, $|v_t|=|u_x| \le |\psi_x|+|\tu_x|$, and $\|\phi \|_{L^\infty(\mathbb{R})}+\|\phi_x \|_{L^\infty(\mathbb{R})}+\|\psi \|_{L^\infty(\mathbb{R})} \le C\varepsilon_1 $ to have
		\begin{align*}
		S_{3,21} &\le C \int_\mathbb{R} |\tv_x| |\phi| |\psi_x| |\omega| \, dx+ C \int_\mathbb{R} |\tv_x| |\phi|  |\tu_x| |\omega| \, dx \\ 
		&\le C \varepsilon_1 \left( \int_\mathbb{R} |\tv_x||\psi_x|^2 \, dx + \int_\mathbb{R} |\tv_x| |\omega|^2 \, dx \right)+C \left( \int_\mathbb{R} |\tv_x|^2|\phi|^2 \, dx + \int_\mathbb{R} |\tu_x|^2 |\omega|^2 \, dx \right) \\ 
		&\le C \varepsilon_1 \delta_S^2 (D_{u_1}+\mathcal{G}_w) + C \delta_S^2 ( G_1 + G^S + \mathcal{G}_w) \\ 
		&\le \frac{1}{8} \mathcal{G}_w + C \delta_S^2 (D_{u_1} + G_1 +G^S).
		\end{align*}
		To estimate $S_{3,22}$, we first observe that
		\begin{align*}
		\omega_t &=-\left( \frac{u_x}{v^{5/2}}-\frac{\tu_x}{\tv^{5/2}} \right)_x + \dot{X} \tw_x =-\left( \frac{\psi_x}{v^{5/2}}  + \tu_x \left( \frac{1}{v^{5/2}}-\frac{1}{\tv^{5/2}} \right) \right)_x + \dot{X} \tw_x.
		\end{align*}
		Then, by using $v_x = \phi_x + \tv_x$, and $|\tu_{xx}| \le C \delta_S |\tu_x|$ and $|\tu_x| \sim |\tv_x|$ in Lemma \ref{lem:shock-property}, we can bound $|\omega_t|$ as
		\begin{equation}\label{eq: omega_t}
		\begin{aligned}
		|\omega_t| &\le C \left( |v_x| |\psi_x| + |\psi_{xx}| + |\tu_{xx}| |\phi| + |\tu_x|( |\phi_x| +|\tv_x||\phi|) + |\dot{X}| |\tw_x| \right) \\
		&\le C  \left( |\phi_x| |\psi_x| + |\tv_x| \left(| \phi_x|+|\psi_x| \right) + |\psi_{xx}| +  |\tv_x| |\phi| + |\dot{X}| |\tw_x| \right) . 
		\end{aligned}
		\end{equation}
		Then, $S_{3,22}$ is estimated by using $|\tw| \le |\tv_x|$ and Young's inequality as 
		\begin{align*}
		S_{3,22} &\le C \int_\mathbb{R} |\phi| |\tw| |\omega_t| \, dx \le C \int_\mathbb{R} |\phi| |\tv_x| |\omega_t| \, dx \\
		&\le C \int_\mathbb{R} |\tv_x|^{3/2} |\phi|^2 \, dx +  C \int_\mathbb{R} |\tv_x|^{1/2} |\omega_t|^2 \, dx=: S_{3,221}+S_{3,222}. 
		\end{align*}
		We estimate $S_{3,221}$ by using Lemma \ref{lem: Useful bad terms estimates} as
		\begin{equation} \label{eq: cS_21}
		\begin{aligned}
		S_{3,221} \le C \delta_S \left( G_1 +G^S \right).
		\end{aligned}
		\end{equation}
		On the other hand, we estimate $S_{3,222}$ by using \eqref{eq: omega_t} and Lemma \ref{lem: Useful bad terms estimates} as 
		\begin{align*}
		S_{3,222} \le C \norm{\tv_x}_{L^\infty(\mathbb{R})}^{1/2} &\Big( \norm{\phi_x}_{L^\infty(\mathbb{R})}^2 \norm{\psi_x}_{L^2(\mathbb{R})}^2+ \norm{\tv_x}_{L^\infty(\mathbb{R})}^2 \big(\norm{\phi_x}_{L^2(\mathbb{R})}^2 + \norm{\psi_x}_{L^2(\mathbb{R})}^2 \big)  \\ 
		& \quad +\delta_S (G_1 + G^S) + |\dot{X}|^2 \norm{\tw_x}_{L^2(\mathbb{R})}^2 \Big).
		\end{align*}
		In addition, we use \eqref{eq: L^infty bounds} and $|\tw_x| \le C (|\tv_x|^2 +|\tv_{xx}|)\le C \delta_S |\tv_x|$ in Lemma \ref{lem:shock-property} to obtain 
		\begin{equation} \label{eq: cS_22 }
		\begin{aligned}
		S_{3,222}
		& \le C \delta_S \left( ( \varepsilon_1^2 +\delta_S^4) D_{u_1}+ \delta_S^4  \norm{\phi_x}_{L^2(\mathbb{R})}^2 +D_{u_2} + \delta_S^2 (G_1+G^S)+  | \dot{X}|^2 \norm{\tv_x}_{L^2(\mathbb{R})}^2 \right)\\ 
		&\le  C \delta_S^2 |\dot{X}|^2 + C \delta_S \left(\norm{\phi_x}_{L^2(\mathbb{R})}^2+ D_{u_1} +D_{u_2} +G_1 +G^S \right).
		\end{aligned}
		\end{equation}
		Combining the above two estimates \eqref{eq: cS_21} and \eqref{eq: cS_22 }, we have
		\begin{align*}
		S_{3,22} \le \frac{\delta_S}{2} |\dot{X}|^2 + C \delta_S \left( \|\phi_x\|_{L^2(\mathbb{R})}^2 + D_{u_1}+D_{u_2} +G_1 + G^S   \right).
		\end{align*}
		Therefore, we estimate $S_{3,2}$ as
		\begin{align*}
		S_{3,2} \le \frac{\delta_S}{2} |\dot{X}|^2 + \frac{1}{8} \mathcal{G}_w + C \delta_S \left( \|\phi_x\|_{L^2(\mathbb{R})}^2 + D_{u_1}+D_{u_2} +G_1 + G^S   \right).
		\end{align*}
		\noindent $\bullet$ (Estimate of $S_4$, $S_5$, and $S_6$): For $S_4$, we observe that
		\begin{align*}
		\left|p'(v)v^{5/2}-p'(\tv)\tv^{5/2} \right| &\le |p'(v)| \left|v^{5/2}-\tv^{5/2} \right| + \big|\tv^{5/2} \big| \left|p'(v)-p'(\tv) \right| \\ 
		&\le C \left|v-\tv_x \right| + C \left|\tv_x\right| \left|v-\tv\right|.
		\end{align*}
		Then, using the above estimate, $|\tw| \le C |\tv_x| $, and Lemma \ref{lem: Useful bad terms estimates}, we estimate $S_4$ as
		\begin{align*}
		S_4 &\le C \int_\mathbb{R} |\tv_x| \left(|\phi_x|+|\tv_x||\phi|\right) |\omega| \, dx \le C  \int_\mathbb{R} |\tv_x| |\phi_x| |\omega| \, dx + C \int_\mathbb{R} |\tv_x|^2 |\phi| |\omega| \,dx \\ 
		&\le C \int_\mathbb{R} |\tv_x| |\phi_x|^2 \, dx + C \int_\mathbb{R} |\tv_x| |\omega|^2 \, dx + C \int_\mathbb{R} |\tv_x|^2 |\phi|^2 \, dx \\
		&\le C \delta_S^2 \left( \norm{\phi_x}_{L^2 (\mathbb{R})}^2 + G_1 + G^S + \mathcal{G}_w \right) \\ 
		&\le \frac{1}{8} \mathcal{G}_w +C \delta_S^2 (\norm{\phi_x}_{L^2 (\mathbb{R})}^2+G_1 +G^S).
		\end{align*}
		Similarly, we use $|\tw_x| \le C (|\tv_x|^2 + |\tv_{xx}|) \le C \delta_S |\tv_x| $ and Lemma \ref{lem: Useful bad terms estimates} to estimate $S_5$ as
		\begin{align*}
		S_5&\le C \int_{\mathbb{R}} |\tw_x| |\phi| | \omega_x|  \, dx \le  C \delta_S \int_{\mathbb{R}} |\tv_x| |\phi| |\omega_x| \, dx \\
		&\le C \delta_S \left( \int_\mathbb{R} |\tv_x| |\phi|^2 \, dx + C \int_\mathbb{R} |\tv_x| |\omega_x| \, dx \right) \le C \delta_S \left( G_1 +G^S + \mathcal{D}_{w_1} \right)\\
		&\le \frac{1}{4} \mathcal{D}_{w_1} + C \delta_S (G_1 + G^S).
		\end{align*}
		Finally, we estimate $S_6$ by using Young's inequality and Lemma \ref{lem: Useful bad terms estimates} as
		\begin{align*}
		S_{6}&\le \int_\mathbb{R} |\tu_x| |\phi| |\omega_x| \, dx \le C \int_\mathbb{R} |\tu_x|^{3/2} |\phi|^2 \, dx + C \int_\mathbb{R} |\tu_x|^{1/2} |\omega_x|^2 \, dx\\
		&\le C \delta_S \left( G_1 +G^S + \mathcal{D}_{w_1} \right) \\ 
		&\le \frac{1}{4} \mathcal{D}_{w_1} + C \delta_S (G_1 + G^S).
		\end{align*}
		\noindent $\bullet$ (Estimate of $S_7$, $S_8$, and $S_9$): We simply estimate $S_7$ as
		\begin{align*}
		S_7 \le C \int_\mathbb{R} |\psi_x|^2 \, dx.
		\end{align*}
		To estimate $S_8$, we use Young's inequality and Lemma \ref{lem: Useful bad terms estimates} as
		\begin{align*}
		S_8 &\le C\int_\mathbb{R} |\tu_x| |\phi| |\psi_x| \, dx \le C\int_\mathbb{R} |\tu_x|^{3/2} |\phi|^2 \, dx + C\int_\mathbb{R} |\tu_x|^{1/2} |\psi_x|^2 \,dx \\ 
		& \le C \delta_S \left( G_1+G^S +  D_{u_1} \right).
		\end{align*}
		We estimate $S_9$ by using $v_x= \phi_x + \tv_x$ and \eqref{eq: L^infty bounds} as
		\begin{align*}
		S_9 & \le C \int_{\mathbb{R}} |v_x| |\psi_x|  |\omega| \, dx 
		\le \int_{\mathbb{R}} | \phi_x|  |\psi_x| | \omega | \, dx + \int_{\mathbb{R}} |\tv_x| |\psi_x |  |\omega | \, dx \\ 
		&\le C\left( \norm{\phi_x}_{L^\infty(\mathbb{R})} \norm{\psi_x}_{L^2 (\mathbb{R})} \norm{\omega}_{L^2(\mathbb{R})}+\norm{\tv_x}_{L^\infty(\mathbb{R})}  \norm{\psi_x}_{L^2 (\mathbb{R})} \norm{\omega}_{L^2(\mathbb{R})} \right)\\
		&\le C (\varepsilon_1 + \delta_S^2 ) \left(  D_{u_1}  + \mathcal{G}_w \right) \\ 
		&\le \frac{1}{8} \mathcal{G}_w + C (\varepsilon_1 + \delta_S^2) D_{u_1}.
		\end{align*}
		Therefore, combining all estimates for $S_i$ $i=1,\ldots,9,$ and using the smallness of $\delta_S$ and $\varepsilon_1$, there exist positive constant $C$ such that 
		\begin{align*}
		&\frac{1}{2}\frac{d}{dt} \int_\mathbb{R} v^{3/2} \omega^2 \, dx + \frac{d}{dt} \int_\mathbb{R} \omega \psi \, dx + \frac{d}{dt} \int_\mathbb{R} \frac{1}{v} (v^{5/2}-\tv^{5/2}) \tw \omega \, dx + \frac{1}{8} \left( \mathcal{G}_w + \mathcal{D}_{w_1} \right) \\ 
		&\le \delta_S | \dot{X}|^2 + C \|\psi_x \|_{L^2(\mathbb{R})}^2 
		+ C  \delta_S \left( \| \phi_x \|_{L^2 (\mathbb{R})}^2 +D_{u_2} +G_1 + G^S \right) .
		\end{align*}
		Integrating the above estimate over $[0,t]$ for any $t \le T$, we obtain 
		\begin{align}
		\begin{aligned}\label{est-vw}
		&\int_\mathbb{R} \frac{v^{3/2} \omega^2}{2} \, dx +  \int_\mathbb{R} \psi \omega \, dx + \int_\mathbb{R} \frac{1}{v} (v^{5/2} -\tv^{5/2}) \tw \omega \, dx + \frac{1}{8} \int_0^t \left(\mathcal{G}_w + \mathcal{D}_{w_1} \right) \, d s \\ 
		&\le  \int_\mathbb{R} \frac{v_0^{3/2} \omega_0^2}{2} \, dx +  \int_\mathbb{R} \psi_0 \omega_0 \, dx + \int_\mathbb{R} \frac{1}{v_0} (v_0^{5/2} -\tv(0,\cdot)^{5/2}) \tw(0,\cdot) \omega_0 \, dx \\
		& \quad + \delta_S \int_0^t| \dot{X}(s)|^2\,d s + C \int_0^t D_{u_1} \, d s +C \delta_S \int_0^t \left( \| (v-\tv)_x \|_{L^2 (\mathbb{R})}^2 +D_{u_2} +G_1+ G^S \right) \, d s. 
		\end{aligned}
		\end{align}
		However, we note that the estimates $G_{w} \le C \mathcal{G}_{w}$, $D_{w_1} \le C \mathcal{D}_{w_1}$, and
		\begin{align*}
		\int_\mathbb{R} |\omega|^2 \, dx \le C \int_\mathbb{R} \frac{v^{3/2} |\omega|^2}{2} \, dx \quad \text{and} \quad -\int_\mathbb{R} |\phi \tw  \omega| \, dx \le C \int_\mathbb{R} \frac{1}{v}(v^{5/2}-\tv^{5/2})\tw \omega \, dx
		\end{align*} 
		hold. Using these estimates to \eqref{est-vw}, we can complete the proof of Lemma \ref{lem: est-1}.
	\end{proof}

	\begin{lemma}\label{lem: est-2}
		Under the assumptions of Proposition \ref{apriori-estimate}, there exists a positive constant $C$ that is independent of $\delta_S$ and $\e_1$, such that for $0 \le t \le T$
		\begin{equation}\label{eq: est-2}
		\begin{aligned}
		&\int_{\mathbb{R}} (u-\tu)_x(w-\tw)_x \, dx+ \int_{0}^t \left( D_{w_1} + D_{w_2} \right) \, ds 
		\\
		&\le \int_{\mathbb{R}} (u_0-\tu(0,\cdot))_x(w_0-\tw(0,\cdot))_x \, dx+C \delta_S \int_0^t |\dot{X}(s)|^2 \, d s + C \int_{0}^{t} D_{u_2} ds \\ 
		&\quad +  C \delta_S \int_0^t \left(\norm{(v-\tv)_x}_{L^2(\mathbb{R})}^2+G_1 + G^S \right) \, d s+C(\varepsilon_1+\delta_S^2) \int_0^t \left(D_{u_1}+\norm{w-\tw}_{L^2(\mathbb{R})}^2 \right) \, d s.
		\end{aligned}
		\end{equation}
		
	\end{lemma}
	\begin{proof}
		Differentiating $\eqref{Underlying wave NSK-w}_2$ with respect to $x$ and then multiplying the result by $\omega_x$, we have
		\begin{equation} \label{psi_tx omega_x}
		\begin{aligned}
		& \psi_{tx} \omega_x - \left( p'(v)v^{5/2} \omega + \tw \left( p'(v)v^{5/2}-p'(\tv)\tv^{5/2} \right) \right)_x \omega_x-\left( \frac{\omega_x}{v^{5/2}} + \tw_x \left( \frac{1}{v^{5/2}} - \frac{1}{\tv^{5/2}} \right) \right)_{xx} \omega_x \\
		&\; =\left( \frac{\psi_x}{v}+\tu_x \left( \frac{1}{v} - \frac{1}{\tv} \right) \right)_{xx} \omega_x + \dot{X} \tu_{xx} \omega_x .
		\end{aligned}
		\end{equation}
		Similarly, we differentiate $\eqref{Underlying wave NSK-w}_3$ with respect to $x$ and multiply the result by $\psi_x$ to obtain
		\begin{equation} \label{omega_tx psi_x}
		\begin{aligned}
		& \omega_{tx} \psi_x = -\left( \frac{\psi_x}{v^{5/2}}+\tu_x \left( \frac{1}{v^{5/2}} - \frac{1}{\tv^{5/2}} \right) \right)_{xx} \psi_x + \dot{X} \tw_{xx} \psi_x.
		\end{aligned}
		\end{equation}
		Then, adding the equation \eqref{psi_tx omega_x} and \eqref{omega_tx psi_x}, we have
		\begin{equation*}
		\begin{aligned}
		& \left( \psi_x \omega_x \right)_t - \left( p'(v)v^{5/2} \omega + \tw \left( p'(v)v^{5/2}-p'(\tv)\tv^{5/2} \right) \right)_x \omega_x-\left( \frac{\omega_x}{v^{5/2}} + \tw_x \left( \frac{1}{v^{5/2}} - \frac{1}{\tv^{5/2}} \right) \right)_{xx} \omega_x \\
		&\; =\dot{X} \left( \tu_{xx} \omega_x + \tw_{xx} \psi_x \right) + \left( \frac{\psi_x}{v}+\tu_x \left( \frac{1}{v} - \frac{1}{\tv} \right) \right)_{xx} \omega_x  -\left( \frac{\psi_x}{v^{5/2}}+\tu_x \left( \frac{1}{v^{5/2}} - \frac{1}{\tv^{5/2}} \right) \right)_{xx} \psi_x .
		\end{aligned}
		\end{equation*}
		We integrate the above equation over $\mathbb{R}$  and using integration-by-parts to derive
		\begin{align*}
		\frac{d}{dt}\int_{\mathbb{R}}\psi_x\omega_x \, dx+\mathcal{G}_{w_2}+\mathcal{D}_{w_2}=\sum\limits_{i=1}^{15}T_i,
		\end{align*}
		where
		\begin{align*}
		\mathcal{G}_{w_2}:=-\int_{\mathbb{R}} p'(v)v^{\frac{5}{2}} |\omega_x|^2dx,\quad
		\mathcal{D}_{w_2}:=\int_{\mathbb{R}} \frac{|\omega_{xx}|^2}{v^{5/2}} dx,
		\end{align*}
		and
		\begin{align*}
		T_1&:=\dot{X} \int_\mathbb{R} \left( \tu_{xx} \omega_x+ \tw_{xx} \psi_x \right) \, dx, &
		T_2&:=\int_\mathbb{R} (p'(v)v^{5/2})_x \omega \omega_x \, dx, \\
		T_3&:=\int_\mathbb{R} \tw_x \left(p'(v)v^{5/2} -p'(\tv)\tv^{5/2} \right) \omega_x \, dx, &
		T_4&:=\int_\mathbb{R} \tw \left(p'(v)v^{5/2} -p'(\tv)\tv^{5/2} \right)_x \omega_x \, dx,\\
		T_5&:=-\int_\mathbb{R} \omega_x \left( \frac{1}{v^{5/2}} \right)_x \omega_{xx} \, dx,\\
		T_6&:=-\int_\mathbb{R} \tw_{xx} \left( \frac{1}{v^{5/2}}-\frac{1}{\tv^{5/2}}\right) \omega_{xx} \, dx, &
		T_7&:=-\int_\mathbb{R} \tw_x \left( \frac{1}{v^{5/2}} -\frac{1}{\tv^{5/2}}\right)_x \omega_{xx} \, dx,\\
		T_8&:=-\int_\mathbb{R} \frac{\psi_{xx}}{v} \omega_{xx} \, dx, &
		T_9&:=\int_\mathbb{R} \psi_x \left( \frac{1}{v} \right)_x \omega_{xx} \, dx, \\
		T_{10}&:=\int_{\mathbb{R}} \tu_{xx} \left( \frac{1}{v}-\frac{1}{\tv} \right)\omega_{xx} \, dx, &
		T_{11}&:=\int_\mathbb{R} \tu_x \left( \frac{1}{v}-\frac{1}{\tv} \right)_x \omega_{xx} \, dx, \\
		T_{12}&:=\int_\mathbb{R}  \frac{|\psi_{xx}|^2}{v^{5/2}} \, dx, &
		T_{13}&:=\int_\mathbb{R} \psi_{xx} \psi_x \left(\frac{1}{v^{5/2}}\right)_x \, dx, \\
		T_{14}&:=\int_\mathbb{R} \psi_{xx}\tu_{xx} \left( \frac{1}{v^{5/2}}-\frac{1}{\tv^{5/2}} \right) \, dx, &
		T_{15}&:=\int_\mathbb{R} \psi_{xx}\tu_x \left(\frac{1}{v^{5/2}}-\frac{1}{\tv^{5/2}}\right)_x  \, dx.
		\end{align*}
		Again, we estimate from $T_1$ to $T_{15}$ separately.\\

		\noindent $\bullet$ (Estimate of $T_1$): We estimate $T_1$ by using  $|\tw_{xx}| \le C (|\tv_x|^3 + |\tv_x||\tv_{xx}| + |\tv_{xxx}|) \le C \delta_S |\tv_x|$ and $|\tu_{xx}| \le C \delta_S |\tu_x|$ , and Young's inequality as
		\begin{align*}
		|T_1|
		&\le C \delta_S |\dot{X}| \left( \sqrt{\int_\mathbb{R} |\tu_x|^2 \, dx} \sqrt{\int_\mathbb{R} |\omega_x|^2 \, dx } + \sqrt{\int_\mathbb{R} |\tv_x|^2 \, dx} \sqrt{\int_\mathbb{R} |\psi_x|^2 \, dx } \, \right) \\ 
		&\le C \delta_S^{\frac{5}{2}} |\dot{X}| \left( \norm{\omega_x}_{L^2(\mathbb{R})} + \norm{\psi_x}_{L^2 (\mathbb{R})} \right) \\
		&\le \delta_S |\dot{X}|^2 + \frac{1}{8} \mathcal{G}_{w_2} + C \delta_S^4 D_{u_1}.
		\end{align*}
		\noindent $\bullet$ (Estimate of $T_2$): We use $|(p'(v)v^{5/2} )_x| \le C |v_x| \le C (|\phi_x| + |\tv_x|)$ and $\norm{\phi_x}_{L^\infty(\mathbb{R})} + \norm{\tv_x}_{L^\infty(\mathbb{R})} \le C (\varepsilon_1 + \delta_S^2) $ to estimate $T_2$ as
		\begin{align*}
		|T_2|&\le C \int_{\mathbb{R}} |v_x| |\omega| |\omega_x| \, dx\le C(\varepsilon_1+ \delta_S^2) \norm{\omega}_{L^2(\mathbb{R})} \norm{\omega_x}_{L^2(\mathbb{R})} \\
		&\le C(\varepsilon_1+ \delta_S^2) \left(\norm{\omega}_{L^2(\mathbb{R})}^2 + \norm{\omega_x}_{L^2(\mathbb{R})}^2 \right) \\ 
		&\le \frac{1}{8} \mathcal{G}_{w_2} +C(\varepsilon_1+ \delta_S^2)\norm{\omega}_{L^2(\mathbb{R})}^2 . \\
		\end{align*}
		\noindent $\bullet$ (Estimate of $T_3$ and $T_4$): We first note that the following inequalities hold:
		\begin{equation} \label{eq: nolinear estimate for T}
		\begin{aligned}
		&\left|p'(v)v^{5/2}-p'(\tv)\tv^{5/2} \right| \le C|v-\tv|, \\
		&\left|(p'(v)v^{5/2}-p'(\tv)\tv^{5/2} )_x \right| \le C\left(|(v-\tv)_x| + |\tv_x| |v-\tv|\right).\\
		\end{aligned}
		\end{equation}
		Then, we estimate $T_3$ using $\eqref{eq: nolinear estimate for T}_1$, $|\tw_x| \le C (|\tv_x|^2 + |\tv_{xx}|) \le C \delta_S |\tv_x|$ in Lemma \ref{lem:shock-property}, and Lemma \ref{lem: Useful bad terms estimates} as 
		\begin{align*}
		|T_3|&\le C\delta_S \int_{\mathbb{R}} |\tv_x| |\phi| |\omega_x| \, dx\le C\delta_S \left( \int_{\mathbb{R}} |\tv_x|^2|\phi|^2 \, dx + \int_\mathbb{R} |\omega_x|^2 \, dx \right)\\
		&\le \frac{1}{8} \mathcal{G}_{w_2} + C \delta_S^3 (G_1 + G^S).
		\end{align*}
		Similarly, we use $\eqref{eq: nolinear estimate for T}_2$, $|\tw| \le C |\tv_x|$, and Lemma \ref{lem: Useful bad terms estimates} to estimate $T_4$ as
		\begin{align*}
		|T_4|&\le C \int_{\mathbb{R}} |\tv_x| (|\phi_x|+|\tv_x||\phi|) |\omega_x| dx \le C \int_\mathbb{R} |\tv_x| |\phi_x|^2 \,dx + C \int_\mathbb{R} |\tv_x|^2 |\phi|^2 \, dx + C \int_\mathbb{R} |\tv_x| |\omega_x|^2 \, dx \\ 
		&\le \frac{1}{8} \mathcal{G}_{w_2} + C \delta_S^2 \left(\|\phi_x\|_{L^2(\mathbb{R})}^2 + G_1 + G^S \right).
		\end{align*}
		\noindent $\bullet$ (Estimate of $T_5$): We estimate $T_5$ by using $\left|\left(\frac{1}{v^{5/2}}\right)_x\right|\le C |v_x| \le C (|\phi_x| + |\tv_x|) $ and $\|\phi_x\|_{L^\infty(\mathbb{R})} + \|\tv_x\|_{L^\infty(\mathbb{R})} \le C (\varepsilon_1 +\delta_S^2)$ as
		\begin{align*}
		|T_5|&\le C \int_{\mathbb{R}} |v_x| |\omega_x| |\omega_{xx}| \, dx\le C(\varepsilon_1+ \delta_S^2) \left( \int_{\mathbb{R}} |\omega_x|^2 \, dx +\int_\mathbb{R} |\omega_{xx}|^2 \, dx \right)\\
		&\le \frac{1}{8} \mathcal{G}_{w_2} +\frac{1}{8} \mathcal{D}_{w_2}.
		\end{align*}
		\noindent $\bullet$ (Estimate of $T_6$ and $T_7$): We use \eqref{eq: (1/v^a-1/tv^a) and (1/v^a-1/tv^a)_x}, $|\tw_{xx}| \le C(|\tv_x|^3 + |\tv_x||\tv_{xx}| + |\tv_{xxx}|)\le C\delta_S^2 |\tv_x|$ in Lemma \ref{lem:shock-property}, and Lemma \ref{lem: Useful bad terms estimates} to estimate $T_6$ as
		\begin{align*}
		|T_6|&\le C\delta_S^2 \int_{\mathbb{R}} |\tv_x||\phi| |\omega_{xx}| \, dx \le C\delta_S^2 \left( \int_{\mathbb{R}} |\tv_x|^2|\phi|^2 \, dx + \int_\mathbb{R} |\omega_{xx}|^2 \, dx \right)\\
		&\le \frac{1}{8} \mathcal{D}_{w_2} + C\delta_S^4 \left(G_1+G^S \right).
		\end{align*}
		To estimate $T_7$, we use \eqref{eq: (1/v^a-1/tv^a) and (1/v^a-1/tv^a)_x}, $|\tw_x| \le C (|\tv_x|^2 + |\tv_{xx}| ) \le C \delta_S |\tv_x|$, and Lemma \ref{lem: Useful bad terms estimates} to derive
		\begin{align*}
		|T_7|&\le C\delta_S \int_{\mathbb{R}} |\tv_x| \left(|\phi_x|+|\tv_x||\phi|\right) |\omega_{xx}| \,dx  \\
		&\le C \int_\mathbb{R} |\tv_x| |\phi_x|^2 \, dx + C \int_\mathbb{R} |\tv_x|^2 |\phi|^2 \, dx + C \int_\mathbb{R} |\tv_x| |\omega_{xx}|^2 \, dx   \\
		&\le \frac{1}{8} \mathcal{G}_{w_2} + C \delta_S^2 \left(\|\phi_x\|_{L^2(\mathbb{R})}^2 + G_1 + G^S \right).
		\end{align*}
		\noindent $\bullet$ (Estimate of $T_8$ and $T_9$): By Young's inequality, there exists a positive constant $C>0$ such that
		\begin{align*}
		|T_8|
		&\le C \int_{\mathbb{R}} v^{1/2} |\psi_{xx}|^2 \, dx + \frac{1}{8} \int_\mathbb{R} \frac{|\omega_{xx}|^2}{v^{5/2}} \, dx \\
		&\le C D_{u_2}+\frac{1}{8} \mathcal{D}_{w_2}.
		\end{align*}
		We estimate $T_9$ by using $\left|\left(\frac{1}{v}\right)_x\right|\le C |v_x| \le C (|\phi_x| + |\tv_x|)$ and $\|\phi_x\|_{L^\infty(\mathbb{R})} + \|\tv_x\|_{L^\infty(\mathbb{R})} \le C (\varepsilon_1 + \delta_S^2)$, and Young's inequality as
		\begin{align*}
		|T_9|&\le C(\varepsilon_1+\delta_S^2) \int_{\mathbb{R}} |\psi_x| |\omega_{xx}| \, dx\le C(\varepsilon_1 + \delta_S^2 ) \left( \int_\mathbb{R} |\psi_x|^2 \, dx + \int_\mathbb{R} |\omega_{xx}|^2 \, dx \right) \\ 
		&\le \frac{1}{8} \mathcal{D}_{w_2} + C(\varepsilon_1 + \delta_S^2 ) D_{u_1}.
		\end{align*}
		\noindent$\bullet$ (Estimate of $T_{10}$ and $T_{11}$): By using $|\tu_{xx}|\le C\delta_S|\tu_x|$ and Lemma \ref{lem: Useful bad terms estimates}, we estimate $T_{10}$ as
		\begin{align*}
		|T_{10}|&\le C\delta_S \int_{\mathbb{R}} |\tu_x| |\phi| |\omega_{xx}|\, dx \le C\delta_S \left( \int_{\mathbb{R}} |\tu_x|^2|\phi|^2 \, dx + \int_\mathbb{R} |\omega_{xx}|^2 \, dx \right) \\
		&\le \frac{1}{8} \mathcal{D}_{w_2} + C\delta_S^2 \left(G_1+G^S \right).
		\end{align*}
		Also, we use $|\tu_x| \sim |\tv_x|$, $\left|(\frac{1}{v}-\frac{1}{\tv})_x \right| \le C (|\phi_x| + |\tv_x| |\phi| ) $, and Lemma \ref{lem: Useful bad terms estimates} to estimate $T_{11}$ as 
		\begin{align*}
		|T_{11}|&\le C  \int_{\mathbb{R}} |\tu_x|\left(|\phi_x|+|\tv_x||\phi|\right) |\omega_{xx}| dx \\
		&\le C \int_\mathbb{R} |\tu_x| |\phi_x|^2 \, dx + C \int_\mathbb{R} |\tv_x|^2 |\phi|^2 \, dx + \int_\mathbb{R} |\tv_x| |\omega_{xx}|^2 \, dx \\
		&\le \frac{1}{8} \mathcal{D}_{w_2} +  C \delta_S^2\left( \norm{\phi_x}_{L^2(\mathbb{R})}^2 + G_1+G^S \right).
		\end{align*}
		\noindent $\bullet$ (Estimate of $T_{12}$ and $T_{13}$): For $T_{12}$, we can find a positive constant $C>0$ such that
		\begin{align*}
		|T_{12}|\le C D_{u_2}.
		\end{align*}
		To estimate $T_{13}$, we use $|(\frac{1}{v^{5/2}})_x| \le C |v_x| \le C(|\phi_x| + |\tv_x|)$ and Young's inequality as
		\begin{align*}
		|T_{13}| &\le C \left( \norm{ \phi_x}_{L^\infty(\mathbb{R})} + \norm{\tv_x}_{L^\infty(\mathbb{R})} \right) \norm{\psi_x}_{L^2 (\mathbb{R})} \norm{\psi_{xx}}_{L^2 (\mathbb{R})} \\ 
		&\le C(\varepsilon_1 + \delta_S^2) \left(\lVert \psi_x \rVert^2_{L^2(\mathbb{R})}+\lVert \psi_{xx} \rVert^2_{L^2(\mathbb{R})}\right) \\ 
		&\le C(\varepsilon_1 + \delta_S^2) (D_{u_1} + D_{u_2}) .  
		\end{align*}
		\noindent $\bullet$ (Estimate of $T_{14}$ and $T_{15}$): We estimate $T_{14}$ by using \eqref{eq: (1/v^a-1/tv^a) and (1/v^a-1/tv^a)_x}, $|\tu_{xx}| \le C \delta_S |\tu_x|$ in Lemma \ref{lem:shock-property} and Lemma \ref{lem: Useful bad terms estimates} as
		\begin{align*}
		|T_{14}| &\le C\delta_S \int_\mathbb{R} |\tu_x| |\psi_{xx}| |\phi| \, dx \le C \delta_S \int_\mathbb{R} |\psi_{xx}|^2 \, dx + C \int_\mathbb{R} |\tu_x|^2 |\phi|^2 \, dx \\ 
		&\le C \delta_S D_{u_2} + C \delta_S^2 (G_1 + G^S).
		\end{align*}
		Finally, by using \eqref{eq: (1/v^a-1/tv^a) and (1/v^a-1/tv^a)_x}, $|\tv_x| \sim |\tv_x|$, and Lemma \ref{lem: Useful bad terms estimates}, we estimate $T_{15}$ as
		\begin{align*}
		|T_{15}|
		&\le C \int_\mathbb{R} |\tv_x| |\psi_{xx}| (|\phi_x| + |\tv_x||\phi|) \, dx \le C \int_\mathbb{R} |\tv_x| |\psi_{xx}|^2 \, dx + C \int_\mathbb{R} |\tv_x| |\phi_x| \, dx + \int_\mathbb{R} |\tv_x|^2 |\phi|^2 \, dx \\
		&\le C\delta_S^2\left( D_{u_2}+\lVert \phi_x \rVert^2_{L^2(\mathbb{R})}+ G_1+G^S \right).    
		\end{align*}
		Therefore, combining all the above estimates on $T_i$ for $i=1,2,\ldots, 15$ and using smallness of $\delta_S$ and $\e_1$, we get the following result
		\begin{align*}
		&\frac{d}{dt}\int_{\mathbb{R}} \psi_x\omega_x \, dx+\frac{1}{8}\left( \mathcal{G}_{w_2} + \mathcal{D}_{w_2} \right) \\
		&\le \delta_S |\dot{X}|^2 + C D_{u_2}+C \delta_S \left(\norm{\phi_x}_{L^2(\mathbb{R})}^2+G_1 + G^S \right)+C(\varepsilon_1+\delta_S^2) \left(D_{u_1}+\norm{\omega}_{L^2(\mathbb{R})}^2 \right).
		\end{align*}
		Integrating the above inequality over $[0,t]$ for any $t \le T$, we obtain 
		\begin{align*}
		&\int_{\mathbb{R}} (u-\tu)_x(w-\tw)_x \, dx+\frac{1}{8} \int_{0}^t \left( \mathcal{G}_{w_2} + D_{w_2} \right) \, ds 
		\\
		&\le \int_{\mathbb{R}} (u_0-\tu(0,\cdot))_x(w_0-\tw(0,\cdot))_x \, dx+ \frac{\delta_S}{2} \int_0^t |\dot{X}(s)|^2 \, d s + C \int_{0}^{t} D_{u_2} \, ds \\ 
		&\quad +  C \delta_S^2 \int_0^t \left(\norm{(v-\tv)_x}_{L^2(\mathbb{R})}^2+G_1 + G^S \right) \, d s+C(\varepsilon_1+\delta_S^2) \int_0^t \left(D_{u_1}+\norm{w-\tw}_{L^2(\mathbb{R})}^2 \right) \, d s.
		\end{align*}
		However, since the bound $D_{w_1} \le C \mathcal{G}_{w_2}$ and $D_{w_2} \le C \mathcal{D}_{w_2}$ hold, this completes the proof of Lemma \ref{lem: est-2}.
		
	\end{proof}

	\subsection{Combining Lemma \ref{lem: est-0}, Lemma \ref{lem: est-1}, and Lemma \ref{lem: est-2}}
	
	We now combine the estimates from Lemma Lemma \ref{lem: est-0}, Lemma \ref{lem: est-1}, and Lemma \ref{lem: est-2}. Since the right-hand side of each term is related to each other, we need to carefully combine those estimates, considering the order. We will summarize the steps as follows.\\
	
	\noindent $\bullet$ (Step 1): We combine Lemma \ref{lem: est-2} and Lemma \ref{lem: est-0} with an appropriate weight to derive the estimate \eqref{eq: est-3}.\\
	
	\noindent $\bullet$ (Step 2): Then, we combine \eqref{eq: est-3} from (Step 1) and Lemma \ref{lem: est-1}, again by considering an appropriate weights to derive \eqref{eq: est-4}.\\
	
	\noindent $\bullet$ (Step 3): Finally, we combine \eqref{eq: est-4} from (Step 2) and the $L^2$ estimates in Lemma \ref{Main Lemma} to conclude the desired $H^1$-estimate \eqref{eq:H1-est} holds. This completes the proof of Lemma \ref{Main Lemma 2}.\\
	
	In the following, we present the details of each step explained above.\\
	
	\noindent $\bullet$ (Step 1): We begin by expressing the inequality \eqref{eq: est-2} from Lemma \ref{lem: est-2} using the notation $(\phi,\psi,\omega)$ in \eqref{Notations: phi,psi,omega}. 
	\begin{equation*} 
	\begin{aligned}
	&\int_{\mathbb{R}} \psi_x \omega_x \, dx+ \int_{0}^t \left( \norm{\omega_x}_{L^2(\mathbb{R})}^2 + \norm{\omega_{xx}}_{L^2(\mathbb{R})}^2 \right) \, ds 
	\\
	&\le \int_{\mathbb{R}}\psi_{0x} \omega_{0x} \, dx+C \delta_S \int_0^t |\dot{X}|^2 \, d s + C_2 \int_{0}^{t} \norm{\psi_{xx}}_{L^2(\mathbb{R})}^2 \, ds \\ 
	&\quad +  C \delta_S \int_0^t \left(\norm{\phi_x}_{L^2(\mathbb{R})}^2+G_1 + G^S \right) \, d s+C(\varepsilon_1+\delta_S^2) \int_0^t \left( \norm{\psi_x}_{L^2(\mathbb{R})}^2+\norm{\omega}_{L^2(\mathbb{R})}^2 \right) \, d s,
	\end{aligned}
	\end{equation*}
	where $C$ and $C_2$ are positive constants which are independent of $\delta_S$ and $\varepsilon_1$. After rearranging the first terms $\int_\mathbb{R} \psi_x \omega_x \, dx$ on the left-hand side to the right-hand side, and applying Young's inequality, we obtain 
	\begin{equation} \label{eq: est-2,1}
	\begin{aligned}
	&\int_{0}^t \left( \norm{\omega_x}_{L^2(\mathbb{R})}^2 + \norm{\omega_{xx}}_{L^2(\mathbb{R})}^2 \right) \, ds 
	\\
	&\le \frac{1}{2} \norm{\psi_x}_{L^2(\mathbb{R})}^2+\frac12 \norm{\omega_x}_{L^2(\mathbb{R})}^2  + \frac12 \norm{\psi_{0x}}_{L^2(\mathbb{R})}^2+\frac12\norm{\omega_{0x}}_{L^2(\bbr)} \\
	&\quad +C \delta_S \int_0^t |\dot{X}|^2 \, d s + C_2 \int_{0}^{t} \norm{\psi_{xx}}_{L^2(\mathbb{R})}^2 \, ds \\ 
	&\quad +  C \delta_S \int_0^t \left(\norm{\phi_x}_{L^2(\mathbb{R})}^2+G_1 + G^S \right) \, d s+C(\varepsilon_1+\delta_S^2) \int_0^t \left( \norm{\psi_x}_{L^2(\mathbb{R})}^2+\norm{\omega}_{L^2(\mathbb{R})}^2 \right) \, d s.
	\end{aligned}
	\end{equation}
	Then, multiplying the inequality \eqref{eq: est-2,1} by $\frac{1}{2 \max(1,C_2)}$ and adding the inequality \eqref{eq: est-0} from Lemma \ref{lem: est-0}, we have
	\begin{equation*} 
	\begin{aligned}
	&\frac{1}{2} \norm{\psi_x}_{L^2(\mathbb{R})}^2+\frac12 \norm{\omega_x}_{L^2(\mathbb{R})}^2 +  \frac{1}{2} \int_0^t \norm{\psi_{xx}}_{L^2(\mathbb{R})}^2 \, d s + \frac{1}{2 \max(1,C_2)} \int_{0}^t  \norm{\omega_x}_{H^1(\mathbb{R})}^2 \, ds  \\
	&\le C \norm{\psi_{0x}}_{L^2(\mathbb{R})}^2+C \norm{\omega_{0x}}_{L^2(\mathbb{R})}^2   +C\delta_S \int_0^t |\dot{X}(s)|^2 \, d s + C \int_0^t \norm{\phi_x}_{L^2 (\mathbb{R})}^2 \, d s +C \delta_S \int_0^t  \left( G_1 + G^S \right) \, d s \\ 
	& \quad +C (\varepsilon_1+\delta_S) \int_0^t  \left( \norm{\phi_x}_{L^2(\mathbb{R})}^2+ \norm{\psi_x}_{L^2(\mathbb{R})}^2 + \norm{\omega}_{H^2(\mathbb{R})}^2  \right) \, d s .\\  
	\end{aligned}
	\end{equation*} 
	Finally, using the smallness of $\delta_S$ and $\varepsilon_1$ in the above inequality, we can arrive at the following conclusion:
	\begin{equation} \label{eq: est-3}
	\begin{aligned}
	&\norm{ \psi_x }_{L^2(\mathbb{R})}^2+\norm{ \omega_x }_{L^2(\mathbb{R})}^2+  \int_0^t \norm{\psi_{xx}}_{L^2(\mathbb{R})}^2 \, d s + \int_{0}^t  \norm{\omega_x}_{H^1(\mathbb{R})}^2 \, ds  \\
	&\le C \norm{\psi_{0x}}_{L^2(\mathbb{R})}^2+C\norm{\omega_{0x}}_{L^2(\R)}  +C\delta_S \int_0^t |\dot{X}(s)|^2 \, d s + C_3 \int_0^t \norm{\phi_x}_{L^2 (\mathbb{R})}^2 \, d s \\
	&\quad +C \delta_S \int_0^t  \left( G_1 + G^S \right) \, d s  +C (\varepsilon_1+\delta_S) \int_0^t  \left(  \norm{\psi_x}_{L^2(\mathbb{R})}^2 + \norm{\omega}_{L^2(\mathbb{R})}^2  \right) \, d s .\\  
	\end{aligned}
	\end{equation} 
	Here, we use the crucial fact that $C$ and $C_2$ are constants independent of $\delta_S$ and $\varepsilon_1$. \\
	
	\noindent $\bullet$ (Step 2): Expressing the inequality \eqref{eq: est-1} in Lemma \ref{lem: est-1} using the notation $(\phi, \psi, \omega)$, it reads as
	\begin{equation*} 
	\begin{aligned}
	&\norm{\omega}_{L^2(\mathbb{R})}^2+  \int_\mathbb{R} \psi \omega \, dx  +  \int_0^t \left( \norm{\omega}_{L^2(\mathbb{R})}^2 + \norm{\omega_x}_{L^2(\mathbb{R})}^2 \right) \, d s \\ 
	&\le  C \norm{\omega_0}_{L^2(\mathbb{R})}^2 + C \int_\mathbb{R} \psi_0 \omega_0 \, dx + \int_\mathbb{R} | \tw \phi \omega |\, dx + C \int_\mathbb{R}  \left|\tw(0,\cdot) \phi_0 \omega_0 \, \right|dx \\
	& \quad + C \delta_S \int_0^t | \dot{X}|^2 \, d s + C \int_0^t \norm{\psi_x}_{L^2(\mathbb{R})}^2 \, d s +C \delta_S \int_0^t \left( \norm{\phi_x}_{L^2 (\mathbb{R})}^2 +\norm{\psi_{xx}}_{L^2(\mathbb{R})}^2+G_1+ G^S \right) \, d s .
	\end{aligned}
	\end{equation*}
	After moving the second term $\int_\mathbb{R} \psi \omega \, dx$  from the left-hand side to the right-hand side in the above inequality, and applying Young's inequality, $|\tw| \le C |\tv_x|$, and Lemma \ref{lem:shock-property}, we obtain 
	\begin{equation} \label{eq: est-1,1}
	\begin{aligned}
	\norm{\omega}_{L^2(\mathbb{R})}^2&+  \int_0^t \left( \norm{\omega}_{L^2(\mathbb{R})}^2 + \norm{\omega_x}_{L^2(\mathbb{R})}^2 \right) \, d s \\ 
	&\le C\left(\norm{\phi_0}_{L^2(\mathbb{R})}^2 + \norm{\psi_0}_{L^2(\R)}^2+\norm{\omega_0}_{L^2(\R)}^2\right)+ \frac{1}{2} \left( \norm{ \psi}_{L^2(\mathbb{R})}^2 + \norm{\omega}_{L^2(\mathbb{R})}^2 \right) \\
	&\quad + C \delta_S \left( \norm{ \phi}_{L^2(\mathbb{R})}^2 + \norm{\omega}_{L^2(\mathbb{R})}^2 \right) + C \delta_S \int_0^t | \dot{X}|^2 \, ds+ C \int_0^t \norm{\psi_x}_{L^2(\mathbb{R})}^2 \, d s \\
	&\quad +C \delta_S \int_0^t \left( \norm{\phi_x}_{L^2 (\mathbb{R})}^2 +\norm{\psi_{xx}}_{L^2(\mathbb{R})}^2+G_1+ G^S \right) \, d s .
	\end{aligned}
	\end{equation}
	
	On the other hand, it follows from the definition of $\omega$ that 
	
	\begin{align*}
	\omega = w-\tw=-\frac{\phi_x}{v^{5/2}}-\tv_x\left(\frac{1}{v^{5/2}}-\frac{1}{\tv^{5/2}}\right).
	\end{align*}
	We use \eqref{eq: (1/v^a-1/tv^a) and (1/v^a-1/tv^a)_x} to obtain
	\[\|\phi_x\|_{L^2(\R)}\le C\|\omega\|_{L^2(\R)}+C\int_\R |\tv_x||v-\tv|\,dx,\]
	which implies that there exists a positive constant $c_*$ independent of $\delta_S$, $\varepsilon_1$ such that
	\begin{align*}
	\|\phi_x\|_{L^2(\R)}^2&\le \frac{1}{c_*}\|\omega\|_{L^2(\R)}^2+C\int_\R |\tv_x|\,dx\int_{\R}|\tv_x||p(v)-p(\tv)|^2\,dx\\
	&\le \frac{1}{c_*}\|\omega\|_{L^2(\R)}^2+C\delta_S(G_1+G^S),
	\end{align*}
	which yields
	\begin{equation} \label{eq: omega 2-norm inequality}
	\begin{aligned}
	c_* \norm{\phi_x}_{L^2(\mathbb{R})}^2\le \norm{\omega}_{L^2(\mathbb{R})}^2 + C\delta_S(G_1 +G^S).
	\end{aligned}
	\end{equation}
	Then, multiplying the inequality \eqref{eq: est-3} by $\frac{c_*}{4 \max(1,C_3)}$, adding the above inequality \eqref{eq: est-1,1}, and using the smallness of $\delta_S$, we obtain 
	\begin{equation*}
	\begin{aligned}
	&\frac{1}{4}\norm{\omega}_{L^2(\mathbb{R})}^2+\frac{c_*}{4 \max(1,C_3)} (\norm{\psi_x}_{L^2(\mathbb{R})}^2+\|\omega_x\|_{L^2(\R)}^2) \\
	&\quad +\frac{c_*}{8 \max(1,C_3)} \int_0^t \left( \norm{\psi_{xx}}_{L^2 (\mathbb{R})}^2 + \norm{\omega_x}_{L^2(\mathbb{R})}^2 + \norm{\omega_{xx}}_{L^2(\mathbb{R})}^2 \right) \, d s+ \frac{3}{4} \int_0^t \left( \norm{\omega}_{L^2(\mathbb{R})}^2 + \norm{\omega_x}_{L^2(\mathbb{R})}^2 \right) \, d s \\ 
	&\le C (\norm{\phi_0}_{L^2(\R)}^2+\norm{\psi_0}_{H^1(\R)}^2+\norm{\omega_0}_{H^1(\R)}^2) + \frac{1}{2} \norm{\phi}_{L^2(\mathbb{R})}^2+\frac12 \norm{\psi}_{L^2(\R)}^2 + C \delta_S \int_0^t | \dot{X}|^2 \, ds \\
	&\quad + C \int_0^t \norm{\psi_x}_{L^2(\mathbb{R})}^2 \, d s + \frac{c_*}{2} \int_0^t \norm{\phi_x}_{L^2(\mathbb{R})}^2 \, d s  +C \delta_S \int_0^t \left(G_1+ G^S \right) \, d s .
	\end{aligned}
	\end{equation*}
	Then, using the inequality \eqref{eq: omega 2-norm inequality}, we obtain the following result:
	\begin{equation} \label{eq: est-4}
	\begin{aligned}
	&\norm{\psi_x}_{L^2(\R)}^2+ \norm{\omega}_{L^2(\R)}^2+\norm{\omega_x}_{L^2(\mathbb{R})}^2 + \int_0^t \left( \norm{\psi_{xx}}_{L^2 (\mathbb{R})}^2 + \norm{\omega}_{H^2(\mathbb{R})}^2  \right) \, d s \\
	&\le C (\norm{\phi_0}_{L^2(\R)}^2+\norm{\psi_0}_{H^1(\R)}^2+\norm{\omega_0}_{H^1(\R)}^2)\\
	&\quad + C_4 \left( \norm{\phi}_{L^2(\R)}^2+\norm{\psi}_{L^2(\mathbb{R})}^2 +  \delta_S \int_0^t \left( | \dot{X}|^2 +G_1+G^S \right) ds +  \int_0^t \norm{\psi_x}_{L^2(\mathbb{R})}^2  d s \right).\\
	\end{aligned}
	\end{equation}
	\noindent $\bullet$ (Step 3): Finally, by multiplying the inequality \eqref{eq: est-4} obtained above by $\frac{1}{2 \max (1,C_4)}$, adding it to the inequality in Lemma $\ref{Main Lemma}$ \eqref{energy-est}, and using the smallness of $\delta_S$, we get
	\begin{equation*}
	\begin{aligned}
	& \frac{1}{2} \left(\norm{\phi}_{L^2(\R)}^2+\norm{\psi}_{L^2(\R)}^2+\norm{\omega}_{L^2(\mathbb{R})}^2\right) + \frac{1}{2 \max (1,C_4)}(\norm{\psi_x}_{L^2(\R)}^2+\norm{\omega_x}_{L^2(\mathbb{R})}^2)\\
	&\quad + \frac{\delta_S}{2} \int_0^t | \dot{X}(s)|^2 \, d s   +\frac{1}{2} \int_0^t \left( G_1+G_3+G^S+\norm{\psi_x}_{L^2(\mathbb{R})}^2 \right) \, ds\\ 
	&\quad + \frac{1}{2 \max(1,C_4)} \int_0^t \left(\norm{\psi_{xx}}_{L^2(\mathbb{R})}^2 + \norm{\omega}_{H^2(\mathbb{R})}^2  \right)\, ds \\ 
	& \le C( \norm{\phi_0}_{L^2(\R)}^2+\norm{\psi_0}_{H^1(\R)}^2+\norm{\omega_0}_{H^1(\R)}^2 ),
	\end{aligned}
	\end{equation*}
	which proves the desired results.

	\begin{appendix}
		\section{Proof of \eqref{Diffusion}} \label{Diffusion proof}
		\setcounter{equation}{0}
		Here, we provide the detailed proof of \eqref{Diffusion}. Recall that $2$-viscous shock wave satisfies the following equations:
		\begin{equation}\label{viscous shock}
		\begin{aligned}
		\begin{cases}
		&-\sigma (\tv)'-(\tu)'=0,\\
		&-\sigma (\tu)'+(p(\tv))'=\Big( \dfrac{(\tu)'}{\tv} \Big)'+\Big( \dfrac{(\tw)'}{\tv^{5/2}}\Big)',\\
		&-\sigma (\tw)'=\Big(-\dfrac{(\tu)'}{\tv^{5/2}}\Big)',\\
		&(\tv,\tu,\tw)(-\infty)=(v_-,u_-,w_-), \quad (\tv,\tu)(+\infty)=(v_+,u_+,w_+).
		\end{cases}
		\end{aligned}
		\end{equation}
		We integrate \eqref{viscous shock} over $(-\infty, x ]$ to obtain
		\[ \frac{(\tu)'}{\tv}=-\sigma(\tu-u_-)+\big( p(\tv)-p(v_-) \big)-\frac{(\tw)'}{\tv^{5/2}}.\]
		On the one hand, since
		\begin{equation}\label{dy/dx} 
		\delta_S \frac{1}{\tv} \Big( \frac{dy}{d x}\Big)=-\frac{(\tu)'}{\tv},
		\end{equation}
		we have
		\[\delta_S \frac{1}{\tv} \Big( \frac{dy}{d \xi}\Big)-\frac{(\tw)'}{\tv^{5/2}}=\sigma(\tu-u_-)-\big( p(\tv)-p(v_-)\big).\]
		On the other hand, from \eqref{viscous shock}, \eqref{dy/dx} and the smallness of $\delta_S$, we have
		\begin{equation}\label{ ineq : tw'/tv^{5/2}}
		\-C \delta_S \frac{1}{\tv} \frac{dy}{dx} \le -\frac{(\tw)'}{\tv^{5/2}} \le C \delta_S \frac{1}{\tv} \frac{dy}{dx}.
		\end{equation}
		Using \eqref{dy/dx} and \eqref{ ineq : tw'/tv^{5/2}}, we obtain 
		\[ \delta_S(1-C \delta_S) \frac{1}{\tv} \frac{dy}{dx} \le \sigma (\tu-u_-)-(p(\tv)-p(v_-)) \le \delta_S(1+C \delta_S) \frac{1}{\tv} \frac{dy}{dx}, \]
		which together with $y=\frac{u_- -\tu}{\delta_S}$ and $1-y=\frac{\tu-u_+}{\delta_S}$ yield
		\begin{equation} \label{ineq of A}
		\frac{(1-C \delta_S)}{y(1-y)} \frac{1}{\tv} \frac{dy}{dx} \le \underbrace{\frac{\delta_S}{(u_--\tu)(\tu-u_+)} \Big( \sigma(\tu-u_-)-\big(p(\tv)-p(v_-)\big) \Big)}_{=:A} \le \frac{(1+C \delta_S)}{y(1-y)} \frac{1}{\tv} \frac{dy}{dx}.
		\end{equation}
		
		Now, we compute $A$ more precisely as
		\begin{align*}
		A&=\frac{\delta_S}{(u_--\tu)(\tu-u_+)} \Big( \sigma(\tu-u_-)-\big(p(\tv)-p(v_-)\big) \Big)\\
		&=\frac{\delta_S}{u_--u_+}\left(\frac{ \sigma(\tu-u_-)-\big(p(\tv)-p(v_-)\big)}{u_--\tu}+\frac{ \sigma(\tu-u_-)-\big(p(\tv)-p(v_-)\big)}{\tu-u_+} \right)\\
		&=\frac{ \sigma(\tu-u_-)-\big(p(\tv)-p(v_-)\big)}{u_--\tu}+\frac{ \sigma(\tu-u_-)-\big(p(\tv)-p(v_-)\big)}{\tu-u_+} \\
		&=\frac{ -\sigma^2(\tv-v_-)-\big(p(\tv)-p(v_-)\big)}{\sigma(\tv-v_-)}+\frac{- \sigma^2(\tv-v_-)-\big(p(\tv)-p(v_-)\big)}{-\sigma(\tv-v_+)}.
		\end{align*}
		In addition, since $\sigma^2=\frac{p(v_-)-p(v_+)}{v_+-v_-}$, we have
		\begin{align*}
		A
		& =-\frac{1}{\sigma}\left(\frac{p(v_-)-p(v_+)}{v_+-v_-}+\frac{p(\tv)-p(v_-)}{\tv-v_-} \right)+\frac{1}{\sigma}\left( \frac{p(v_-)-p(v_+)}{v_+-v_-} \frac{\tv-v_-}{\tv-v_+}+\frac{p(\tv)-p(v_-)}{\tv-v_+}\right)\\
		& =\frac{1}{\sigma}\left(\frac{p(v_-)-p(v_+)}{v_+-v_-} \frac{v_+-v_-}{\tv-v_+}+\frac{p(\tv)-p(v_-)}{\tv-v_+}-\frac{p(\tv)-p(v_-)}{\tv-v_-} \right)\\
		& =\frac{1}{\sigma}\left(\frac{p(\tv)-p(v_+)}{\tv-v_+}-\frac{p(\tv)-p(v_-)}{\tv-v_-} \right).
		\end{align*}
		Then, we estimate perturbation between $A$ and $\frac{\sigma}{2\sigma_\ell}\frac{ \delta_S v''(p_-)}{ |v'(p_-)|^2 }$ as
		\begin{equation} \label{A-est}
		\begin{aligned}
		\left|   A-\frac{\sigma}{2\sigma_\ell}\frac{ \delta_S v''(p_-)}{ |v'(p_-)|^2 } \right|
		\le \underbrace{\Bigg| A-\frac{1}{2}\frac{\delta_S v''(p_+)}{ |v'(p_+)|^2 } \Bigg|}_{=:A_1}  +\underbrace{\Bigg| \frac{1}{2}\frac{ \delta_S v''(p_+)}{ |v'(p_+)|^2  }  -\frac{\sigma}{2\sigma_\ell}\frac{ \delta_S v''(p_-)}{ |v'(p_-)|^2 }\Bigg| }_{=:A_2 }.
		\end{aligned}
		\end{equation}
		Indeed, we apply Lemma \ref{B.1.} below, when $p_-=p(v_+),p_+=p(v_-)$ and $p=p(\tv)$, we have
		\begin{equation*}
		\begin{aligned}
		A_1
		&=\frac{1}{|\sigma|}\left|   \frac{p(\tv)-p(v_+)}{\tv-v_+}+\frac{p(\tv)-p(v_-)}{v_--\tv}+\frac{1}{2}\frac{ v''(p_+)}{ |v'(p_+)|^2 }(p(v_+)-p(v_-)) \right| \leq C \delta_S^2.
		\end{aligned}
		\end{equation*}
		On the other hand, we use \eqref{shock_speed_est} and \eqref{shock-property} to observe that $A_2 \le C \delta_S^2$. 
		
		Therefore, we have the desired estimate by using \eqref{ineq of A} and \eqref{A-est} with $|A| \le C \delta_S$ as
		\begin{equation*} 
		\left|   \frac{1}{y(1-y)} \frac{1}{\tv} \left( \frac{dy}{dx} \right)-\frac{\sigma}{2\sigma_\ell}\frac{ \delta_S v''(p_-)}{ |v'(p_-)|^2 } \right| \le C \delta_S^2.
		\end{equation*}
		

		Finally, we close this section by giving the proof of the following lemma.
		\begin{lemma}\label{B.1.} For any $r>0$,  there exists $\varepsilon_0>0$ and $C>0$ such that the following holds.  For any $v_-,v_+,v>0$ such that $v_- \in (r,2r),  v_+-v_-=: \varepsilon \in (0,\varepsilon_0)$,  $v_- \leq v \leq v_+$,  and $p,p_-,p_+$ such that $p(v)=p,p(v_\pm)=p_\pm$,  we have
			\begin{equation*}
			\left|\frac{p-p_-}{v-v_-}+\frac{p-p_+}{v_+-v}+\frac{1}{2}\frac{v''(p_-)}{v'(p_-)^2}(p_--p_+) \right| \leq C \varepsilon^2.
			\end{equation*}
		\end{lemma}
		\begin{proof}
			Consider the function $p(v)=v^{-\gamma}$.  Then,  using a Taylor expansion at $v_-$ and $v_+$,  we find that there exists $\varepsilon_0$ such that for any $|p-p_-| \leq \varepsilon_0$ and $|p-p_+| \leq \varepsilon_0$ we have
			\begin{eqnarray} 
			\label{1}
			\left| p-p_- - \frac{dp}{dv}(v_-)(v-v_-)-\frac{1}{2} \frac{d^2p}{dv^2}(v_-)(v-v_-)^2 \right| \leq C|v-v_-|^3,\\
			\label{2}
			\left| p-p_+ - \frac{dp}{dv}(v_+)(v-v_+)-\frac{1}{2} \frac{d^2p}{dv^2}(v_+)(v-v_+)^2 \right| \leq C|v-v_+|^3.
			\end{eqnarray}
			Since
			\begin{equation*}
			\frac{d^2 p}{dv^2}=\frac{d}{dv}\left(\frac{1}{v'(p)}\right)=-\frac{v''(p)}{v'(p)^2}\frac{dp}{dv},
			\end{equation*}
			we get
			\begin{align}
			\begin{aligned}\label{3}
			&\left| \frac{1}{2} \frac{v''(p_-)}{v'(p_-)^2}(p_- - p_+)+\frac{1}{2} \frac{d^2p}{dv^2}(v_-)(v_- - v_+) \right| \\
			&\hspace{1cm} \leq \frac{v''(p_-)}{2v'(p_-)^2} \left| p_+ - p_ - - \frac{dp}{dv}(v_-)(v_+ -v_-)\right| \leq C \varepsilon^2.
			\end{aligned}
			\end{align}
			Moreover, we note that
			\begin{equation}
			\label{4}
			\begin{split}
			\left| \frac{1}{2} \frac{d^2p}{dv^2}(v_+)(v-v_+)-\frac{1}{2}\frac{d^2p}{dv^2}(v_-)(v-v_-)+\frac{1}{2}\frac{d^2p}{dv^2}(v_-)(v_+-v_-) \right|\\
			=\frac{1}{2} \left|\left( \frac{d^2p}{dv^2}(v_+)-\frac{d^2p}{dv^2}(v_-)\right)(v-v_+)\right| \leq C \varepsilon^2.
			\end{split}
			\end{equation}
			Now, dividing \eqref{1} by $v-v_-$,  \eqref{2} by $v_+-v$, and adding both terms, we obtain
			\[\left|\frac{p-p_-}{v-v_-}+\frac{p-p_+}{v_+-v}-\frac{dp}{dv}(v_-)+\frac{dp}{dv}(v_+)-\frac{1}{2}\frac{d^2p}{dv^2}(v_-)(v-v_-)+\frac{1}{2}\frac{d^2p}{dv^2}(v_+)(v-v_+)\right|\le C\e^2,\]
			which, together with the estimates in \eqref{3} and \eqref{4}, yields
			\begin{equation*}
			\begin{split}
			&\Bigg|\frac{p-p_-}{v-v_-}+\frac{p-p_+}{v_+-v}+\frac{1}{2}\frac{v''(p_-)}{v'(p_-)^2}(p_--p_+) \\
			&\qquad -\left( \frac{dp}{dv}(v_-)-\frac{dp}{dv}(v_+)-\frac{d^2p}{dv^2}(v_-)(v_--v_+)\right) \Bigg| \leq C\varepsilon^2.
			\end{split}
			\end{equation*}
			This gives the result since the second line terms is itself of order $\varepsilon^2$.
		\end{proof}

		\section{Proof of Lemma \ref{lem:shock-property}}\label{sec:app-shock-proof}
		\setcounter{equation}{0}
		
		Since the equivalence of $\tv$ and $\tu$ directly follows from $\tu' = -\sigma \tv'$, we focus on showing the properties of $\tv$. We will here use a parameter $\delta$ to denote the shock strength $\delta= v_+- v_-\ll1$ whereas we used, in the a priori estimates, the shock strength $\delta_S= u_- -u_+$ that is equivalent to $\delta$.  
		
		We first note that the system of ODEs \eqref{viscous-dispersive-shock} can be written as a second-order ODE with respect to $\tv$ as
		\begin{equation} \label{eq : tv''}
		\tv''=f(\tv)-\sigma \tv^4 \tv' + \frac{5 (\tv')^2}{2 \tv},
		\end{equation}
		where
		\begin{equation} \label{eq : f(tv)}
		f(\widetilde{v})=-\widetilde{v}^5 \left(\sigma^2(\widetilde{v}-v_-)+p(\widetilde{v})-p(v_-) \right).
		\end{equation}
		For the notational simplicity, we drop the tilde, i.e., $v=\tv$ in the rest of the proof. 
		
		We will apply Fenichel's theory on invariant manifold (for example, see \cite{Jones} or references therein).
		For that, we first rescale $v$ as $\overline{v}$, which is defined by
		\beq\label{scalev}
		v(x) = \delta\overline{v}(\delta x)+v_-+\frac{\delta}{2},
		\eeq
		and introduce a slow variable $z=\delta x$, where $\delta$ denotes the shock strength $\delta= v_+- v_-\ll1$.
		Note that
		\[v''(x) = \delta^3 \overline{v}_{zz}(z),\quad v'(x) = \delta^2\overline{v}_z(z).\]
		This and \eqref{eq : tv''} imply that $\overline{v}=\overline{v}(z)$ satisfies the following ODE:
		\begin{align*}
		\delta\overline{v}_{zz}&=-\left(\delta \overline{v}+v_-+\frac{\delta }{2}\right)^5\frac{1}{\delta^2}\left(\sigma^2\left(\delta\overline{v}+\frac{\delta}{2}\right)+p\left(\delta \overline{v}+v_-+\frac{\delta }{2}\right)-p(v_-)\right)\\
		&\quad -\sigma\left(\delta\overline{v}+v_-+\frac{\delta}{2}\right)^4 \overline{v}_z+\frac{5}{2}\frac{\delta^2}{\delta\overline{v}+v_--\frac{\delta}{2}}(\overline{v}_z)^2.
		\end{align*}
		We now rewrite the above equation as the system of first-order ODEs with respect to $(\overline{v},w:=\overline{v}_z)$:
		\begin{align}
		\begin{aligned}\label{system-slow}
		\overline{v}_z &= w,\\
		\delta w_z &=-\left(\delta \overline{v}+v_-+\frac{\delta }{2}\right)^5\frac{1}{\delta^2} \underbrace{\left(\sigma^2\left(\delta\overline{v}+\frac{\delta}{2}\right)+p\left(\delta \overline{v}+v_-+\frac{\delta }{2}\right)-p(v_-)\right)}_{=:J(\overline{v};\delta)}\\
		&\quad -\sigma\left(\delta\overline{v}+v_-+\frac{\delta}{2}\right)^4 w+\frac{5}{2}\frac{\delta^2}{\delta\overline{v}+v_-+\frac{\delta}{2}}w^2.\\
		\end{aligned}
		\end{align}
		First, using $\sigma^2=-\frac{p(v_+)-p(v_-)}{v_+ -v_-},  \delta=v_+-v_-$, and $\delta \overline{v}+v_-+\frac{\delta }{2}>0$ by \eqref{scalev}, the above system has the two critical points $(\overline{v}, w)=(\pm\frac{1}{2},0)$ only.
		Indeed, $\overline{v}= \pm\frac{1}{2}$ are the unique solutions of
		\[
		J(\overline{v};\delta) = 0.
		\]
		Next, in order to find the critical manifold associated with \eqref{system-slow} at $\delta=0$, we use the Taylor expansion of $J(\overline{v};\delta)$ w.r.t. $\delta$ as follows:
		\begin{align*}
		\begin{aligned}
		J(\overline{v};\delta) = \sigma^2\left(\delta\overline{v}+\frac{\delta}{2}\right)+ p'(v_-)\left(\delta\overline{v}+\frac{\delta}{2}\right) +\frac{p''(v_-)}{2} \left(\delta\overline{v}+\frac{\delta}{2}\right)^2 + O(\delta^3),
		\end{aligned}
		\end{align*}
		which together with $\sigma^2=-\frac{p(v_+)-p(v_-)}{v_+ -v_-}= -\frac{p(v_+)-p(v_-)}{\delta}$ implies
		\[
		J(\overline{v};\delta) = \delta^2 \frac{p''(v_-)}{2} \left(\overline{v}^2-\frac{1}{4}\right) + O(\delta^3).
		\]
		Thus, by letting $\delta=0$ on the second equation of \eqref{system-slow}, we deduce that the critical manifold $M_0$ is the graph of the following equation:
		\[w=\frac{v_-p''(v_-)}{2\sigma}\left(\frac{1}{4}-\overline{v}^2\right).\]
		Thus, we use the Fenichel's first theorem (see \cite[Theorem 2]{Jones}) to derive that
		\[M_\delta:=\left\{(\overline{v},w) \in \bbr^2 ~  \bigg|~w=\frac{v_-p''(v_-)}{2\sigma}\left(\frac{1}{4}-\overline{v}^2\right)+\delta h(\overline{v},\delta)\right\}\]
		is a locally invariant manifold under the flow of the system \eqref{system-slow}, where the function $h$ is smooth jointly in $\overline{v}$ and $\delta$. In other words, there exists a neighborhood $\mathcal{N}\subset \bbr^2$ such that if a solution $(\overline{v}(z),w(z))$ of \eqref{system-slow} starting from a point of $M_\delta$ stays on $\mathcal{N}$, then it stays on $M_\delta$. 
		
		We will now show that there exists a solution $(\overline{v},w):(-\infty,\infty)\to \R^2$ to \eqref{system-slow} whose range (globally) lies on $M_\delta$, and show that it is indeed the viscous-dispersive shock profile that we are looking for, that is, $(\overline{v}(z),w(z))\to (0,\pm\frac12)$ as $z\to\pm \frac{1}{2}$.
		
		
		
		To this end, we observe that the smooth function $F(y)= \frac{v_-p''(v_-)}{2\sigma}\left(\frac{1}{4}-y^2\right)+\delta h(y,\delta)$ satisfies $F(0)>0$ by $\delta\ll1$, since  $h$ is smooth. Thus, there exists a unique non-decreasing solution $\overline{v}_0:\bbr\to\bbr$ of
		\begin{equation}\label{ODE2}
		(\overline{v}_0)_z = \frac{v_-p''(v_-)}{2\sigma}\left(\frac{1}{4}-\overline{v}_0^2\right)+\delta h(\overline{v}_0,\delta),
		\end{equation}
		and $\overline{v}_0(0)=0$, satisfying
		\beq\label{standingend}
		\lim_{z\to\pm \infty} \overline{v}_0(z) = \overline{v}_\pm,\qquad \mbox{where $F(\overline{v}_\pm)=0$ and $F(y)>0$ for all $y\in (\overline{v}_-,\overline{v}_+)$}.
		\eeq
		Then, by letting $w_0(z) := (\overline{v}_0)_z(z)$, it holds from \eqref{ODE2} that the curve $(\overline{v}_0(z),w_0(z))$ lies on the manifold $M_\delta$. \\
		We will show that $(\overline{v}_0(z),w_0(z))$ is the solution to the system \eqref{system-slow}. For a fixed $z_0\in\bbr$, the Cauchy-Lipschitz theorem implies that 
		the system \eqref{system-slow} has a unique local Lipschitz continuous solution  $(\overline{v}(z),w(z))$ starting from $(\overline{v}(z_0),w(z_0))= (\overline{v}_0(z_0),w_0(z_0)) \in M_\delta$. In addition, since the manifold $M_\delta$ is locally invariant, the range of the unique local solution $(\overline{v}(z),w(z))$ belongs to $M_\delta$. 
		That is, the local solution $(\overline{v}(z),w(z))$ satisfies 
		\begin{equation}\label{ODE3}
		\overline{v}_z= w= \frac{v_-p''(v_-)}{2\sigma}\left(\frac{1}{4}-\overline{v}^2\right)+\delta h(\overline{v},\delta).
		\end{equation}
		This together with the uniqueness of \eqref{ODE2} implies that $\overline{v}=\overline{v}_0$ and $w=w_0$ near the point $z_0$. Thus, $(\overline{v}_0,w_0)$ is a solution to \eqref{system-slow} near $z_0$. Since this holds for any $z_0\in \R$, we conclude that $(\overline{v}_0,w_0)$ globally satisfies \eqref{system-slow}. Finally, by \eqref{standingend}, the two end points $\overline{v}_{\pm}$ should be the $v$-component of the critical point of the system \eqref{system-slow}, that is, $\overline{v}_{\pm} = \pm\frac{1}{2}$. Therefore, we conclude that the solution $(\overline{v}_0(z),w_0(z))$ connects the end points $(\pm\frac{1}{2},0)$. By the uniqueness of the heteroclinic orbit from $(-\frac{1}{2},0)$ to $(\frac{1}{2},0)$, any solution $(\overline{v},w)$ to \eqref{system-slow} coincides with $(\overline{v}_0,w_0)$. This implies that the smooth monotone solution to the second-order ODE \eqref{eq : tv''} coincides with $\delta\overline{v}_0(\delta x)+v_-+\frac{\delta}{2}$, whose profile monotonically connects $v_-$ to $v_+$.
		
		The remaining part is to show that the desired bounds on the derivatives of $v$ hold. However, since $\overline{v}$ satisfies the ODE \eqref{ODE3}, and $|\overline{v}|\le \frac12$, we have $|\overline{v}_z|\le C$ for a constant $C$ independent of $\delta$. Therefore, we get
		\[|v'(x)| = \delta^2 |\overline{v}_z(z)| \le C\delta^2\]
		and
		\[|v''(x)|=\delta^3|\overline{v}_{zz}(z)|\le C\delta^3|\overline{v}||\overline{v}_z|+C\delta^4|\pa_{\overline{v}}h(\overline{v},\delta)||\overline{v}_z|\le C\delta |v'(x)|.\]
		
		Now, from \eqref{eq : tv''} and \eqref{eq : f(tv)}, we obtain
		\[C\tv'\ge \frac{\sigma}{\tv}  \tv'+\frac{\tv''}{\tv^5}-\frac{5(\tv')^2}{2\tv^6} = -(\sigma^2(\tv-v_\pm)+p(\tv)-p(\tv_{\pm}))\ge C(\tv-v_-)(v_+-\tv), \]
		which yields
		\[|\tv(x)-v_-|\le |\tv(0)-v_-|e^{-C\delta|x|},\quad x<0.\]
		The other inequalities can be obtained in a similar manner.

	\end{appendix}
	
	\bibliographystyle{amsplain}
	\bibliography{reference} 
	
\end{document}